\documentclass[11pt,reqno]{amsart}
\usepackage{amsmath,amsthm,amssymb,mathrsfs,stmaryrd,color,iftex}
\usepackage[all]{xy}
\usepackage{url}
\usepackage{geometry}
\usepackage{tikz-cd}
\usepackage[utf8]{inputenc}
\usepackage[T1]{fontenc}

\usepackage{enumitem}
\usepackage[colorlinks=true,hyperindex, linkcolor=purple, pagebackref=false, citecolor=magenta,pdfpagelabels]{hyperref}
\usepackage{hyperref}

\newcommand{\cosimp}[3]{\xymatrix@1{#1 \ar@<.4ex>[r] \ar@<-.4ex>[r] & {\ }#2 \ar@<0.8ex>[r] \ar[r] \ar@<-.8ex>[r] & {\ } #3 \ar@<1.2ex>[r] \ar@<.4ex>[r] \ar@<-.4ex>[r] \ar@<-1.2ex>[r] & \cdots }}

\newcommand{\colim}{\mathop{\mathrm{colim}}}

\newcommand{\adjunction}[4]{\xymatrix@1{#1{\ } \ar@<0.3ex>[r]^{ {\scriptstyle #2}} & {\ } #3 \ar@<0.3ex>[l]^{ {\scriptstyle #4}}}}

\newtheorem{theorem}{Theorem}[subsection]
\newtheorem*{theorem*}{Theorem}
\newtheorem*{definition*}{Definition}
\newtheorem{proposition}[theorem]{Proposition}
\newtheorem{lemma}[theorem]{Lemma}
\newtheorem{corollary}[theorem]{Corollary}

\theoremstyle{definition}
\newtheorem{definition}[theorem]{Definition}

\newtheorem{remark}[theorem]{Remark}

\newtheorem{example}[theorem]{Example}

\newtheorem{claim}[theorem]{Claim}

\newtheorem{definition/proposition}[theorem]{Definition/Proposition} 

\newcommand{\Perf}{\mathrm{Perf}}

\newcommand{\Hom}{\mathrm{Hom}}
\newcommand{\sHom}{R\mathscr{H}\mathrm{om}}

\newcommand{\Spa}{\mathrm{Spa}}
\newcommand{\Spec}{\mathrm{Spec}}

\newcommand{\IC}{\mathrm{IC}}
\newcommand{\Bun}{\mathrm{Bun}}
\newcommand{\Groth}{\mathrm{Groth}}

\newcommand{\wtBun}{\widetilde{\mathrm{Bun}}}
\newcommand{\Eis}{\mathrm{Eis}}
\newcommand{\CT}{\mathrm{CT}}

\newcommand{\calE}{\mathcal{E}}
\newcommand{\calF}{\mathcal{F}}

\newcommand{\calH}{\mathcal{H}}
\newcommand{\bD}{\mathbb{D}}

\def\basic{\mathrm{basic}}
\def\Aut{\mathrm{Aut}}
\def\Lie{\mathrm{Lie}}
\def\cInd{\mathrm{cInd}}

\newcommand{\ra}{\rightarrow}
\newcommand{\la}{\leftarrow}

\newcommand{\Mcoinv}{\mathbb{X}_{*}(M^{\mathrm{ab}}_{\ol{E}})_{\Gamma}}
\newcommand{\Mcoinvdom}{\mathbb{X}_{*}(M^{\mathrm{ab}}_{\ol{E}})_{\Gamma}^{+}}

\def\Spd{\mathop{\rm Spd}}
\def\ul{\underline}
\def\ol{\overline}
\def\wt{\widetilde}
\def\Div{\mathrm{Div}}

\def\nmEisshriekP{\mathrm{Eis}_{P!}}
\def\nmEisstarP{\mathrm{Eis}_{P*}}

\def\IC{\mathrm{IC}}

\def\Gr{\mathrm{Gr}}
\def\basic{\mathrm{basic}}
\def\Coh{\mathrm{Coh}}
\def\mf{\mathfrak}
\def\mc{\mathcal}
\def\Mod{\mathrm{Mod}}

\def\disj{\mathrm{disj}}

\def\GL{\mathrm{GL}}
\def\ren{\mathrm{ren}}

\def\Gr{\mathrm{Gr}}

\makeatletter
\newcommand*{\da@rightarrow}{\mathchar"0\hexnumber@\symAMSa 4B }
\newcommand*{\da@leftarrow}{\mathchar"0\hexnumber@\symAMSa 4C }
\newcommand*{\xdashrightarrow}[2][]{%
  \mathrel{%
    \mathpalette{\da@xarrow{#1}{#2}{}\da@rightarrow{\,}{}}{}%
  }%
}
\newcommand{\xdashleftarrow}[2][]{%
  \mathrel{%
    \mathpalette{\da@xarrow{#1}{#2}\da@leftarrow{}{}{\,}}{}%
  }%
}
\newcommand*{\da@xarrow}[7]{%
  \sbox0{$\ifx#7\scriptstyle\scriptscriptstyle\else\scriptstyle\fi#5#1#6\m@th$}%
  \sbox2{$\ifx#7\scriptstyle\scriptscriptstyle\else\scriptstyle\fi#5#2#6\m@th$}%
  \sbox4{$#7\dabar@\m@th$}%
  \dimen@=\wd0 %
  \ifdim\wd2 >\dimen@
    \dimen@=\wd2 %
  \fi
  \count@=2 %
  \def\da@bars{\dabar@\dabar@}%
  \@whiledim\count@\wd4<\dimen@\do{%
    \advance\count@\@ne
    \expandafter\def\expandafter\da@bars\expandafter{%
      \da@bars
      \dabar@ 
    }%
  }%
  \mathrel{#3}%
  \mathrel{%
    \mathop{\da@bars}\limits
    \ifx\\#1\\%
    \else
      _{\copy0}%
    \fi
    \ifx\\#2\\%
    \else
      ^{\copy2}%
    \fi
  }%
  \mathrel{#4}%
}
\makeatother

\providecommand{\bysame}{\leavevmode\hbox to3em{\hrulefill}\thinspace}
\providecommand{\MR}{\relax\ifhmode\unskip\space\fi MR }

\providecommand{\href}[2]{#2}

\begin{document}

\title{Geometric Eisenstein series I: Finiteness theorems}
\author{Linus Hamann, David Hansen and Peter Scholze}

\begin{abstract} We develop the theory of geometric Eisenstein series and constant term functors for $\ell$-adic sheaves on stacks of bundles on the Fargues-Fontaine curve. In particular, we prove essentially optimal finiteness theorems for these functors, analogous to the usual finiteness properties of parabolic inductions and Jacquet modules. We also prove a geometric form of Bernstein's second adjointness theorem, generalizing the classical result and its recent extension to more general coefficient rings proved in \cite{DHKM}. As applications, we decompose the category of sheaves on $\Bun_G$ into cuspidal and Eisenstein parts, and show that the gluing functors between strata of $\Bun_G$ are continuous in a very strong sense.
\end{abstract}

\maketitle

\tableofcontents

\section{Introduction}
\subsection{Prologue}
In \cite{FarguesOverview}, Fargues laid out a remarkable and far-reaching vision, recasting many structures around the representation theory of $p$-adic groups and the local Langlands correspondence as \emph{geometric} Langlands over the Fargues-Fontaine curve. The work of Fargues-Scholze \cite{FarguesScholze} then put this vision on solid technical foundations and applied it to construct a general semisimple local Langlands correspondence.

In this paper and its sequel \cite{HHS2}, we develop some additional foundations for this program. More precisely, our goal in these papers is to geometrize the classical operations of parabolic induction and Jacquet module into functors between sheaf categories on suitable moduli stacks. In the present paper we define these functors and prove their basic finiteness properties.

\subsection{Main results} To state our main results, let $E$ be a non-archimedean local field, and let $H/E$ be a linear algebraic group. One then has the associated Artin v-stack $\Bun_H$ parametrizing $H$-bundles on the Fargues-Fontaine curve. This is functorial in $H$. In particular, if $G$ is a connected reductive group and $P=MU \subset G$ is a parabolic subgroup, we have a natural diagram \[\Bun_M \xleftarrow{\mathfrak{q}} \Bun_P \xrightarrow{\mathfrak{p}} \Bun_G.\] We note that $\mathfrak{q}$ is cohomologically smooth, but not representable, while $\mathfrak{p}$ is representable in locally spatial diamonds but is neither smooth nor proper.

Now, let $\Lambda$ be a $\mathbb{Z}_\ell[\sqrt{q}]$-algebra, where $q$ denotes the cardinality of the residue field of $E$ and $\ell \nmid q$. One can define a canonical invertible object $\IC_{\Bun_P} \in D(\Bun_P,\Lambda)$, which is a square root of the dualizing complex and can be described explicitly in terms of the modulus character \cite[Theorem~1.5]{HI}. Our main interest is in the functors\footnote{At present, the six-functor formalism required to make sense of these formulas is only developed in the literature under the additional assumption that $\Lambda$ is killed by a power of $\ell$. However, forthcoming work \cite{MotivicGeometrization} of the third author will correct this issue. For the moment, the reader may wish to assume that $\Lambda$ is a torsion ring. However, we have written this paper in such a way that once the formalism developed in \cite{MotivicGeometrization} is available, everything will apply to general $\mathbb{Z}_\ell[\sqrt{q}]$-algebras.}
\[\begin{aligned}\Eis_{P!}(-)&:= \mathfrak{p}_!(\mathfrak{q}^\ast(-) \otimes \IC_{\Bun_P} ), \\
\Eis_{P\ast}(-)&:= \mathfrak{p}_{\ast}(\mathfrak{q}^\ast(-) \otimes \IC_{\Bun_P} ), \end{aligned}\]
and
\[\begin{aligned}\CT_{P*}(-)&:= \mathfrak{q}_\ast(\mathfrak{p}^!(-) \otimes \IC_{\Bun_P}^{-1} ), \\
\CT_{P!}(-)&:= \mf{q}_{!}(\mathfrak{p}^{*}(-) \otimes \IC_{\Bun_{P}}). \end{aligned}\]
The first two functors here are two possible geometrizations of normalized parabolic induction: Over the locus of trivial $G$-bundles, the map $\mathfrak p$ is proper and so both functors agree there. The functor $\CT_{P!}$ is the natural geometrization of the normalized Jacquet module and is the left adjoint of $\Eis_{P*}$ via the identification $\mf{q}^{*}(-) \otimes \IC_{\Bun_{P}}^{\otimes 2} = \mf{q}^{!}(-)$. On the other hand, $\CT_{P*}$ is the right adjoint of $\Eis_{P!}$ by definition. By the ``second adjunction'' theorem of Bernstein, this happens to agree with the normalized Jacquet module for the opposite parabolic in classical representation theory; one of our main theorems will be an extension of second adjunction to our geometric setting.

In the classical world of representation theory, parabolic inductions and Jacquet modules have extremely good finiteness properties. With characteristic zero coefficients these are classical results of Bernstein, and with general coefficients these are very recent results of Dat-Helm-Kurinczuk-Moss \cite{DHKM} (incidentally, relying on \cite{FarguesScholze}). Informally, our main result says that the geometrizations of these functors still have very strong finiteness properties.

More precisely, recall that the sheaf category $D(\Bun_G,\Lambda)$ constructed in \cite{FarguesScholze} comes with two intrinsic finiteness conditions, namely the notions of compact sheaves and ULA sheaves. Moreover, there are two natural duality operations on $D(\Bun_G,\Lambda)$, Verdier duality, denoted $\mathbb{D}_{\mathrm{Verd}}$, and Bernstein-Zelevinsky duality, denoted $\mathbb{D}_{\mathrm{BZ}}$, which induce self-anti-equivalences on ULA sheaves and compact sheaves, respectively.  Our main results explain how these finiteness conditions and their dualities interact with the functors of Eisenstein series and constant term. 

\begin{theorem}[Theorem \ref{thm:Eiscompact}, Corollary \ref{cor:CTprelim}, Theorem \ref{thm:secondadj}, Theorem \ref{thm:hardfiniteness}] \label{thm:maintheorem}Let $G \supset P=MU$ be as above, with $P^-$ the opposite parabolic.
\begin{enumerate}

\item The functor $\Eis_{P!}$ preserves compact objects. If $A\in D(\Bun_M,\Lambda)$ is ULA and supported on finitely many connected components, then $\Eis_{P!}(A)$ and $\Eis_{P\ast}(A)$ are ULA.

\item The functor $\CT_{P*}$ preserves ULA objects. If $A\in D(\Bun_G,\Lambda)$ is compact, then $\CT_{P!}(A)|_{\Bun_{M}^\alpha}$ is compact for all $\alpha \in \pi_0(\Bun_M)$.

\item There is a canonical isomorphism of functors $\CT_{P!} \cong \CT_{P^- \ast}$.

\item We have the following duality isomorphisms:

i. $\mathbb{D}_{\mathrm{BZ}}\Eis_{P!}\cong \Eis_{P^- !}\mathbb{D}_{\mathrm{BZ}}^M$ on compact objects.

ii. $\mathbb{D}_{\mathrm{Verd}}\Eis_{P!}\cong \Eis_{P\ast}\mathbb{D}_{\mathrm{Verd}}^M$ on all objects.

iii. $\mathbb{D}_{\mathrm{Verd}}^M\CT_{P!}\cong \CT_{P\ast} \mathbb{D}_{\mathrm{Verd}}$ on all objects.

\end{enumerate}
\end{theorem}

We will also show that (3) is logically equivalent to the duality isomorphism (4).i. These statements can both be regarded as geometrizations of Bernstein's famous second adjointness theorem.

We now sketch some of the ideas in the proofs. The duality isomorphisms (4).ii-(4).iii are straightforward consequences of the definitions, and we state them for completeness. The first serious task is to prove that $\Eis_{P!}$ preserves compact objects. We prove this by a complicated but reasonably direct induction, building on the fact that after restricting to certain strata in $\Bun_M$, $\Eis_{P!}$ is closely related to classical parabolic inductions and to the gluing functors $i_{b\sharp}$ defined in \cite{FarguesScholze}, both of which are known to preserve compact objects. Note that the work of Dat-Helm-Kurinczuk-Moss \cite{DHKM} is a key input to this induction, and we do not reprove their results. Once we know that $\Eis_{P!}$ preserves compact objects, it is formal to see that $\CT_{P\ast}$ preserves ULA objects. With this preservation of compacts in hand, we then define a canonical map $\mathbb{D}_{\mathrm{BZ}}\Eis_{P!}\to \Eis_{P^- !}\mathbb{D}_{\mathrm{BZ}}^M$. We prove this map is an isomorphism by an inductive argument parallel to the proof that $\Eis_{P!}$ preserves compact objects, using classical second adjunction as well as a similar adjunction for the gluing functor $i_{b\sharp}$ as input.\footnote{While we were finalizing this paper, these results have been proved independently by Takaya \cite{Takaya}, along very similar lines.}

The remaining results in (1)-(2) lie deeper, and require some new geometric input. The key idea here is that $\Bun_{P}$ admits a relative compactification over $\Bun_G$ which is still well-behaved relative to $\Bun_M$. This is the famous \emph{Drinfeld compactification} $\wt{\Bun}_P$, which is a fundamental object in the classical geometric Langlands program. Roughly speaking, this stack parametrizes $G$-bundles equipped with a generic reduction to a $P$-bundle; see Definition \ref{def: tildecompactification} and \S \ref{sec: fixingthecenter} for the precise definition. However, we emphasize that while it is reasonably straightforward to copy the definition of $\wt{\Bun}_P$ into the setting of \cite{FarguesScholze}, its desired geometric properties are not obvious in this setting.

\begin{theorem}[Proposition \ref{prop: openimmersion}, Theorem \ref{thm: drinfeldisrelativecompactification},  Proposition \ref{prop: tildestratadesc}, Theorem \ref{thm: ULAtheorem}, Proposition \ref{prop: keypropertiesofDrinfeldInGeneralCase}]\label{thm:maingeometry} Let $G \supset P=MU$ be as above.

\begin{enumerate}

\item There is a naturally defined Artin v-stack $\wt{\Bun}_P$ sitting in a commutative diagram
\[\xymatrix{
\Bun_{P}\ar[drr]^{\mathfrak{p}_{P}}\ar@{_(->}[dr]_{\tilde{j}_P}\ar[dd]_{\mathfrak{q}_{P}}\\
 & \widetilde{\mathrm{Bun}}_{P}\ar[dl]^{\tilde{\mathfrak{q}}_{P}}\ar[r]_{\tilde{\mathfrak{p}}_{P}} & \mathrm{Bun}_{G}\\
\text{} \mathrm{Bun}_{M} & 
}\]
where $\tilde{j}_P$ is an open immersion. The map $\tilde{\mathfrak{p}}_{P}$ is partially proper and representable in locally spatial diamonds, and of locally finite $\mathrm{dim.trg}$.

\item If $\Bun_{M}^{\alpha} \subset \Bun_M$ is any connected component with open-closed preimage $\wt{\Bun}_{P}^{\alpha}$, the induced map $\tilde{\mathfrak{p}}_{P}^{\alpha}: \wt{\Bun}_{P}^{\alpha} \to \Bun_G$ is proper.

\item The stack $\wt{\Bun}_{P}$ admits a locally closed stratification by certain strata $\phantom{}_{\theta}\wt{\Bun}_{P}$, indexed by $\theta \in \Lambda_{G,P}^{\mathrm{pos}}$. Each stratum has a canonical fiber product decomposition
\[  \phantom{}_{\theta}\wt{\Bun}_{P} \cong \Bun_{P} \times_{\Bun_{M}} \Mod^{+,\theta}_{\Bun_{M}}, \]
where $\Mod^{+,\theta}_{\Bun_{M}}$ is a certain closed subspace of a relative Beilinson-Drinfeld Grassmannian living over $\Div^{(\theta)}$, a partially symmetrized version of the mirror curve. 

\item The sheaves $\tilde{j}_{P!}(\IC_{\Bun_P})$ and $\tilde{j}_{P\ast}(\IC_{\Bun_P})$ are $\tilde{\mathfrak{q}}_{P}$-ULA.
\end{enumerate}
\end{theorem}

The proofs of (1)-(3) require a good understanding of relatively flat coherent sheaves on relative Fargues-Fontaine curves. This theory was initiated in \cite{AL1}, further refined in \cite[Section~5.2.3]{HamGeomES}, and we now develop it further here. In particular, given a relative Fargues-Fontaine curve $X_S$ equipped with a relatively flat coherent sheaf $\mathcal{E}/X_S$, we make a careful study of the moduli space $\mathcal{S}\mathrm{ub}_{\mathcal{E}}^{n,d}$ parametrizing fiberwise-injective maps $\mathcal{F} \to \mathcal{E}$ from a (variable) relatively flat coherent sheaf $\mathcal{F}$ of fiberwise-constant generic rank $n$ and degree $d$. The first key result is that $\mathcal{S}\mathrm{ub}_{\mathcal{E}}^{n,d}$ is a locally spatial diamond which is partially proper and locally of finite dim.trg over $S$, and the subfunctor where the quotient sheaf $\mathcal{E}/\mathcal{F}$ is fiberwise a vector bundle is open. This formally implies (1), and related arguments reduce the properness in (2) to the following theorem.

\begin{theorem}\label{thm:properness} The natural map $\mathcal{S}\mathrm{ub}_{\mathcal{E}}^{n,d} \to S$ is proper.
\end{theorem}

When $\mathcal{E}$ is a vector bundle and $n=1$, this follows easily from results in \cite{FarguesScholze}. In the general case, one can easily reduce to the case that $\mathcal E$ is a vector bundle. Then the statement is equivalent to the assertion that any given set of sub-vector bundles $\mathcal F_i\subset \mathcal E\times_{X_S} X_{C_i}$ (defined over fields $\Spa C_i$) of constant rank and degree induce a sub-vector bundle $\mathcal F$ of $\mathcal E\times_{X_S} X_T$ for the ``product of points'' $T=\Spa (\prod_i \mathcal O_{C_i})[\tfrac 1\varpi]$. Producing this sub-vector bundle is a question that is local on the Fargues-Fontaine curve, but the assumption of constant rank and degree is a (necessary) global condition that we did not see how to make use of. Eventually, we realized the (in hindsight obvious) possibility of using the determinant map. This reduces one to proving properness of the map
\[
\bigwedge^n: \mathcal{S}\mathrm{ub}_{\mathcal{E}}^{n,d}\to \mathcal{S}\mathrm{ub}_{\bigwedge^n \mathcal{E}}^{1,d}
\]
where now one can in fact forget the degree restriction. After this reduction, it is possible to localize the statement from the Fargues-Fontaine curve $X_S$ to its cover $Y_S$ or some affinoid piece thereof, and argue by explicit commutative algebra.

With the geometry under control, it remains to prove Theorem \ref{thm:maingeometry}.(4). The analogue of this result in classical geometric Langlands is well-known \cite[Theorem~5.1.5]{BG}, but the standard proof makes use of Deligne's generic ULA theorem to ensure the existence of some open subspace in $\Bun_{M}$ over which the ULA result holds, which can then be spread out to all of $\Bun_{M}$ by Hecke correspondences. Since the generic ULA theorem is definitely false in $p$-adic geometry, this approach is not available to us. Another approach for the Borel is provided in \cite[Section~5.2]{BG}; however, this argument fails in our situation at the very first step, since being ULA over the point is not a trivial condition in the setting of adic geometry and the sheaf categories we are considering. To solve our problems, we instead make a careful induction on the boundary strata of $\wt{\Bun}_P$, using hyperbolic localization to propagate the ULA property deeper and deeper into the boundary. The key idea here is embodied in the following result, which works in essentially any six-functor formalism.

\begin{proposition}
Let $\pi:X \to S$ be a reasonable map equipped with a section $i:S\to X$ and a $\mathbb{G}_m$-action which is totally attracting towards the section. Let $X^\circ = X\smallsetminus S$ with $j:X^\circ \to X$ the natural open immersion, and set $\pi^\circ = \pi \circ j$. If $A\in D(X^\circ,\Lambda)$ is $\mathbb{G}_m$-equivariant and $\pi^\circ$-ULA, then $j_!A$ is $\pi$-ULA.
\end{proposition}
Note that as a formal consequence, we get that also $j_{\ast}A$ is $\pi$-ULA, and that $i^\ast j_\ast A \in D(S,\Lambda)$ is dualizable. We also see that if $B \in D(X,\Lambda)$ is any $\mathbb{G}_m$-equivariant sheaf, then it is $\pi$-ULA iff $j^\ast B$ is $\pi^\circ$-ULA and $i^\ast B$ is dualizable.

To employ this technique, we need good local models for the boundary strata of $\wt{\Bun}_P$. These are provided by the theory of \emph{Zastava spaces}, which smooth-locally model suitable portions of $\wtBun_P$ in terms of certain intersections of semi-infinite orbits inside the $B_{\mathrm{dR}}^{+}$-affine Grassmannians. For all $\theta \in \Lambda_{G,P}^{\mathrm{pos}}$, there exists a Zastava space $\tilde{Z}^{\theta} \ra \Mod^{+,\theta}_{M}$, equipped with a section $\mf{s}^{\theta}: \Mod^{+,\theta}_{M} \ra \tilde{Z}^{\theta}$ and (\'etale locally on $\Mod^{+,\theta}_{M}$) a $\mathbb{G}_{m}$-action which attracts to the section. The image of this section is smooth-locally isomorphic to the strata $\phantom{}_{\theta}\wt{\Bun}_{P}$, which puts us in a situation where we can apply the previous Theorem (or rather a slight generalization of it, see Theorem \ref{thm: hyplocalpropvariant} for the exact statement we need) to inductively propagate the ULA property of the sheaf $\tilde{j}_{P!}(\IC_{\Bun_{P}})$ into deeper and deeper strata of the Drinfeld compactification.

With Theorem \ref{thm:maingeometry} in hand, the remaining statements in Theorem \ref{thm:maintheorem}.(1)-(2) follow from some straightforward manipulations with the definitions of the functors, using the miraculous properties of ULA sheaves quite heavily; see Theorem \ref{thm:hardfiniteness} for details.

We also note that our argument for Theorem \ref{thm:maingeometry}.(4) goes through without change in the setting of classical geometric Langlands, and gives a new proof which works uniformly in all the sheaf theories usually considered in that setting.

\subsection{Some applications}

Our results have immediate applications to the gluing functors between strata in $\Bun_G$. More precisely, recall that, for any $b\in B(G)$, we have functors \[i_{b\sharp}, i_{b!}, i_{b\ast}:D(G_b(E),\Lambda) \to D(\Bun_G,\Lambda)\] and also functors \[i_{b}^{\ast}, i_{b}^{!}:D(\Bun_G,\Lambda)\to D(G_b(E),\Lambda).\] These satisfy adjunctions
\[ i_{b\sharp} \dashv i_{b}^{\ast} \dashv i_{b\ast}\;\;\mathrm{and}\;\;i_{b!} \dashv i_{b}^{!}.  \]
 Moreover, the analysis in \cite{FarguesScholze} shows that $i_{b \sharp}, i_{b!},$ and $i_{b}^{\ast}$ preserve compact objects, while $i_{b!}$, $i_{b\ast}$, $i_{b}^{\ast}$ and $i_{b}^!$ preserve ULA objects. The following result settles the remaining finiteness questions for these functors.

\begin{theorem}The functor $i_{b\sharp}$ preserves ULA objects, and the functor $i_{b}^{!}$ preserves compact objects. If $A \in D(G_b(E),\Lambda)$ is compact, then $i_{b\ast}A$ has compact stalks.
\end{theorem}

As another application, with rational coefficients, we decompose the category $D(\Bun_G,\Lambda)$ into cuspidal and Eisenstein parts. More precisely, assume that $G$ is quasisplit, and define $D(\Bun_G,\Lambda)^{\Eis}$ as the full subcategory generated under colimits by the images of the functors $\Eis_{P!}$ for all proper parabolics $P$, and define $D(\Bun_G,\Lambda)^{\mathrm{cusp}}$ as the common kernel of the functors $\CT_{P\ast}$ for all proper parabolics $P$.
\begin{theorem}\label{thm:cuspEisdecomposition}
\begin{enumerate}
    \item The subcategory $D(\Bun_G,\Lambda)^{\mathrm{cusp}}$ is stable under all limits and colimits, and under Verdier duality. Any object $A \in D(\Bun_G,\Lambda)^{\mathrm{cusp}}$ is both $!$- and $\ast$-extended from the semistable locus $\Bun_{G}^{\mathrm{ss}} \subset \Bun_G$, and $i_{b}^{\ast} A \in D(G_b(E),\Lambda)$ is cuspidal in the sense of Vigneras (i.e all of its Jacquet modules with respect to proper parabolics are $0$) for any basic $b$.

    \item The pair $D(\Bun_G,\Lambda)^{\Eis}, D(\Bun_G,\Lambda)^{\mathrm{cusp}}$ forms a semiorthogonal decomposition of $D(\Bun_G,\Lambda)$. 

    \item When $\Lambda=\overline{\mathbb{Q}_\ell}$, there is a canonical \emph{orthogonal} decomposition
    \[D(\Bun_G,\overline{\mathbb{Q}_\ell}) = D(\Bun_G,\overline{\mathbb{Q}_\ell})^{\Eis} \oplus D(\Bun_G,\overline{\mathbb{Q}_\ell})^{\mathrm{cusp}} \]
    stable under Verdier duality, and (on compact objects) under Bernstein-Zelevinsky duality.
\end{enumerate}
\end{theorem}
In the proof of Theorem \ref{thm:cuspEisdecomposition}.(3), we use the theory of Bernstein components rather critically, which restricts us to characteristic zero coefficients. In fact, via reduction modulo $\ell$, cuspidal Bernstein components can meet with Eisenstein components, so it is necessary to invert $\ell$ in part (3).

Note that in classical geometric Langlands, one only has a semiorthogonal decomposition: an orthogonal decomposition as in (3) cannot hold, since the locus of irreducible local systems is dense in the stack of $L$-parameters. In our setting, however, the irreducible locus (i.e the locus defined by unramified twists of supercuspidal parameters) is open and closed in the stack of $L$-parameters $\mathrm{Par}_G$. In fact, writing $e^{\mathrm{irr}},e^{\mathrm{red}} \in \mathcal{O}(\mathrm{Par}_G)$ for the complementary idempotents cutting out the irreducible and reducible loci, the spectral action gives a similar decomposition
\[
D(\Bun_G,\overline{\mathbb Q_\ell})\cong e^{\mathrm{red}} D(\Bun_G,\overline{\mathbb{Q}_\ell})\oplus e^{\mathrm{irr}} D(\Bun_G,\overline{\mathbb{Q}_\ell})
\]
and some mild results on the compatibility of Eisenstein functors with $L$-parameters (that follow easily from our inductive arguments analyzing the Eisenstein functor in terms of parabolic induction and the gluing functor $i_{b\sharp}$) imply that
\[D(\Bun_G,\overline{\mathbb{Q}_\ell})^{\mathrm{Eis}}\subset e^{\mathrm{red}} D(\Bun_G,\overline{\mathbb{Q}_\ell})\]
and hence dually
\[D(\Bun_G,\overline{\mathbb{Q}_\ell})^{\mathrm{cusp}}\supset e^{\mathrm{irr}} D(\Bun_G,\overline{\mathbb{Q}_\ell})\]
We (and also Takaya \cite{Takaya}) conjecture that these inclusions are equalities.

\subsection{Other applications and future developments}
We end the introduction by commenting on the sequel paper \cite{HHS2} and some other works which rely on the results in this paper.

First, we note that some of the main results in the paper \cite{HamGeomES}, which were applied in \cite{HL} to prove vanishing theorems for the cohomology of \emph{global} Shimura varieties, crucially relied on Theorem \ref{thm:maingeometry}.(4) in the case where $P=B$ is a Borel.\footnote{This case does not seem substantially easier than the case of a general parabolic.} Thus the present paper makes \cite{HamGeomES} and \cite{HL} entirely unconditional.

In the follow-up paper \cite{HHS2}, we will establish the compatibility of Eisenstein and constant term functors with Hecke operators, and give applications to the cohomology of local Shimura varieties. We also expect to construct and analyze the intersection cohomology sheaf $\mathrm{IC}_{\wtBun_P}$ on $\wtBun_P$ and the associated compactified Eisenstein series functor $\Eis_{P!\ast}$. Since there is no general theory of perverse sheaves in $p$-adic geometry, the sheaf $\mathrm{IC}_{\wtBun_P}$ does not exist for soft reasons. We eventually realized that it is possible to construct it directly by adapting some ideas from \cite{GaitsgorySemiinf}.\footnote{We thank Joakim Faergeman and Andreas Hayash for patient explanations about these matters.} This construction requires us to import yet more techniques from geometric Langlands - in particular, ideas around Ran geometry and factorization structures - into the Fargues-Scholze setting.

Remarkably, the categorical geometric Langlands conjecture has a natural adaptation to the Fargues-Scholze setting, where it translates into a categorical upgrade of the local Langlands conjecture (see \cite[Conjecture I.10.2]{FarguesScholze} for a precise statement). In ongoing work \cite{HM}, the second author and Lucas Mann have formulated a program to prove this conjecture with $\overline{\mathbb{Q}_{\ell}}$-coefficients for many groups, in particular for $\mathrm{GL}_n$. Since the cuspidal part of the categorical conjecture is known for many groups (including $\mathrm{GL}_n$) by global methods, the decomposition of Theorem \ref{thm:cuspEisdecomposition}.(3) allows one to focus on the Eisenstein part, where the problem clearly has a flavor which is inductive on Levis. To deal with the Eisenstein part, a key ingredient in the strategy of \cite{HM} is the a priori compatibility of a certain functor from the automorphic side to the spectral side with $\Eis_{P!}$ and its spectral analogue. This compatibility will be addressed elsewhere, building on the techniques in this paper and \cite{HHS2}.

\subsection*{Acknowledgments} We would like to thank Brian Conrad, Joakim Faergeman, Dennis Gaitsgory, Andreas Hayash, Lucas Mann, and Xinwen Zhu for some very helpful conversations. Part of this work was carried out while the first author was visiting the MPIM, and he thanks them very much for their hospitality and support. Scholze was supported by a Leibniz Prize, and through the Hausdorff Centre of Mathematics (grant no. EXC-2047/1–390685813) at the MPIM Bonn.

\subsection*{Notation} Throughout this paper, $E$ denotes a fixed nonarchimedean local field, with ring of integers $\mathcal O_E$, uniformizer $\pi$ and residue field $\mathbb F_q$. We fix a separable closure $\ol{E}$ with completion $C$, and let $\ol{\mathbb F}_q$ be its residue field. We denote the absolute Galois group by $\Gamma=\mathrm{Gal}(\ol{E}/E)$. As usual, we denote by $\breve{E}\subset C$ the completion of the maximal unramified extension of $E$, with residue field $\ol{\mathbb F}_q$. We write $\sigma$ for the automorphism of $\breve{E}$ lifting the $q$-th power Frobenius.

For coefficients, we fix a $\mathbb Z_\ell[\sqrt{q}]$-algebra $\Lambda$. Using the fixed choice of $\sqrt{q}\in \Lambda$, we can form distinguished square roots of modulus characters, etc. As our sheaf theory, we will be using the base change of the motivic $6$-functor formalism from \cite{MotivicGeometrization} along the symmetric monoidal functor $D_{\mathrm{et,mot}}(\ol{\mathbb F}_q)\to D(\mathbb Z_\ell)\to D(\Lambda)$ given by the \'etale realization. If some power of $\ell$ vanishes in $\Lambda$, this agrees with $D_{\mathrm{et}}$ considered in \cite{FarguesScholze}. In general, when specialized to $\mathrm{Bun}_G$, it agrees with the hand-crafted $D_{\mathrm{lis}}$ considered in \cite{FarguesScholze}. Contrary to the $D_{\mathrm{lis}}$-theory, the motivic formalism satisfies excision in general, which we will excessively use.

Throughout, we fix a reductive group $G$ over $E$. As a rule, we do not make any further assumptions on $G$. At certain points, it is however convenient to assume that $G$ is quasisplit, in which case we usually fix a choice of maximal torus and Borel $T\subset B\subset G$.

As usual, $B(G) = G(\Breve{E})/(g \sim hg\sigma(h)^{-1})$ denotes the Kottwitz set of $G$. For $b\in B(G)$ (or rather a representative in $G(\breve{E})$) we write $G_b$ for the $\sigma$-centralizer of $b$. We recall that the Kottwitz set $B(G)$ is equipped with two maps:
\begin{enumerate}
    \item The slope homomorphism
    \[ \nu: B(G) \rightarrow \mathbb{X}_{*}(T_{\overline{E}})^{+,\Gamma}_{\mathbb{Q}} \]
    \[ b \mapsto \nu_{b}, \]
    where $\mathbb{X}_{*}(T_{\overline{E}})_{\mathbb{Q}}^{+}$ is the set of rational geometric dominant cocharacters of a maximal torus $T_{\ol{E}}\subset B_{\ol{E}}\subset G_{\ol{E}}$.
    \item The Kottwitz invariant
    \[ \kappa_{G}: B(G) \rightarrow \pi_{1}(G)_{\Gamma}, \]
    where $\pi_1(G)=\pi_1(G_{\ol{E}})$ denotes the algebraic fundamental group of Borovoi.
\end{enumerate}
We write $B(G)_{\basic} \subset B(G)$ for the set of basic elements; i.e.~the elements for which $\nu_{b}$ factors through the center of $G$. We follow the standard normalization of these invariants so that for $G=\mathbb G_m$, the element $b=\pi\in \breve E^\times$ maps (under both maps) to $1\in \pi_1(\mathbb G_m)=\mathbb Z$. If we identify $B(G)=|\Bun_G|$, this means, however, that $\nu(\mathcal O(1))=\kappa(\mathcal O(1))=-1$, noting that the line bundle $\mathcal O(1)$ corresponds to the $\sigma$-conjugacy class of $\pi^{-1}$. Thus, in bundle-theoretic terms, $-\nu$ gives the usual slopes.

By $\Bun_G$, we denote the stack of $G$-bundles, taking any perfectoid space $S$ to the groupoid of $G$-bundles on the Fargues-Fontaine curve $X_S$. We note that we always work over the fixed base $\ast := \Spd{\ol{\mathbb{F}}_{q}}$, and denote by $\Perf$ the category of perfectoid spaces in characteristic $p$ over $\ol{\mathbb F}_q$ endowed with the $v$-topology. For $b \in B(G)$, we write $\mathcal{F}_{b}$ for the associated $G$-bundle on $X_S$, for any $S\in \Perf$. We write $i_{b}: \Bun_{G}^{b} \hookrightarrow \Bun_{G}$ for the locally closed Harder-Narasimhan-stratum on it.

As usual, we denote by $\Div^{1} := \Spd{\Breve{E}}/\mathrm{Frob}^{\mathbb{Z}}$ the mirror curve, and, for a finite extension $E'/E$, we write $\Div^{1}_{E'}$ for the corresponding object for $E'$. More generally, for $I$ a finite index set, we let $\Div^{I}$ denote the product $\prod_{i \in I} \Div^{1}$. For $n \in \mathbb{Z}$, we let $\Div^{(n)} = (\Div^{1})^{n}/S_{n}$ denote the $n$th symmetric power of the mirror curve, where $S_{n}$ is the symmetric group on $n$ letters. 

We consider parabolics $P\subset G$ with Levi factor $M$ and unipotent radical $U$. Unless we also use an opposite parabolic $P^-$, we treat $M$ merely as a quotient of $P$ and do not fix an embedding $M\subset P$. We write $\tilde{\mathcal{J}}_{G,P}$ for the set of fundamental coroots corresponding to $P\subset G$; this is a $\Gamma$-invariant finite set of cocharacters $\mathbb G_m\to Z(M)_{\ol{E}}$. We let $\mathcal J_{G,P} = \tilde{\mathcal{J}}_{G,P}/\Gamma$ be the set of Galois orbits.

We write $M^{\mathrm{ab}}$ for the abelianization of $M$, and let $\Lambda_{G,P} := \Mcoinv=B(M^{\mathrm{ab}})$ denote the coinvariant lattice. The coroots define a canonical injective map
\[
\mathcal J_{G,P} = \tilde{\mathcal J}_{G,P}/\Gamma\to \Mcoinv = \Lambda_{G,P}: i\mapsto \alpha_i
\]
and we let $\Lambda_{G,P}^{\mathrm{pos}}\subset \Lambda_{G,P}$ be the free abelian monoid generated by $\mathcal J_{G,P}$.

The adjoint action of $P$ on $\Lie (U)$ induces 
\[
P \to \Aut( \det \Lie (U) ) \cong \mathbb{G}_{m}, 
\] 
which factors through $M^{\mathrm{ab}}$ and defines an element which we denote by $\xi_P \in \mathbb{X}^{*}(M^{\mathrm{ab}})$. 

\section{The geometry of $\Bun_{P}$ and Eisenstein functors}{\label{section: geometryofBunP}}
We will now review the geometry of $\Bun_{P}$ and the description of the geometric Eisenstein functors following \cite{HI}. 
\subsection{Eisenstein functors and the geometry of $\Bun_{P}$}{\label{section: defnofEisFunctors}}
Let $E$ be a nonarchimedean local field, let $G$ be a reductive group over $E$, and let $P \subset G$ be a parabolic with Levi quotient $M$. Let $M^{\mathrm{ab}}$ be its torus quotient. We recall that we have a natural diagram
\[ \begin{tikzcd}
\Bun_{P} \arrow[r,"\mf{p}_{P}"] \arrow[d,"\mf{q}_{P}"] \arrow[dd,bend left=60,"\mf{q}^{\dagger}_{P}"] & \Bun_{G}& \\
\Bun_{M} \arrow[d,"q_{M}"] & & \\
\Bun_{M^{\mathrm{ab}}}& &.
\end{tikzcd} \] 
By \cite[Corollary~IV.1.23]{FarguesScholze}, we have decompositions 
\[ \Bun_{M} = \bigsqcup_{\alpha \in \pi_{1}(M)_{\Gamma}} \Bun_{M}^{\alpha} \text{  and   
 } \Bun_{M^{\mathrm{ab}}} = \bigsqcup_{\ol{\alpha} \in \Mcoinv} \Bun_{M^{\mathrm{ab}}}^{\ol{\alpha}} \]
into connected components via the Kottwitz invariant. We let $\Bun_{P}^{\alpha}$ denote the fiber of $\mf{q}_{P}$ over the connected component $\Bun_{M}^{\alpha}$, and write $\mf{q}_{P}^{\alpha}$ (resp. $\mf{p}_{P}^{\alpha}$) for the base-change of $\mf{q}_{P}$ (resp. $\mf{p}_{P}$) to this connected component. We write $\ol{(-)}$ for the natural composition: 
\[ \pi_{1}(M)_{\Gamma} \xrightarrow{\simeq}^{\kappa_{M}^{-1}} B(M)_{\basic} \ra B(M^{\mathrm{ab}}) \xrightarrow{\simeq}^{\kappa_{M^{\mathrm{ab}}}} \mathbb{X}_{*}(M^{\mathrm{ab}}_{\ol{E}})_{\Gamma}. \]
This map is an isomorphism when the derived group of $M$ is simply connected, and in general is surjective with finite fibers. In Section~\ref{sec: fixingthecenter}, we will explain a general reduction to the case that the derived group of $G$ (hence, that of $M$) is simply connected, so the reader is invited to make this extra assumption.

\begin{remark}
The notation $\Bun_{M}^{\alpha}$ for the connected component attached to $\alpha \in \pi_{1}(M)_{\Gamma} \xrightarrow{\simeq}^{\kappa_{M}^{-1}} B(M)_{\basic}$ is perhaps a bit confusing, since it could also refer to the HN-strata attached to the element $\alpha$. To alleviate the confusion, we will always write $b_{\alpha} \in B(M)_{\basic}$ for the element attached to $\alpha \in \pi_{1}(M)_{\Gamma}$ under the isomorphism $\kappa_{M}^{-1}$, and in turn always write $\Bun_{M}^{b_{\alpha}}$ for the HN-strata.
\end{remark}

We recall the following result.
\begin{proposition}{\label{prop: qissmooth}}{\cite[Proposition~3.5]{HI}}
The map $\mf{q}_{P}$ is a cohomologically smooth map of Artin $v$-stacks with connected fibers. Over the connected component $\Bun_{M}^{\alpha}$ it is pure of $\ell$-dimension equal to $d_{\alpha} := \langle \xi_{P}, \ol{\alpha} \rangle$. 
\end{proposition}
This gives the following corollary, using that cohomologically smooth maps are universally open.
\begin{corollary}{\label{cor: BunPconncomp}}
We have a decomposition into connected components:
\[ \Bun_{P} = \bigsqcup_{\alpha \in B(M)_{\basic}} \Bun_{P}^{\alpha}. \]
\end{corollary}
We recall that $\Bun_{M}$ is cohomologically smooth of pure $\ell$-dimension $0$ by \cite[Theorem~1.4.1]{FarguesScholze}. It therefore follows from the previous corollary that $\Bun_{P}$ is cohomologically smooth. We write 
\[ \dim(\Bun_{P}): |\Bun_{P}| \ra \mathbb{Z} \]
for the locally constant function specified by the $\ell$-dimension. The previous Proposition tells us that this is equal to $d_{\alpha}$ on the connected component $\Bun_{P}^{\alpha}$. However, the dualizing complex on $\Bun_{P}$ also includes some character twists which one needs to incorporate in order to get the morally correct ``normalized'' Eisenstein functor.
\begin{definition}
Let $\delta_{P}=|\xi_P|: M(E) \ra \Lambda^{*}$ (resp. $\delta_{P}^{1/2}=|\xi_P|^{\tfrac 12}$) denote the modulus character of $P$ (resp.~square root of the modulus character of $P$, fixed by our choice of $\sqrt{q}\in \Lambda$). As $\xi_P$ factors over $M^{\mathrm{ab}}$, these factor over characters $\ol{\delta}_{P}$ (resp. $\ol{\delta}_{P}^{1/2}$) of $M^{\mathrm{ab}}(E)$. We let $\ol{\Delta}_{P} \in D(\Bun_{M^{\mathrm{ab}}},\Lambda)$ (resp. $\ol{\Delta}_{P}^{1/2}$), be the sheaf whose value on the connected component $\Bun_{M^{\mathrm{ab}}}^{\ol{\alpha}} \simeq [\ast/\underline{M^{\mathrm{ab}}(E)}]$ is given by the smooth characters $\ol{\delta}_{P}$ (resp. $\ol{\delta}_{P}^{1/2}$) for all $\ol{\alpha} \in \mathbb{X}_{*}(M^{\mathrm{ab}}_{\ol{E}})_{\Gamma}$. We  write $\Delta_{P}$ (resp. $\Delta_{P}^{1/2}$) for $q_{M}^{*}(\ol{\Delta}_{P})$ (resp. $q_{M}^{*}(\ol{\Delta}_{P}^{1/2})$), and set $\IC_{\Bun_{P}} := \mf{q}_{P}^{*}(\Delta_{P}^{1/2})[\dim(\Bun_{P})]$.
\end{definition}
The key result is then as follows. 
\begin{theorem}{\cite{HI}}{\label{thm: dualcomplexonBunP}}
There is an isomorphism 
\[ K_{\Bun_{P}} \simeq \mf{q}_{P}^{*}(\Delta_{P})[2\dim(\Bun_{P})] \]
of sheaves in $D(\Bun_{P},\Lambda)$. In particular, the sheaf $\IC_{\Bun_{P}}$ is Verdier self-dual. 
\end{theorem}
We recall that we have an isomorphism $K_{\Bun_{M}} \simeq \Lambda$ \cite[Proposition~1.1]{HI}. Using this, we deduce the following corollary. 
\begin{corollary}{\label{cor: reldualcomploverBunM}}
We have an isomorphism 
\[ K_{\Bun_{P}/\Bun_{M}} \simeq \mf{q}_{P}^{*}(\Delta_{P})[2\dim(\Bun_{P})] \]
of sheaves in $D(\Bun_{P},\Lambda)$.
\end{corollary}

\begin{definition}
We define the normalized Eisenstein functors
\[ \nmEisshriekP: D(\Bun_{M},\Lambda) \ra D(\Bun_{G},\Lambda)   \]
\[ A \mapsto \mf{p}_{P!}(\mf{q}_{P}^{*}(A) \otimes \IC_{\Bun_{P}}) \]
and
\[ \nmEisstarP: D(\Bun_{M},\Lambda) \ra D(\Bun_{G},\Lambda) \]
\[ A \mapsto \mf{p}_{P*}(\mf{q}_{P}^{*}(A) \otimes \IC_{\Bun_{P}}). \]
Similarly, we define the normalized constant term functors
\[\CT_{P*}: D(\Bun_{G},\Lambda) \ra D(\Bun_{M},\Lambda) \] 
\[ A \mapsto \mathfrak{q}_\ast(\mathfrak{p}^!(-) \otimes \IC_{\Bun_P}^{-1}) \]
and
\[ \CT_{P!}: D(\Bun_{G},\Lambda) \ra D(\Bun_{M},\Lambda) \]
\[ A \mapsto \mf{q}_{!}(\mathfrak{p}^{*}(A) \otimes \IC_{\Bun_{P}}). \]
If we want to emphasize the dependence on the group $G$, we will do this by using the superscript $(-)^{G}$.
\end{definition}
Using Corollary \ref{cor: BunPconncomp}, we get evident decompositions of all of these functors. We use the superscript $(-)^\alpha$ to denote the constituents of these decompositions, so e.g. $\Eis_{P!}^\alpha:D(\Bun_{M}^{\alpha},\Lambda)\to D(\Bun_{G},\Lambda)$ is the precomposition of $\Eis_{P!}$ with extension by zero from the connected component labelled by $\alpha$. 

For completeness, we note the following simple observation.
\begin{lemma}{\label{lemma: constanttermEisensteinSeriesAdjunctions}}
There are adjunctions 
\[ \nmEisshriekP \dashv \CT_{P*} \]
and
\[ \CT_{P!} \dashv \nmEisstarP \]
\end{lemma}
\begin{proof}
The first adjunction follows immediately from the definitions, and the second adjunction follows from the definitions and the identifications: $\mf{q}_{P}^{!}(-) \simeq \mf{q}_{P}^{*}(-) \otimes \mf{q}_{P}^{*}(\Delta_{P}) = \mf{q}_{P}^{*}(-) \otimes \IC_{\Bun_{P}}^{\otimes 2}$ guaranteed by Corollary \ref{cor: reldualcomploverBunM}.
\end{proof}
We now give some more flavor for how these functors behave.
\subsection{$M$-reducible elements and stalks of Eisenstein series}
In this section, we will describe the stalks of the geometric Eisenstein and constant term functors in a variety of situations that will be relevant for our arguments later. We will begin with the easiest case.  
\begin{corollary}\label{cor: Eisparabinduction}
We have identifications
\begin{enumerate}
\item $\Eis_{P!}i^{M}_{1!}(-) \simeq i_{1!}i_{P}^{G}(-)$
\item $\Eis_{P*}i^{M}_{1*}(-) \simeq i_{1*}i_{P}^{G}(-)$
\item  $i_{1}^{*,M}\CT_{P!}(-) \simeq r_{P}^{G}i_{1}^{*}(-)$
\item $i_{1}^{*,M}\CT_{P*}(-) \simeq r_{P^{-}}^{G}i_{1}^{*}(-)$
\end{enumerate}
In the last part, $P^-\subset G$ denotes any parabolic opposite to $P$.
\end{corollary}
\begin{proof}
We have an identification $\Bun_{P}^{1_{M}} \simeq [\ast/\ul{P(E)}]$, which follows from HN-slope considerations. The first two isomorphisms follow easily from this identification combined with smooth and proper base-change, noting that the map $[\ast/\ul{P(E)}] \ra [\ast/\ul{G(E)}]$ is proper. The last two isomorphism follow by passing to left/right adjoints via Lemma \ref{lemma: constanttermEisensteinSeriesAdjunctions}, Frobenius reciprocity, and classical second adjointness \cite[Corollary~1.3]{DHKM}.
\end{proof}
\begin{remark}{\label{rem: relationtoclassicalsecondadjointness}}
In particular, this result (and its proof) showcases the direct connection between geometric second adjointness (Theorem \ref{thm:maintheorem}) and classical second adjointness. 
\end{remark}
We will now move into a slightly more complicated situation. In order to formulate our results properly, we need to properly take into account certain modulus character twists. To do this, we will introduce the following renormalized functors. We recall $i_{b\sharp}$, the left adjoint of the $*$-pullback considered in \cite[Proposition~VII.7.2]{FarguesScholze}. 
\begin{definition}
For $b \in B(G)$, we let $\delta_{b}: G_{b}(E) \ra \Lambda^{*}$ be the character defined in \cite[Definition~3.14]{HI}, which by our choice of $\sqrt{q}$ has a distinguished square root. For $? \in \{\sharp,*,!\}$, we define 
\[ i_{b?}^{\ren}(-) := i_{b?}^{\ren}(- \otimes \delta_{b}^{-1/2})[-\langle 2\rho_{G}, \nu_{b} \rangle] \]
Similarly, for $? \in \{*,!\}$, we define
\[ i_{b}^{\ren?}(-) := i_{b}^{?}(- \otimes \delta_{b}^{1/2})[\langle 2\rho_{G}, \nu_{b} \rangle]. \]
We note that we have natural adjunctions 
\[ i_{b!}^{\ren} \vdash i_{b}^{\ren!} \]
and 
\[ i_{b\sharp}^{\ren} \vdash i_{b}^{\ren*} \vdash i_{b*}^{\ren} \]
inherited from the usual adjunction relationships.
\end{definition}
The point is that while one has a canonical identification $D(\Bun_{G}^{b},\Lambda) = D(G_{b}(E),\Lambda)$ as in \cite[Proposition~VII.7.1]{FarguesScholze}, the space $\tilde{G}_{b}$ carries an $\ell$-adically contractible unipotent part whose compactly supported cohomology/dualizing complex carries a non-trivial $G_{b}(E)$-structure. In particular, these twists give rise to the relationship
\[ \mathbb{D}_{\mathrm{Verd}}i_{b!}^{\ren} \simeq i_{b*}^{\ren}\mathbb{D}_{\mathrm{sm}} \]
using \cite[Corollary~1.6]{HI}, where $\mathbb D_{\mathrm{sm}}$ denotes smooth duality.

Assume now that $G$ is quasisplit and fix a Borel $B\subset G$ as well as a maximal torus $T\subset B$. In that case, for any $b\in B(G)$, the slope homomorphism $\nu_b$ has a unique representative as a dominant rational cocharacter of $T$. For a fixed $b \in B(G)$, we let $G_{\{ b \}}$ denote the centralizer of $\nu_b$. We recall, by \cite[Propositions~6.2,6.3]{KottIsocrystalsI} that there exists a unique basic reduction of $b \in B(G)$ to $B(G_{\{b\}})$, which we abusively also denote by $b$, satisfying the property that $\nu_{b,G_{\{b\}}}$ is dominant with respect to the choice of Borel; in other words, $\nu_{b,G_{\{b\}}}$ agrees with $\nu_b$, as rational cocharacters of $T$. The basic reduction has the additional feature that $(G_{\{b\}})_{b} \simeq G_{b}$. In particular, $G_{\{b\}}$ is a quasi-split inner form of $G_b$. We take $P_{\{b\}}$ to be the dynamic parabolic of the slope cocharacter $\nu_{b}$, which is the standard parabolic with Levi $G_{\{b\}}$. Let $P^-_{\{b\}}$ be the opposite parabolic.

If we write $\Bun_{P_{\{b\}}}^{b}$ and $\Bun_{P^{-}_{\{b\}}}^{b}$ for the fibers of $\mf{q}_{P_{\{b\}}}$ and $\mf{q}_{P_{\{b\}}^{-}}$ over $i_{b}^{G_{\{b\}}}: \Bun_{G_{\{b\}}}^{b} \hookrightarrow \Bun_{G_{\{b\}}}$, respectively, then we have isomorphisms 
\[ \Bun_{P_{\{b\}}}^{b}\simeq \mathcal{M}_{b} \]
and
\[ \Bun_{P_{\{b\}}^{-}}^{b} \xrightarrow{\simeq} \Bun_{G}^{b}, \]
as in \cite[Corollary~1.6]{HI}. Here $\pi_{b}: \mathcal{M}_{b} \ra \Bun_{G}$ is the cohomologically smooth chart  of $\Bun_{G}$ considered in \cite[Section~V.3]{FarguesScholze}. Moreover, the map $\mf{q}$ induces the natural maps $q_{b}: \mathcal{M}_{b} \ra \Bun_{G_{\{b\}}}^{b} \simeq [\ast/\ul{G_{b}(E)}]$ and $p_{b}: \Bun_{G}^{b} \simeq [\ast/\tilde{G}_{b}] \ra [\ast/\ul{G_{b}(E)}]$ considered in \cite{FarguesScholze}. 
\begin{remark}{\label{rem: confusingslopes}}
We note that, since $\nu_b$ is dominant and $P_{\{b\}}$ is a standard parabolic, we would expect that $\Bun_{P_{\{b\}}}^{b}$ should parametrize split $P_{\{b\}}$-structures and not $\Bun_{P_{\{b\}}^{-}}^{b}$. This counter intuitive behavior comes from the minus sign when passing between $G$-bundle and isocrystals, which leads to the Harder-Narasimhan slopes of the bundle $\mathcal{F}_{b}$ being the negatives of the isocrystal slopes attached to $b$. One could of course normalize this sign away, but then one would run into discrepancies when comparing with the literature on local and global Shimura varieties.
\end{remark}
We obtain the following description of the functors described above.
\begin{corollary}\label{cor: dominantreductionConstantEisensteinDescription}
With notation as above, we have the following isomorphisms. Note that for $i_b^{G_{\{b\}}}$, the renormalized functors agree with the standard functors as $b\in B(G_{\{b\}})_{\basic}$.
\begin{enumerate}
\item $\Eis_{P_{\{b\}}^{-}!}i_{b!}^{G_{\{b\}}}(-) \simeq i_{b!}^{\ren}(-)$
\item $\Eis_{P_{\{b\}}^{-}*}i_{b\ast}^{G_{\{b\}}}(-) \simeq i_{b*}^{\ren}(-)$
\item $i_{b}^{G_{\{b\}}*}\CT_{P_{\{b\}}^{-}*}(-) \simeq i_{b}^{\ren!}(-)$
\item $i_{b}^{G_{\{b\}}*}\CT_{P_{\{b\}}^{-}!}(-) \simeq i_{b}^{\ren*}(-)$
\item $\Eis_{P_{\{b\}}!}i_{b!}^{G_{\{b\}}}(-) \simeq i_{b\sharp}^{\ren}(-)$.
\end{enumerate}
\end{corollary}
\begin{proof}
The first and second isomorphisms follow easily from the isomorphism $\Bun_{P_{\{b\}}^{-}}^{b} \simeq [\ast/\tilde{G}_{b}]$, and the cohomological smoothness of $\mathfrak q_P$. Recall that the identification $D(\Bun_{G}^{b},\Lambda) \simeq D(G_{b}(E),\Lambda)$ is induced by the pullback map $p_{b}: \Bun_{G}^{b} \simeq [\ast/\tilde{G}_{b}] \ra [\ast/\ul{G_{b}(E)}]$.

The next two isomorphisms follow from the first by passing to adjoints via Lemma \ref{lemma: constanttermEisensteinSeriesAdjunctions}. It remains to justify the fifth isomorphism. To do this we note, by \cite[Proposition~VII.7.2]{FarguesScholze}, that the left adjoint $i_{b\sharp}$ is isomorphic to $\pi_{b\natural}q_{b}^{*}$, where $\pi_{b\natural}$ is the left adjoint of $\pi_{b}^{*}$. However, by \cite[Theorem~V.3.7]{FarguesScholze}, the map $\pi_{b}$ is cohomologically smooth and therefore it follows that $\pi_{b!}(- \otimes \pi_{b}^{!}(\Lambda)) \simeq \pi_{b\natural}(-)$. However, we note by \cite[Proposition~1.1]{HI}, the dualizing complex on $\Bun_{G}$ is the constant sheaf, and therefore it follows that $\pi_{b}^{!}(\Lambda)$ identifies with the dualizing complex of $\mathcal{M}_{b}$, which by \cite[Corollary~1.6]{HI} is isomorphic to $q_{b}^{*}(\delta_{b})[2\langle 2\rho_{G}, \nu_{b} \rangle]$. Therefore, we deduce that $\pi_{b!}q_{b}^{*}(- \otimes \delta_{b})[2\langle 2\rho_{G}, \nu_{b} \rangle] \simeq i_{b\sharp}(-)$, which implies that 
\[  \pi_{b!}q_{b}^{*}(- \otimes \delta_{b}^{1/2})[\langle 2\rho_{G}, \nu_{b} \rangle] \simeq i_{b\sharp}^{\ren}(-), \]
but the LHS is precisely $\Eis_{P_{b}!}i_{b!}^{G_{\{b\}}}$, by the above discussion and proper base-change.
\end{proof}
\begin{remark}
We note that in this particular case Theorem \ref{thm:maintheorem} (3) can be deduced from the adjunctions 
\[ i_{b!}^{\ren} \vdash i_{b}^{\ren!} \]
and 
\[ i_{b\sharp}^{\ren} \vdash i_{b}^{\ren*} \vdash i_{b*}^{\ren}. \]
\end{remark}

In the remainder of this section, we combine the previous propositions to give a description of the stalks of geometric Eisenstein series in certain simple cases. More precisely, we evaluate $\Eis_{P!}$ when restricted to basic strata of $\Bun_M$ that are either dominant or anti-dominant with respect to $P$. We continue to assume that $G$ is quasisplit, and fix $T\subset B\subset G$. We first recall a combinatorial description of the following set, for some standard Levi $M$ of $G$.

\begin{definition}
We define $B(G)_{M} := \mathrm{Im}(i_{M}: B(M)_{\basic} \ra B(G))$, the set of elements of $B(G)$ admitting a basic reduction to $M$.
\end{definition}

For $b \in B(G)_{M}$, we consider the collection of Weyl group elements $W_{M,b} := W[M,G_{\{b\}}]$, as defined in \cite[Section~5.3]{BMO}. We recall the definition: Let $W_G$ denote the (relative) Weyl group of $G$, i.e.~the quotient
\[ W_G = N(T)(E)/T(E). \]
Then $W[M,G_{\{b\}}]\subset W_G$ is the set of elements $w\in W_G$ such that 
\[ w(M) \subset G_{\{b\}} \]
\[ w^{-1}(G_{\{b\}} \cap B) \subset B. \] 

We will use the following notions of dominance throughout the paper. Let $\Mcoinvdom \subset \Mcoinv$ denote the positive Weyl chamber inside the coinvariant lattice for a standard Levi $M$, where this is defined with respect to our fixed choice of Borel.

\begin{definition}
We say an element $\theta \in B(M)_{\basic}$ is \emph{dominant} if $\ol{\theta}$ lies in the positive Weyl chamber $\Mcoinvdom \subset \Mcoinv$. Similarly, if it lies in the negative Weyl chamber we say that it is \emph{anti-dominant}.
\end{definition}

\begin{remark} The previous definition applies more generally for any reductive group $G$ equipped with a parabolic $P\subset G$ with Levi $M$: We say that an element $\theta\in B(M)_{\basic}$ is \emph{$P$-dominant} (resp. \emph{$P$-antidominant}) if $\kappa_M(\theta)$ lies in the positive (resp. negative) Weyl chamber defined with respect to any choice of Borel $B\subset P_{\ol{E}}$ .\footnote{Note that as $\theta$ is assumed to be basic, this notion depends only on $P$ and not the choice of Borel $B\subset P_{\ol{E}}$.}
\end{remark}

We have the following lemma.
\begin{lemma}{\cite[Lemma~4.23]{HI}}{\label{lemma: paramoffibers}}
For every element $b \in B(G)_{M}$, there exists an injective map of sets
\begin{align*}
i_{M}^{-1}(b) &\ra W_{M,b} \\  
b_{\alpha} &\mapsto w_{\alpha} 
\end{align*}
that sends $b_{\alpha}$ attached to $\alpha \in \pi_{1}(M)_{\Gamma}$ to the unique $w_{\alpha} \in W_{M,b}$ such that $w_{\alpha}(b_{\alpha}) \in B(w_{\alpha}(M))_{\basic}$ is dominant. It follows that $w_{\alpha}(b_{\alpha}) \in B(w_{\alpha}(M))_{\basic}$ is a reduction of $b \in B(G_{\{b\}})_{\basic}$ to $B(w_{\alpha}(M))_{\basic}$. Moreover, the image of this map is given by the set of $w \in W_{M,b}$ such that $w(M) \subset G_{\{b\}}$ transfers to a Levi subgroup of $G_{b}$ under the inner twisting $(G_{\{b\}})_{b} \simeq G_{b}$. 
\end{lemma}
The element $w_{\alpha}$ will be important for describing the stalks of the Eisenstein functor. We will also need the following complementary lemma. Here $P$ denotes the standard parabolic with Levi $M$.
\begin{lemma}{\cite[Lemma~4.24]{HI}}{\label{lemma: parabolicstransfer}}
An element $b \in B(G)$ lies in $B(G)_{M}$ if and only if there exists $w \in W_{M,b}$ such that the parabolic $w(P) \cap G_{\{b\}}$ of $G_{\{b\}}$ transfers to a parabolic subgroup $Q_{b,w} \subset G_{b}$ under the inner twisting between $G_{\{b\}}$ and $G_{b}$. More precisely, if $b_{\alpha}$ maps to $b \in B(G)$ with corresponding Weyl group element $w_{\alpha}$ as in the previous lemma then $w_{\alpha}(P) \cap G_{\{b\}}$ transfers to a parabolic subgroup of $G_{\{b\}}$. Moreover, for every element $\alpha \in B(M)_{\mathrm{basic}}$ mapping to $b$, the Levi factor of $Q_{b,w_{\alpha}}$ is equal to $w_{\alpha}(M)_{w_{\alpha}(b_{\alpha})}$.
\end{lemma}

We write 
\[ (-)^{w_{\alpha}}: D(M_{b_\alpha},\Lambda) \xrightarrow{\simeq} D(w_{\alpha}(M)_{w_{\alpha}(b_{\alpha})},\Lambda) \]
for the equivalence induced by the isomorphism $M_{b_\alpha} \simeq w_{\alpha}(M)_{w_{\alpha}(b_{\alpha})}$ given by $w_{\alpha}$. We write $(-)^{w_{\alpha}^{-1}}$ for its inverse.

We now have the following, which is simply a reinterpretation of \cite[Theorem~4.26]{HI} in terms of the renormalized pushforwards.
\begin{proposition}{\label{prop: splitcontribution to Eisensteinfunctor}} 
For all $\alpha \in \pi_{1}(M)_{\Gamma}$ corresponding to $b_{\alpha} \in B(M)_{\mathrm{basic}}$ mapping to $b \in B(G)_{M}$ with associated Weyl group element $w_{\alpha} \in W_{M,b}$, we consider $A \in D(M_{b_{\alpha}}(E),\Lambda)$. Then we have an isomorphism
\[ i_{b}^{\ren*}\Eis_{P!}i_{b_{\alpha}!}(A) \simeq i_{Q_{b,w_{\alpha}}}^{G_{b}}(A^{w_{\alpha}}) \]
in $D(\Bun_{G},\Lambda)$. Moreover, if $b_{\alpha}$ is anti-dominant then this defines an isomorphism
\[ \Eis_{P!}i_{b_{\alpha!}}(-) \simeq i_{b!}^{\ren}i_{Q_{b,w_{\alpha}}}^{G_{b}}(-)^{w_{\alpha}}: D(M_{b_{\alpha}}(E),\Lambda)\ra D(\Bun_{G},\Lambda) \]
where $w_{\alpha}$ is a minimal length representative of the element of longest length. 
\end{proposition}

Using this, we can also deduce the following description of the functor $\CT_{P*}$.
\begin{corollary}{\label{cor: shriekEisandStarCT}}
With notation as in Proposition \ref{prop: splitcontribution to Eisensteinfunctor}, we consider $\alpha \in \pi_{1}(M)_{\Gamma}$ such that $b_{\alpha}$ is anti-dominant. Then we have the following identification of functors
\[
i_{b_{\alpha}}^{*}\CT_{P*}(-) \simeq (r_{Q_{b,w_{\alpha}}^{-}}^{G}i_{b}^{\ren!}(-))^{w_{\alpha}^{-1}}: D(\Bun_{G},\Lambda) \ra D(M_{b_{\alpha}}(E),\Lambda).
\]
\end{corollary}
\begin{proof}
This follows from the previous proposition by passing to right adjoints and using Lemma \ref{lemma: constanttermEisensteinSeriesAdjunctions} as well as usual second adjointness \cite[Corollary~1.3]{DHKM}.
\end{proof}

We can also give a more general description of the dominant case.
\begin{corollary}{\label{cor: naturalfunctorscalculation}}
Let $\alpha \in \pi_{1}(M)_{\Gamma}$ be a basic element such that $b_{\alpha} \in B(M)_{\basic}$ is dominant and maps to $b \in B(G)$. Then $w_\alpha=1$ and we have natural isomorphisms
\[ \Eis_{P!}i_{b_{\alpha}!} \simeq i_{b\sharp}^{\ren} \circ i_{Q_{b,1}}^{G_{b}}: D(M_{b_{\alpha}}(E),\Lambda)\ra D(\Bun_{G},\Lambda) \]
and
\[
 i_{b_\alpha}^\ast \CT_{P\ast}\simeq r_{Q_{b,1}^-}^{G_b} i_b^{\ren\ast}: D(\Bun_{G},\Lambda) \ra D(M_{b_{\alpha}}(E),\Lambda).
\]
\end{corollary}
\begin{proof} Using compatibility of Eisenstein functors with composition, along $P\subset P_{\{b\}}\subset G$, the first statement is a combination of Corollary~\ref{cor: Eisparabinduction}~(1) and Corollary~\ref{cor: dominantreductionConstantEisensteinDescription}~(5). The second statement follows by passing to right adjoints.
\end{proof}

\section{Finiteness conditions and first results}

In this section, we will begin our study of the finiteness properties of the functors $\Eis_{P!}$ and $\Eis_{P*}$. Here, $G$ denotes a reductive group over $E$ with parabolic subgroup $P\subset G$ and Levi quotient $M$. We do not assume that $G$ is quasisplit.

\subsection{Preservation of compact objects}

We recall that the derived category $D(\Bun_{G},\Lambda)$ has two natural full subcategories of sheaves satisfying certain finiteness properties. 
\begin{definition}{\label{defn: compactandULAobjects}}
\begin{enumerate}
\item There is the full subcategory $D(\Bun_{G},\Lambda)^{\omega}$ of compact objects. Equivalently, by \cite[Theorem~I.5.1 (iii)]{FarguesScholze}, these are the objects $A \in D(\Bun_{G},\Lambda)$ which have quasi-compact support, and their restriction to the HN-strata $i_{b}^{*}(A) \in D(\Bun_{G}^{b},\Lambda) \simeq D(G_{b}(E),\Lambda)$ defines a compact object, where this is equivalent to lying in the thick triangulated subcategory generated by objects $\cInd_{K}^{G_{b}(E)}(\Lambda)$ for $K \subset G_{b}(E)$ an open pro-$p$ subgroup.
\item There is the full subcategory of $D^{\mathrm{ULA}}(\Bun_{G},\Lambda)$ of objects which are ULA over $\ast$. Equivalently, by \cite[Theorem~I.5.1 (v)]{FarguesScholze}, these are the objects $A \in D(G_{b}(E),\Lambda)$ whose restriction $i_{b}^{*}(A)$ to the HN-strata $D(\Bun_{G}^{b},\Lambda) \simeq D(G_{b}(E),\Lambda)$ are admissible. In other words, $i_{b}^{*}(A)^{K}$ is a perfect complex for all open pro-$p$ subgroups $K \subset G_{b}(E)$.
\end{enumerate}
\end{definition}

We will now commence our study on how these subcategories interact with the Eisenstein and Constant Term functors. In this section, we prove the following theorem.

\begin{theorem}\label{thm:Eiscompact} The functor $\mathrm{Eis}_{P!}$ preserves compact objects.
\end{theorem}

\begin{corollary}\label{cor:CTprelim} The functor $\CT_{P \ast}$ commutes with all limits and colimits, and preserves ULA objects.
\end{corollary}

\begin{proof} Since $\CT_{P\ast}$ is a right adjoint, it clearly commutes with limits. Since its left adjoint preserves compact objects by Theorem \ref{thm:Eiscompact}, it also commutes with colimits. For preservation of ULA objects, we observe more generally that if $L:D(\Bun_G,\Lambda) \leftrightarrows D(\Bun_H,\Lambda) : R$ are a pair of adjoint functors for some reductive groups $G,H$ such that $L$ preserves compact objects, then $R$ preserves ULA objects. This follows immediately from the characterization of ULA objects implicit in \cite[Prop. VII.7.4 and Prop. VII.7.9]{FarguesScholze}.
\end{proof}

\begin{proof}[Proof of Theorem \ref{thm:Eiscompact}] Assume first that $G$ is quasisplit, and fix $T\subset B\subset G$ as usual. 
By induction, we can assume that the theorem is true for all proper Levi subgroups of $G$. We will prove the following claim by induction on quasicompact open substacks $U\subset \Bun_G$ (equivalently, finite subsets $|U|\subset |\Bun_G|=B(G)$ stable under generization).

\begin{claim} For any standard Levi $M\subset G$ and any $b\in B(M)$ mapping into $|U|\subset B(G)$, and any (semistandard) parabolic $P$ with Levi $M$, the functor
\[
\mathrm{Eis}_{P!}\circ i_{b!}^{\ren,M}: D(M(E),\Lambda)\to D(\Bun_G,\Lambda)
\]
preserves compact objects.
\end{claim}

Consider a general pair $(P,b)$ as above. Our goal is to show that $\mathrm{Eis}_{P!}\circ i_{b!}^{M}$ preserves compact objects.
By induction, we can assume that $\mathrm{Eis}_{P'!}\circ i_{b'!}^{M'}$ preserves compact objects for every pair $(P',b')$ such that the image of $b'$ in $B(G)$ is \emph{strictly} more semistable than the image of $b$.

First, by reducing $b$ further to a basic element of a Levi subgroup of $M$, and using compatibility of Eisenstein functors with composition and Corollary~\ref{cor: dominantreductionConstantEisensteinDescription}~(3), we can reduce to the case that $b=b_\alpha\in B(M)$ is basic. Then there is the Levi subgroup $G_{\{b\}}\subset G$ and the element $w_\alpha\in W[M,G_{\{b\}}]$. We have $M\subset w_\alpha^{-1}(G_{\{b\}})$, and the image of $b$ in $B(w_\alpha^{-1}(G_{\{b\}}))$ stays basic. We can then use compatibility of Eisenstein functors with composition again to reduce to the case $M=w_\alpha^{-1}(G_{\{b\}})$. As $P$ was an arbitrary parabolic with Levi $M$, we can undo the Weyl group twist. In other words, we can assume that $M=G_{\{b\}}$. We have to show that for all parabolics $P$ with Levi $M=G_{\{b\}}$, the functor
\[
\mathrm{Eis}_{P!}\circ i_{b!}^M
\]
preserves compact objects. We note that this is true for the standard parabolic, by Corollary~\ref{cor: naturalfunctorscalculation} and the preservation of compact objects under parabolic induction and the functors $i_{b\sharp}$.

We will now show that the result is true for some parabolic $P$ with Levi $M$ if and only if it is true for any other $P'$. Moving from one parabolic with Levi $M$ to any other via a sequence of minimal modifications, we can assume that $P$ and $P'$ are contained in some parabolic $Q$ whose Levi $N\supset M$ has the property that $P\cap N$ and $P'\cap N$ are maximal parabolic subgroups. In particular, one of them is standard and the other is anti-standard in $N$; without loss of generality $P\cap N$ is the standard one. Then
\[
\mathrm{Eis}_{P\cap N!}\circ i_{b!}^M = i_{b\sharp}^{\mathrm{ren},N}
\]
and
\[
\mathrm{Eis}_{P'\cap N!}\circ i_{b!}^M = i_{b!}^{\mathrm{ren},N}
\]
by Corollary~\ref{cor: dominantreductionConstantEisensteinDescription}~(5) resp.~(1). In particular, there is a triangle
\[
i_{b\sharp}^{\mathrm{ren},N}\to i_{b!}^{\mathrm{ren},N}\to C
\]
where the cone $C$ is supported at points of $\Bun_N$ whose image in $\Bun_G$ is strictly more semistable than $b$, using the Lemma below. Moreover, the functor $C$ preserves compact objects.

Applying $\mathrm{Eis}_{Q!}$ to the last displayed triangle, and using the isomorphisms before as well as the composability of Eisenstein functors, we get a triangle
\[
\mathrm{Eis}_{P!}\circ i_{b!}^M\to \mathrm{Eis}_{P'!}\circ i_{b!}^M\to \mathrm{Eis}_{Q!}\circ C.
\]
By induction, $\mathrm{Eis}_{Q!}\circ C$ preserves compact objects. It follows that $\mathrm{Eis}_{P!}\circ i_{b!}^M$ preserves compact objects if and only if $\mathrm{Eis}_{P'!}\circ i_{b!}^M$ does, as desired.

This finishes the case that $G$ is quasisplit. If $G$ has connected center, then so does $M$, and then there is a basic element of $M$ such that $M_b$ is quasisplit, hence so is $G_b$. Twisting everything by $b$ and using $\Bun_{M_b}\cong \Bun_M$ etc., this reduces the case of $G$ to the case of $G_b$.

It remains to reduce the case of a general $G$ to the case of a $G$ with connected center. One can always find a z-extension
\[
1\to G\to \tilde{G}\to D\to 1
\]
where $D$ is a torus and $\tilde{G}$ has connected center (by embedding the center $Z_G$ of $G$ into a torus, and taking a pushout). In that case,
\[
\Bun_G = \Bun_{\tilde{G}}\times_{\Bun_D} \ast
\]
and the corresponding functor
\[
D(\Bun_{\tilde{G}},\Lambda)\otimes_{D(\Bun_D,\Lambda)} D(\Lambda)\to D(\Bun_G,\Lambda)
\]
is an equivalence. Any parabolic $P$ lifts uniquely to a parabolic $\tilde{P}$ of $\tilde{G}$, the functor $\mathrm{Eis}_{\tilde{P}!}$ is $D(\Bun_D,\Lambda)$-linear and as it preserves compact objects by the case already established, its right adjoint is automatically colimit-preserving and also $D(\Bun_D,\Lambda)$-linear. (The $D(\Bun_D,\Lambda)$-linearity can either be verified by hand, or deduced from the second adjunction proved in the next section.) Then the same follows for $\mathrm{Eis}_{P!} = \mathrm{Eis}_{\tilde{P}!}\otimes_{D(\Bun_D,\Lambda)} D(\Lambda)$, as desired.
\end{proof}
In the preceding argument we used the following lemma, which is probably
well-known.
\begin{lemma}
If $M\subset G$ is a standard Levi and $b,b'\in B(M)$ are any elements
such that $b\prec^{M}b'$, then their images in $B(G)$ satisfy $b\prec^{G}b'$. 
\end{lemma}

\begin{proof}
This follows from the fact that the natural map on slope polygons is induced by the natural quotient map $\mathbb{X}_{*}(T_{\ol{E}})/W_{M_{\ol{E}}} \ra \mathbb{X}_{*}(T_{\ol{E}})/W_{G_{\ol{E}}}$, where $W_{G_{\ol{E}}}$ (resp. $W_{M_{\ol{E}}}$) denotes the absolute Weyl group of $G$ (resp. $M$). In particular, if two elements in $\mathbb{X}_{*}(T_{\ol{E}})/W_{M_{\ol{E}}}$ differ by a non-zero (resp. zero) linear combination of positive absolute roots of $M$ then the same is true for their image in $\mathbb{X}_{*}(T_{\ol{E}})/W_{G_{\ol{E}}}$. The claim follows.
\end{proof}

\subsection{Second adjointness}

In this section, we prove the following theorem. We stress that we are not assuming that $G$ is quasisplit (or fixing a choice of Borel for that matter).

\begin{theorem}\label{thm:secondadj} Let $P,P^-\subset G$ be a pair of opposite parabolics with Levi $M=P\cap P^-$.

\begin{enumerate}
    \item There is a natural isomorphism of functors \[\CT_{P^- \ast} \cong \CT_{P!}:D(\Bun_G,\Lambda) \to D(\Bun_M,\Lambda).\]
    \item There is a natural isomorphism of functors \[\mathbb{D}_{\mathrm{BZ}} \Eis_{P!} \cong \Eis_{P^- !} \mathbb{D}_{\mathrm{BZ}}^{M}:D(\Bun_M,\Lambda)^\omega \to D(\Bun_G,\Lambda)^{\omega}.\]
\end{enumerate}    
\end{theorem}
Note the restriction in part (2) to compact sheaves, in which case the statement is well-posed by the main result in the previous section.

Both parts of this theorem can be interpreted as geometrizations of Bernstein's famous second adjointness theorem, and, as mentioned in Remark \ref{rem: relationtoclassicalsecondadjointness}, our result specializes to classical second adjointness. However, this is completely circular, as we will use classical second adjointness in the proof, via our dependence on \cite{DHKM}.

In the argument below, we will first show that (1) and (2) are logically equivalent, in a precise sense. In an ideal world, we would then directly prove (1) as an application of hyperbolic localization, as in \cite{DrGaConstTerms}. However, we were unable to make such an argument work.\footnote{Perhaps this is inevitable. In some sense, there cannot be an argument for (1) which is completely formal, e.g. making use only of the fact that our sheaf category underlies a six-functor formalism with excision, since (1) fails with mod-$p$ coefficients.} Instead, we will prove (2) by an excision argument parallel to the proof that $\Eis_{P!}$ preserves compact objects given in the previous section.

We begin by constructing a canonical natural transformation
\[
\CT_{P^-\ast}\to \CT_{P !}.
\]
It is enough to construct a natural transformation
\[
\mathrm{id}\to \CT_{P!} \mathrm{Eis}_{P^-!}
\]
as then the desired transformation arises via precomposition with the functor $\CT_{P^-\ast}$ and postcomposition with the counit map $\mathrm{Eis}_{P^-!}\CT_{P^-\ast}\to \mathrm{id}$. Now $\CT_{P!} \mathrm{Eis}_{P^-!}$ is given by a kernel, namely the fiber product
\[
\Bun_{P^-}\times_{\Bun_G} \Bun_P\to \Bun_M\times \Bun_M,
\]
with the sheaf $\mathrm{IC}_{\Bun_{P^-}}\boxtimes \mathrm{IC}_{\Bun_P}$. This naturally contains the diagonal $\Bun_M$ as an open substack, yielding the desired map $\mathrm{id}\to \CT_{P!} \mathrm{Eis}_{P^-!}$ once observing that the sheaf $\mathrm{IC}_{\Bun_{P^-}}\boxtimes \mathrm{IC}_{\Bun_P}$ restricts to the trivial sheaf on $\Bun_M$.

Let us note some equivalent descriptions of the same map. First, one can also construct a natural transformation $t:\mathbb{D}_{\mathrm{BZ}} \Eis_{P!} \to \Eis_{P^- !} \mathbb{D}_{\mathrm{BZ}}^{M}$. For this we will make reference to the commutative diagram
\[
\xymatrix{\mathrm{Bun}_{M}\ar[dr]^{j}\ar[dddr]\ar[drrr]\\
 & \mathrm{Bun}_{P}\times_{\mathrm{Bun}_{G}}\mathrm{Bun}_{P^-}\ar[r]^{\;\;\;\;f}\ar[d]^{f^{-}} & \mathrm{Bun}_{P^-}\ar[d]^{\mathfrak{p}^{-}}\ar[r]_{\mathfrak{q}^{-}} & \mathrm{Bun}_{M}\ar[dd]^{\pi_{M}}\\
 & \mathrm{Bun}_{P}\ar[d]^{\mathfrak{q}}\ar[r]^{\mathfrak{p}} & \mathrm{Bun}_{G}\ar[dr]^{\pi_G}\\
 & \mathrm{Bun}_{M}\ar[rr]^{\pi_{M}} &  & \ast
}
\]
of Artin v-stacks that was already implicitly used above. Note that the inner square is Cartesian, $j$ is an open immersion, and the unlabelled slanted arrows are the identity on $\Bun_M$.

For compact sheaves $A,B$ on $\Bun_M$, we have a canonical isomorphism
\[R\Hom(\mathbb{D}_{\mathrm{BZ}} \Eis_{P!} A, \Eis_{P^- !}  B) \cong \pi_{G!}(\Eis_{P!} A \otimes \Eis_{P^- !} B) \]
by the isomorphism characterizing Bernstein-Zelevinsky duality on $\Bun_G$ (\cite[Proposition~VII.7.6]{FarguesScholze}, where we are implicitly using \cite[Proposition~1.1]{HI} to identify $\pi_{\natural}$ and $\pi_{!}$). We will now construct a canonical map 
\[ \gamma: \pi_{M!}( A \otimes  B) \to \pi_{G!}(\Eis_{P!} A \otimes \Eis_{P^- !} B)\]
functorial in $A$ and $B$. Using again the isomorphism characterizing Bernstein-Zelevinsky duality on $\Bun_M$, the source identifies canonically with $R\Hom(\mathbb{D}_{\mathrm{BZ}}^{M}A,  B)$,
so composing we get a functorial map
\[R\Hom(\mathbb{D}_{\mathrm{BZ}}^{M}A,  B) \to R\Hom(\mathbb{D}_{\mathrm{BZ}} \Eis_{P!} A, \Eis_{P^- !}  B). \]
When $B=\mathbb{D}_{\mathrm{BZ}}^{M} A$, the source contains the obvious identity map, and its image in ($H^0$ of) the target induces the desired natural transformation $t$.

To construct $\gamma$, the key observation is that, using the K\"unneth formula, there is a canonical isomorphism
\[ \pi_{G!}(\Eis_{P!} A \otimes \Eis_{P^- !}  B) \cong \pi_{X!}(f^{-\ast} \mathfrak{q}^\ast A \otimes f^{-\ast} \mathfrak{q}^{-\ast}  B\otimes C),\] where we set $X=\mathrm{Bun}_{P}\times_{\mathrm{Bun}_{G}}\mathrm{Bun}_{P^-}$ and $C= f^{-\ast} \mathrm{IC}_{\Bun_P} \otimes f^\ast \mathrm{IC}_{\Bun_{P^-}}$ for brevity. Since $j$ is an open immersion, the adjunction $j_! j^\ast \to \mathrm{id}$ induces a canonical map
\[\pi_{X!}j_! j^*(f^{-\ast} \mathfrak{q}^\ast A \otimes f^\ast \mathfrak{q}^{-\ast}  B \otimes C) \to \pi_{X!}(f^{-\ast} \mathfrak{q}^\ast A \otimes f^\ast \mathfrak{q}^{-\ast}  B \otimes C). \]
Using that $\pi_{X!}j_! = \pi_{M!}$, $j^*(f^{-\ast} \mathfrak{q}^\ast A \otimes f^\ast \mathfrak{q}^{-\ast}  B) = A \otimes  B$, and $j^\ast C = \Lambda$, the source identifies with $\pi_{M!}( A \otimes  B)$. This gives the desired map $\gamma$. 

Next we explain how the transformation $t$ canonically induces a natural transformation $s:\CT_{P^- \ast} \to \CT_{P !}$. Given $A \in D(\Bun_M,\Lambda)^\omega$ and $B\in D(\Bun_G,\Lambda)$, we have a natural isomorphism
\[ R\Hom(\Eis_{P^- !} \mathbb{D}_{\mathrm{BZ}}^M A,B) \cong R\Hom( \mathbb{D}_{\mathrm{BZ}}^M A,\CT_{P^- \ast} B)\]
by the evident adjunction between $\Eis_{P^- !}$ and $\CT_{P^- \ast}$. On the other hand, we have canonical isomorphisms
\begin{align*}
R\Hom(\mathbb{D}_{\mathrm{BZ}}\Eis_{P !}A,B) & \cong \pi_!(\Eis_{P !}A \otimes B) \\
& \cong \pi_!(\mathfrak{p}_!(\mathrm{IC}_{\Bun_P} \otimes \mathfrak{q}^\ast A) \otimes B) \\
& \cong (\pi \circ \mathfrak{p})_!(\mathrm{IC}_{\Bun_P} \otimes \mathfrak{q}^\ast A \otimes \mathfrak{p}^\ast B) \\
& \cong (\pi_M \circ \mathfrak{q})_!(\mathrm{IC}_{\Bun_P} \otimes \mathfrak{q}^\ast A \otimes \mathfrak{p}^\ast B) \\
& \cong \pi_{M!} (A \otimes \mathfrak{q}_{!}(\mathrm{IC}_{\Bun_P} \otimes \mathfrak{p}^\ast B)) \\
& \cong \pi_{M!} (A \otimes \CT_{P!}B) \\
& \cong 
R\Hom(\mathbb{D}_{\mathrm{BZ}}^M A, \CT_{P!}B). 
\end{align*}
Here the first and last lines follow from the isomorphisms characterizing Bernstein-Zelevinsky duality, and the remaining lines follow from unwinding definitions and liberally applying the projection formula. Combining these calculations, we see that the transformation $t$ induces a natural map
\begin{align*}
R\Hom( \mathbb{D}_{\mathrm{BZ}}^M A,\CT_{P^- \ast} B) & \cong R\Hom(\Eis_{P^- !} \mathbb{D}_{\mathrm{BZ}}^M A,B) \\
& \to R\Hom(\mathbb{D}_{\mathrm{BZ}}\Eis_{P !}A,B) \\
& \cong R\Hom(\mathbb{D}_{\mathrm{BZ}}^M A, \CT_{P!}B).
\end{align*}
By Yoneda, this induces the desired transformation $s$. One checks that this is equivalent to the map $\CT_{P^-\ast}\to \CT_{P!}$ constructed at the beginning. It is also clear that if $t$ is an equivalence, then $s$ is an equivalence, and conversely. In particular, to prove Theorem \ref{thm:secondadj}, it is enough to show that $t$ is an equivalence. 

We also need some basic compatibility properties of the transformation $t$. Since we will begin varying the parabolic, we write $t_P$ to indicate the dependence on the parabolic $P$ when necessary, and similarly for the map $\gamma$.

\begin{proposition}The transformation $t$ is compatible with composition of parabolics. More precisely, suppose given a parabolic $P_1$ with Levi quotient $M_1$, and let $P_2\subset M_1$ be a parabolic with Levi quotient $M_2$. Let $Q=P_1 \times_{M_1} P_2$ be the evident parabolic in $G$ with Levi quotient $M_2$. Then the composition 
\[\mathbb{D}_{\mathrm{BZ}}\Eis_{Q !} \cong \mathbb{D}_{\mathrm{BZ}}\Eis_{P_1 !} \Eis_{P_2 !} \to \Eis_{P_1 ^- !} \mathbb{D}_{\mathrm{BZ}}^{M_1}\Eis_{P_2 !} \to \Eis_{P_1 ^- !}\Eis_{P_2 ^- !}\mathbb{D}_{\mathrm{BZ}}^{M_2} \cong \Eis_{Q^- !}\mathbb{D}_{\mathrm{BZ}}^{M_2} \]
agrees with $t_Q$, where the first arrow is induced by $t_{P_1}$ and the second arrow is induced by $t_{P_2}$.
\end{proposition}

\begin{proof} After unwinding the definition of $t$ in terms of the map $\gamma$, this reduces to the claim that for all compact sheaves $A,B$ on $\Bun_{M_2}$, the composition
\begin{align*}
\pi_{M_2 !}(A \otimes B) \overset{\gamma_{P_2}}{\to} \pi_{M_1 !}(\Eis_{P_2 !}A \otimes \Eis_{P_2 ^- !}B) \overset{\gamma_{P_1}}{\to} \pi_{G !}(\Eis_{P_1 !} \Eis_{P_2 !}A \otimes \Eis_{P_1 ^- !}\Eis_{P_2 ^- !}B) \\
 \cong \pi_{G !}(\Eis_{Q !}A \otimes \Eis_{Q ^- !}B)
\end{align*}
agrees with $\gamma_{Q}$. But these maps are all induced by kernels, more precisely the open immersions
\[
\Bun_{M_2}\subset \Bun_{P_2^-}\times_{\Bun_{M_1}} \Bun_{P_2}\subset \Bun_{Q^-}\times_{\Bun_G} \Bun_Q.\qedhere
\]
\end{proof}

We also need that $t_P$ induces an isomorphism in the following seed cases. As all maps are induced by explicit kernels, these compatibilities are easy to check by checking them on kernels.

$\bullet$If $b\in B(M)$ is an element whose image in $B(G)$ is basic, then $\mathrm{Eis}_{P!}\circ i_{b!}^{M}\simeq i_{b!}^{G}\circ i_{P_{b}}^{G_{b}}$ resp.~$\mathrm{Eis}_{P^- !}\circ i_{b!}^{M}\simeq i_{b!}^{G}\circ i_{P_{b}^-}^{G_{b}}$,
and in this case precomposition of $t_P$ with $i_{b!}^M$ reduces to the map
\[\mathbb{D}_{\mathrm{coh}} i_{P_b}^{G_b} \to i_{P_{b}^-}^{G_b}\mathbb{D}_{\mathrm{coh}}^{M}  \]
on compact objects which is an isomorphism as an equivalent form of the second adjointness proved in \cite{DHKM}.

$\bullet$If $b \in B(M)$ is basic and strictly dominant, then $\mathrm{Eis}_{P!}\circ i_{b!}^{M}\simeq i_{b\sharp}^{\mathrm{ren}}$ and $\mathrm{Eis}_{P^- !}\circ i_{b!}^{M}\simeq i_{b!}^{\mathrm{ren}}$ by Corollary \ref{cor: dominantreductionConstantEisensteinDescription}~(5) and (1). In this case precomposition of $t_P$ with $i_{b!}^{M}$ reduces to the known isomorphism $\mathbb{D}_{\mathrm{BZ}}i_{b\sharp}^{\mathrm{ren}} \cong i_{b!}^{\mathrm{ren}}\mathbb{D}_{\mathrm{coh}}^{G_b}$. This follows from the computation of $\mathbb{D}_{\mathrm{BZ}}$ on compact generators performed in the proof of \cite[Theorem V.5.1]{FarguesScholze}.

\begin{proof}[Proof of Theorem~\ref{thm:secondadj}] Again, we assume first that $G$ is quasisplit, and fix $T\subset B\subset G$. We can assume that $M$ contains $T$, and so $P$ and $P^-$ are semistandard. By induction, we can assume that the theorem holds true for all proper Levi subgroups of $G$. We will prove the following claim by induction on quasicompact open substacks $U\subset \Bun_G$ (equivalently, finite subsets $|U|\subset |\Bun_G|=B(G)$ stable under generization).

\begin{claim} For any standard Levi $M\subset G$, any $b\in B(M)$ mapping into $|U|\subset B(G)$, and any pair of opposite parabolics $P$ and $P^-$ with Levi $M$, the natural transformation
\[
\mathbb D_{\mathrm{BZ}}\circ \mathrm{Eis}_{P!} i_{b!}^{\mathrm{ren},M}\to \mathrm{Eis}_{P^- !} \mathbb D_{\mathrm{BZ}}^M  i_{b!}^{\mathrm{ren},M}
\]
is an equivalence on compact objects.
\end{claim}

We follow the same steps as in the proof of Theorem~\ref{thm:Eiscompact}. By reducing $b$ further to a basic element of a Levi subgroup of $M$, and using compatibility of Eisenstein functors (and the transformation $t$) with composition, we can reduce to the case that $b=b_\alpha\in B(M)$ is basic. Then there is the Levi subgroup $G_{\{b\}}\subset G$ and the element $w_\alpha\in W[M,G_{\{b\}}]$. We have $M\subset w_\alpha^{-1}(G_{\{b\}})$, and the image of $b$ in $B(w_\alpha^{-1}(G_{\{b\}}))$ stays basic. We can then use compatibility of Eisenstein functors (and the transformation $t$) with composition again to reduce to the case $M=w_\alpha^{-1}(G_{\{b\}})$. As $P$ and $P^-$ were arbitrary opposite parabolics with Levi $M$, we can undo the Weyl group twist. In other words, we can assume that $M=G_{\{b\}}$. We have to show that for all parabolics $P$ with Levi $M=G_{\{b\}}$ and opposite parabolic $P^-$, the transformation
\[
\mathbb D_{\mathrm{BZ}}\circ \mathrm{Eis}_{P!}  i_{b!}^M\to \mathrm{Eis}_{P^- !} \mathbb D_{\mathrm{BZ}}^M  i_{b!}^M
\]
is an equivalence on compact objects. By our seed case, we know that this is true when $P$ is the standard parabolic.

We will now show that the result is true for some parabolic $P$ with Levi $M$ if and only if it is true for any other $P'$. Moving from one parabolic with Levi $M$ to any other via a sequence of minimal modifications, we can assume that $P$ and $P'$ are contained in some parabolic $Q$ whose Levi $N\supset M$ has the property that $P\cap N$ and $P'\cap N$ are maximal parabolic subgroups. In particular, one of them is standard and the other is anti-standard in $N$; without loss of generality $P\cap N$ is the standard one. As in the proof of Theorem~\ref{thm:Eiscompact}, we have
\[
\mathrm{Eis}_{P\cap N!} i_{b!}^M = i_{b\sharp}^{\mathrm{ren},N}
\]
and
\[
\mathrm{Eis}_{P'\cap N!} i_{b!}^M = i_{b!}^{\mathrm{ren},N}
\]
and the triangle
\[
i_{b\sharp}^{\mathrm{ren},N}\to i_{b!}^{\mathrm{ren},N}\to C
\]
where the cone $C$ is supported at elements strictly more semistable than $b$. By induction, we know that the natural transformation
\[
\mathbb D_{\mathrm{BZ}} \mathrm{Eis}_{Q!}\to \mathrm{Eis}_{Q^-!} \mathbb D_{\mathrm{BZ}}^N
\]
is an equivalence on $C$. Thus, it is an equivalence on $i_{b!}^{\mathrm{ren},N}$ if and only if it is an equivalence on $i_{b\sharp}^{\mathrm{ren},N}$.

Now this transformation applied to $i_{b\sharp}^{\mathrm{ren},N} = \mathrm{Eis}_{P\cap N!} i_{b!}^M$ yields the transformation
\[
\mathbb D_{\mathrm{BZ}}\mathrm{Eis}_{P!} i_{b!}^M\to \mathrm{Eis}_{P^-!} \mathbb D_{\mathrm{BZ}}^M i_{b!}^M
\]
again by compatibility with composition (and as the result is known for $N$, as $P\cap N$ is standard). Similarly, the transformation applied to $i_{b!}^{\mathrm{ren},N} = \mathrm{Eis}_{P'\cap N!} i_{b!}^M$ yields the transformation
\[
\mathbb D_{\mathrm{BZ}}\mathrm{Eis}_{P'!} i_{b!}^M\to \mathrm{Eis}_{P^{\prime -}!} \mathbb D_{\mathrm{BZ}}^M i_{b!}^M.
\]
Thus, the transformation for $P$ is an isomorphism if and only if the transformation for $P'$ is an isomorphism, as desired.

This finishes the case that $G$ is quasisplit. If $G$ has connected center, then so does $M$, and there is some basic element $b\in B(M)$ such that $M_b$ (and hence $G_b$) is quasisplit. One can then twist $P$ and $P'$ by $b$ to obtain opposite parabolics of $G_b$ with Levi $M_b$, and this reduces the case of $G$ to the case of $G_b$.

It remains to reduce the case of a general $G$ to the case of a $G$ with connected center. Again, we pick a z-extension
\[
1\to G\to \tilde{G}\to D\to 1
\]
where $D$ is a torus and $\tilde{G}$ has connected center, so
\[
\Bun_G = \Bun_{\tilde{G}}\times_{\Bun_D} \ast
\]
and the corresponding functor
\[
D(\Bun_{\tilde{G}},\Lambda)\otimes_{D(\Bun_D,\Lambda)} D(\Lambda)\to D(\Bun_G,\Lambda)
\]
is an equivalence. The parabolics $P$ and $P^-$ lift to parabolics $\tilde{P}$ and $\tilde{P}^-$, and the isomorphism
\[
\CT_{\tilde{P}^-\ast}\to \CT_{\tilde{P}!}
\]
shows that $\CT_{\tilde{P}^-\ast}$ is $D(\Bun_D,\Lambda)$-linear. In particular, $\mathrm{Eis}_{\tilde{P}^-!}$ has a $D(\Bun_D,\Lambda)$-linear right adjoint, which then base changes along $D(\Bun_D,\Lambda)\to D(\Lambda)$ to give the right adjoint of $\mathrm{Eis}_{P^-!} = \mathrm{Eis}_{\tilde{P}^-!}\otimes_{D(\Bun_D,\Lambda)} D(\Lambda)$. This gives the isomorphism $\CT_{P^-\ast}\cong \CT_{P!}$.
\end{proof}

\section{The geometry of $\wtBun_P$}

In the last section we showed that the geometric Eisenstein functor $\Eis_{P!}(-)$ preserves compact objects, establishing the first group of finiteness results claimed in Theorem \ref{thm:maintheorem}. To make further progress, we will now need to address how the functor interacts with the ULA objects. This will be accomplished by using the geometry of a certain relative compactification of the map $\mf{p}_{P}$ known in the classical geometric Langlands program as the Drinfeld compactification. In this section, we will establish some of the fundamental geometric properties of these compactifications.

The definition of the Drinfeld compactifications classically involves certain injections of vector bundles whose cokernel is flat in a suitable sense, where many of its basic properties in turn follow from the classical theory of Quot schemes. We now study the analogue of this in the Fargues-Fontaine setting. Several of the results can also be found in \cite[\S~5.2.3]{HamGeomES}, but we recall the proofs here for convenience.
\subsection{Flat coherent sheaves on relative Fargues-Fontaine curves}

Let $S \in \Perf$, with $X_{S}$ the associated relative curve. If $\mathcal{E}$ is an $\mathcal{O}_{X_S}$-module and $T\to S$ is any map of perfectoid spaces, we write $\calE_T$ for the pullback of $\calE$ along the map of ringed spaces $X_T \to X_S$. The following key notion is due to Ansch\"utz-le Bras \cite[Definition 2.8]{AL1}.

\begin{definition}A (relatively) flat coherent sheaf on $X_S$ is an $\mathcal{O}_{X_S}$-module $\mathcal{E}$ which, locally in the analytic topology on $S$, can be written as the cokernel of a fiberwise-injective map of bundles $\mathcal{E}_1 \to \mathcal{E}_2$.
\end{definition}
For some explanation of the name, see \cite[Remark 2.9]{AL1}. 

We next record some foundational stability properties of this notion. Part (3) of the next proposition is \cite[Theorem 2.11]{AL1}. Part (2) is \cite[Lemma~5.14]{HamGeomES}.

\begin{proposition}\label{prop:flat-coherent-basic-stabilities}
\begin{enumerate}
    \item If $\mathcal{F}$ is a flat coherent sheaf on $X_S$ and $f:\mathcal{E}\twoheadrightarrow \mathcal{F}$ is any surjection from a vector bundle, then $\ker f$ is a vector bundle and the natural map $\ker f \to \mathcal{E}$ is fiberwise injective.

    \item Flat coherent sheaves are stable under base change on $S$, kernels of surjections, cokernels of fiberwise-injective maps, and tensoring by vector bundles.

    \item The functor sending $S$ to the groupoid of flat coherent sheaves on $X_S$ is a stack for the v-topology.
\end{enumerate}
\end{proposition}
\begin{proof}
(1) is local on $|X_S|$, where it reduces to the following standard fact from commutative algebra: if $f: A^n \twoheadrightarrow M$ is any surjection from a finite free $A$-module onto a finitely presented module of projective dimension $\leq 1$, then $\ker f$ is finitely generated and projective.

For (2), the statements about tensoring by vector bundles is clear. For the statement about base change, use that a map of vector bundles on a relative curve $X_S$ is fiberwise-injective iff it remains injective after any base change on $S$. For stability under kernels of surjections, let $f:\mathcal{F}_1 \twoheadrightarrow \mathcal{F}_2$ be a surjection of flat coherent sheaves. Arguing locally on $S$, we can pick a surjection $s: \mathcal{E} \to \mathcal{F}_1$ from a vector bundle. Let $\mathcal{K}_1$ resp. $\mathcal{K}_2$ denote the kernel of $s$ resp. the kernel of the surjection $f \circ s$. Then $\mathcal{K}_1$ and $\mathcal{K}_2$ are vector bundles, by two applications of (1), and the snake lemma gives a short exact sequence
\[0 \to \mathcal{K}_1 \to \mathcal{K}_2 \to \ker f \to 0. \]
The first arrow here is fiberwise-injective, since it is compatible with the natural fiberwise-injective maps from the $\mathcal{K}_i$'s into $\mathcal{E}$. Therefore $\ker f$ is a flat coherent sheaf. 

The statement about cokernels of fiberwise-injective maps is similar and left to the reader.
\end{proof}

If $\mathcal{E}/X_S$ is a flat coherent sheaf then let $\calH^0(\calE)$ be the functor on perfectoid spaces over $S$ sending $T/S$ to $H^0(X_T,\calE_T)$. Likewise, for $\calF$ a second flat coherent sheaf, let $\calH om(\calE,\calF)$ be the functor sending $T/S$ to $\Hom_{X_T}(\calE_T,\calF_T)$.
\begin{proposition}\label{prop:flat-coherent-homs-nice}The functors $\calH^0(\calE)$ and $\calH om(\calE,\calF)$ are representable by locally spatial diamonds which are partially proper over $S$. Moreover, after deleting the zero section and quotienting by $E^{\times}$, they become proper over $S$.
\end{proposition}
\begin{proof}We can assume $S$ is quasicompact. Twisting $\calE$ by $\mathcal{O}(n)$ for some large $n$ and using \cite[Corollary 2.15]{AL1}, we can find a short exact sequence $0\to \calE_1 \to \calE_2 \to \calE \to 0$ where $\calE_i$ are vector bundles with strictly negative slopes. This gives a presentation $\calH^0(\calE)\cong \calH^1(\calE_1)\times_{\calH^1(\calE_2),s} S$ where $s: S\to \calH^1(\calE_2)$ is the zero section. This proves the result for $\calH^0$. For $\calH om$, the result immediately follows if $\calE$ is a vector bundle, since in this case $\calF \otimes \calE^\vee$ is flat coherent and $\calH om(\calE,\calF)\cong \calH^0(\calF \otimes \calE^\vee)$. In general, pick any presentation $\calE_1 \to \calE_2 \to \calE \to 0$ with $\calE_i$ vector bundles. Then
\[ \calH om(\calE,\calF)\cong \calH om(\calE_2,\calF)\times_{\calH om(\calE_1,\calF),s}S \]
where again $s$ is the zero section.

The final claim reduces analogously to the case of $\calH^0(\calE)$ with $\calE$ a vector bundle, which is proved in \cite{FarguesScholze}.
\end{proof}

Suppose $S=\Spa(K,K^+)$ is a point with $K$ algebraically closed. If $\calF$ is a flat coherent sheaf on $X_S$, there is a canonical short exact sequence $0 \to \calF^{\mathrm{tors}} \to \calF \to \calF^{\mathrm{vb}} \to 0$, where the first term is a torsion sheaf and the third term is a vector bundle. In particular, the length of $\calF^{\mathrm{tors}}$ is a well-defined invariant of $\calF$. For general $S$, we therefore get a function $\lambda_\calF:|S| \to \mathbf{Z}_{\geq 0}$ sending $s \in |S|$ to $\mathrm{length}(\calF_{\tilde{s}}^{\mathrm{tors}})$, where $\tilde{s}:\Spa(K,K^+)\to S$ is any choice of map with $K$ algebraically closed and such that $|\tilde{s}|$ sends the closed point in $|\Spa(K,K^+)|$ to $s$.

\begin{lemma}\label{lem:torsion-semicontinuous-stuff}{\label{lemma: uppersemicont}} Let $\calF/X_S$ be a flat coherent sheaf.
\begin{enumerate}
    \item The function $\lambda_\calF$ is upper semicontinuous.

    \item If $\lambda_\calF$ is locally constant, there is a unique short exact sequence
of flat coherent sheaves $0\to\mathcal{F}^{\mathrm{tors}}\to\mathcal{F}\to\mathcal{F}^{\mathrm{vb}}\to0$
where $\mathcal{F}^{\mathrm{tors}}$ is a torsion sheaf and $\mathcal{F}^{\mathrm{vb}}$
is a vector bundle.
\end{enumerate}
\end{lemma}

For the proof, we will need the following notion.

\begin{definition}
Let $\mathcal{E}/X_{S}$ be a flat coherent sheaf. A section
$s\in H^0(X_S,\mathcal{E})$ is \emph{simple} if it is killed by $\mathcal{I}_{D}$,
where $\mathcal{I}_{D}\subset\mathcal{O}_{X_{S}}$ is the ideal sheaf of some degree one relative
Cartier divisor $D\subset X_{S}$.
\end{definition}

Note that giving a fiberwise-nonzero simple section $s \in \mathcal{E}$
is the same as giving a pair $(D,f)$ where $D\subset X_{S}$ is a
degree one relative Cartier divisor and $f:\mathcal{O}_{X_S}/\mathcal{I}_{D}\to\mathcal{E}$
is a fiberwise-nonzero (and hence injective) map: one sends the map $f$ to the section $s=f(1)$. In particular, if
$s\in\mathcal{E}$ is a fiberwise-nonzero simple section, then $\mathcal{E}/(\mathcal{O}_{X_S}\cdot s)$ is also a flat coherent sheaf, because it is the cokernel of a fiberwise-injective map of flat coherent sheaves.

\begin{lemma}\label{lem:nonzero-simple-lemma}
The moduli space of fiberwise-nonzero simple sections $\mathcal{S}\mathrm{imp}(\mathcal{E})$
is a locally spatial diamond, partially proper over $S$. Moreover,
its projectivization $\mathcal{S}\mathrm{imp}(\mathcal{E})/E^{\times}\to S$
is proper and representable in spatial diamonds. In particular, $\mathcal{S}\mathrm{imp}(\mathcal{E})$
is a v-cover of its (closed and generalizing) image in $S$.    
\end{lemma}
\begin{proof}
The map $\mathcal{S}\mathrm{imp}(\mathcal{E})\to S$ factors as $\mathcal{S}\mathrm{imp}(\mathcal{E})\to\mathrm{Div}^{1}\times S\to S$,
where the first arrow is a fibration in spaces of the form $\mathcal{H}\mathrm{om}(\mathcal{O}/\mathcal{I}_{D},\mathcal{E})\smallsetminus\{0\}$.
Since the second arrow is proper, surjective, and representable in
spatial diamonds by \cite[Proposition~II.1.21]{FarguesScholze}, all claims now follow from Proposition \ref{prop:flat-coherent-homs-nice}.    
\end{proof}

\begin{proof}[Proof of Lemma \ref{lem:torsion-semicontinuous-stuff}]

For (1), the claim is v-local on $S$. It is clear that
the locus where $\lambda_\calF>0$
coincides with the image of $|\mathcal{S}\mathrm{imp}(\mathcal{F})|\to S$,
and this image is closed by the previous lemma. Replacing $S$ by
a v-cover of this image, we can pick a fiberwise-nonzero simple section
$s\in\mathcal{F}$ by the previous lemma. Then $\mathcal{F}'=\mathcal{F}/\mathcal{O}_{X_{S}}\cdot s$
is flat coherent, and the locus where $\lambda_\calF >1$
coincides with the locus where $\lambda_{\calF'}>0$. Now continue by induction.

For (2), we can clearly assume $S$ is quasicompact and
$\lambda_\calF$
is constant. Uniqueness is clear, so for existence we can work v-locally
on $S$ by Proposition \ref{prop:flat-coherent-basic-stabilities}.(3). Now we
argue by induction on the constant $\lambda_\calF$
as in the proof of (1), by passing to a v-cover on $S$, picking a
fiberwise-nonzero simple section $s$ and replacing $\mathcal{F}$
by the quotient by the subsheaf generated by $s$.
\end{proof}

Let $\mathcal{S}\mathrm{ub}_{\mathcal{E}}^{n,d}$
be the functor sending $T/S$ to the set of subsheaves $\mathcal{F}\subset\mathcal{E}_{T}$
with $\mathcal{F}/X_{T}$ flat coherent of (constant) fiberwise generic
rank $n$ and degree $d$, and such that $\mathcal{F}\to\mathcal{E}_{T}$
remains injective after any base change on $T$. Note that if $\mathcal{E}$
is a vector bundle, then any such $\mathcal{F}$ is also a vector
bundle.
\begin{theorem}{\label{thm: Subislocallyspatial}}
The functor $\mathcal{S}\mathrm{ub}_{\mathcal{E}}^{n,d}$ is a locally
spatial diamond, and $\mathcal{S}\mathrm{ub}_{\mathcal{E}}^{n,d}\to S$
is partially proper of locally finite dim.trg.
\end{theorem}

\begin{proof}
First we show that $\mathcal{S}\mathrm{ub}_{\mathcal{E}}^{n,d}\to S$
is separated. Let $T\to\mathcal{S}\mathrm{ub}_{\mathcal{E}}^{n,d}\times_{S}\mathcal{S}\mathrm{ub}_{\mathcal{E}}^{n,d}$
be any map, corresponding to a pair $\mathcal{F}_{i}\subset\mathcal{E}_{T}$
for $i=1,2$. We want to see that the fiber product $W=T\times_{\mathcal{S}\mathrm{ub}_{\mathcal{E}}^{n,d}\times_{S}\mathcal{S}\mathrm{ub}_{\mathcal{E}}^{n,d},\Delta}\mathcal{S}\mathrm{ub}_{\mathcal{E}}^{n,d}\subset T$
is closed. From the given datum, we get maps $\mathcal{F}_{1}\to\mathcal{E}_{T}/\mathcal{F}_{2}$
and $\mathcal{F}_{2}\to\mathcal{E}_{T}/\mathcal{F}_{1}$, and we want
to cut out the locus where these maps are both fiberwise identically
zero. We can interpret these maps as a morphism $s:T\to\mathcal{H}\mathrm{om}(\mathcal{F}_{1}\oplus\mathcal{F}_{2},\mathcal{E}_{T}/\mathcal{F}_{2}\oplus\mathcal{E}_{T}/\mathcal{F}_{1})$,
which is a section of the structure map $\mathcal{H}\mathrm{om}(\mathcal{F}_{1}\oplus\mathcal{F}_{2},\mathcal{E}_{T}/\mathcal{F}_{2}\oplus\mathcal{E}_{T}/\mathcal{F}_{1})\to T$.
The fiber product we care about can be presented as
\[
T\times_{s,\mathcal{H}\mathrm{om}(\mathcal{F}_{1}\oplus\mathcal{F}_{2},\mathcal{E}_{T}/\mathcal{F}_{2}\oplus\mathcal{E}_{T}/\mathcal{F}_{1}),e}T,
\]
where $e:T\to\mathcal{H}\mathrm{om}(\mathcal{F}_{1}\oplus\mathcal{F}_{2},\mathcal{E}_{T}/\mathcal{F}_{2}\oplus\mathcal{E}_{T}/\mathcal{F}_{1})$
is the zero section. By Proposition \ref{prop:flat-coherent-basic-stabilities} (2) above, $\mathcal{E}_{T}/\mathcal{F}_{2}\oplus\mathcal{E}_{T}/\mathcal{F}_{1}$
is a flat coherent sheaf, so by the partial properness result of Proposition \ref{prop:flat-coherent-homs-nice} we get that $e$ is a closed immersion. Thus
$W\subset T$ is a closed immersion.

Next, we reduce the claim that $\mathcal{S}\mathrm{ub}_{\mathcal{E}}^{n,d}$
is a locally spatial diamond to the case where $\mathcal{E}$ is a
vector bundle. For this, pick a presentation
\[
0\to\mathcal{E}_{1}\overset{i}{\to}\mathcal{E}_{2}\to\mathcal{E}\to0
\]
where $i:\mathcal{E}_{1}\to\mathcal{E}_{2}$ is a fiberwise-injective
map of vector bundles. We claim that $\mathcal{S}\mathrm{ub}_{\mathcal{E}}^{n,d}$
is naturally a closed subfunctor of $\mathcal{S}\mathrm{ub}_{\mathcal{E}_{2}}^{n+\mathrm{rk}(\mathcal{E}_{1}),d+\mathrm{deg}(\mathcal{E}_{1})}$.
It is clear that it is a subfunctor, since by pullback along $\mathcal{E}_{2}\twoheadrightarrow\mathcal{E}$,
flat coherent subsheaves $\mathcal{F}\subset\mathcal{E}$ correspond to
flat coherent subsheaves $\mathcal{F}'\subset\mathcal{E}_{2}$ containing
$\mathcal{E}_{1}$, and it is easy to see that the ranks and degrees
work as indicated. For closedness, let $X$ be the functor parametrizing
$\mathcal{F}'\subset\mathcal{E}_{2}$ of the indicated rank and degree
together with an \emph{arbitrary} map $\mathcal{E}_{1}\to\mathcal{E}_{2}/\mathcal{F}'$.
Forgetting the second piece of data gives a map
\[
X\to\mathcal{S}\mathrm{ub}_{\mathcal{E}_{2}}^{n+\mathrm{rk}(\mathcal{E}_{1}),d+\mathrm{deg}(\mathcal{E}_{1})},
\]
which is a fibration in spaces of the form $\mathcal{H}^{0}(\mathcal{E}_{1}^{\vee}\otimes(\mathcal{E}_{2}/\mathcal{F}'))$.
In particular, there is a natural zero section $e:\mathcal{S}\mathrm{ub}_{\mathcal{E}_{2}}^{n+\mathrm{rk}(\mathcal{E}_{1}),d+\mathrm{deg}(\mathcal{E}_{1})}\to X$,
which is a closed immersion. On the other hand, given any $\mathcal{F}'\subset\mathcal{E}_{2}$,
composing $i$ with the quotient map $\mathcal{E}_{2}\to\mathcal{E}_{2}/\mathcal{F}'$
gives rise to a map $\mathcal{E}_{1}\to\mathcal{E}_{2}/\mathcal{F}'$.
In other words, the additional datum of $i$ induces a second section
$s:\mathcal{S}\mathrm{ub}_{\mathcal{E}_{2}}^{n+\mathrm{rk}(\mathcal{E}_{1}),d+\mathrm{deg}(\mathcal{E}_{1})}\to X$.
It is now clear from this discussion that the square
\[
\xymatrix{\mathcal{S}\mathrm{ub}_{\mathcal{E}}^{n,d}\ar[d]\ar[r] & \mathcal{S}\mathrm{ub}_{\mathcal{E}_{2}}^{n+\mathrm{rk}(\mathcal{E}_{1}),d+\mathrm{deg}(\mathcal{E}_{1})}\ar[d]^{e}\\
\mathcal{S}\mathrm{ub}_{\mathcal{E}_{2}}^{n+\mathrm{rk}(\mathcal{E}_{1}),d+\mathrm{deg}(\mathcal{E}_{1})}\ar[r]^{s} & X
}
\]
is Cartesian, and since $e$ is a closed immersion we get the desired
result.

Thus, in proving that $\mathcal{S}\mathrm{ub}_{\mathcal{E}}^{n,d}$
is a locally spatial diamond, we may assume that $\mathcal{E}$ is
a vector bundle. We impose this assumption from now on. Next we check
that $\mathcal{S}\mathrm{ub}_{\mathcal{E}}^{n,d}$ is a locally spatial
v-sheaf. Let
\[
\mathcal{S}\mathrm{ub}_{\mathcal{E}}^{n,d}\to\mathrm{Bun}_{n}\times S
\]
be the map sending $\mathcal{F}\subset\mathcal{E}$ to $\mathcal{F}$.
This map is a fibration in open subspaces of the Banach-Colmez space $\mathcal{H}^{0}(\mathcal{F}^{\vee}\otimes\mathcal{E})$, 
so is representable in locally spatial diamonds. Now let $U\to\mathrm{Bun}_{n}\times S$
be a surjective cohomologically smooth map from a locally spatial
diamond. We then get that
\[
V=U\times_{\mathrm{Bun}_{n}\times S}\mathcal{S}\mathrm{ub}_{\mathcal{E}}^{n,d}
\]
is a locally spatial diamond, and the induced map $V\to\mathcal{S}\mathrm{ub}_{\mathcal{E}}^{n,d}$
is surjective and cohomologically smooth. This implies that $\mathcal{S}\mathrm{ub}_{\mathcal{E}}^{n,d}$
is a locally spatial v-sheaf, using also that $\mathcal{S}\mathrm{ub}_{\mathcal{E}}^{n,d}$
is separated over $S$ (which we already proved).

It remains to see that $\mathcal{S}\mathrm{ub}_{\mathcal{E}}^{n,d}$ is actually a locally spatial diamond. By the Artin criterion \cite[Theorem 12.18]{Ecod}, it suffices to stratify $\mathcal{S}\mathrm{ub}_{\mathcal{E}}^{n,d}$
into locally closed subfunctors, each of which is a locally spatial
diamond. We do this as follows. For any integer $t\geq0$, let $\mathcal{S}\mathrm{ub}_{\mathcal{E}}^{n,d;t}$
be the subfunctor where, after pullback to any geometric point, the
torsion subsheaf of $\mathcal{E}/\mathcal{F}$ has length exactly
$t$. By Lemma \ref{lem:torsion-semicontinuous-stuff}.(1), each of these is a locally closed subfunctor. By
Lemma \ref{lem:torsion-semicontinuous-stuff}.(2), on each of these subfunctors we can globally saturate
the bundle $\mathcal{F}$ to a subbundle $\mathcal{F}'\subset\mathcal{E}$
such that $\mathcal{E}/\mathcal{F}'$ is also a vector bundle. Said
differently, a point of $\mathcal{S}\mathrm{ub}_{\mathcal{E}}^{n,d;t}$
corresponds to a saturated subbundle $\mathcal{F}'\subset\mathcal{E}$
of rank $n$ and degree $d+t$, together with a length $t$ modification
$\mathcal{F}\subset\mathcal{F}'$. Let $\mathcal{S}\mathrm{ub}_{\mathcal{E},+}^{n,d+t}\subset\mathcal{S}\mathrm{ub}_{\mathcal{E}}^{n,d+t}$
be the open subfunctor parametrizing \emph{saturated} subbundles $\mathcal{F}'\subset\mathcal{E}$
of the indicated rank and degree (here openness again follows from Lemma \ref{lem:torsion-semicontinuous-stuff}.(1)). Then
sending $\{\mathcal{F}\subset\mathcal{E}\}\in\mathcal{S}\mathrm{ub}_{\mathcal{E}}^{n,d;t}$
to $\mathcal{F}'\subset\mathcal{E}$, where $\mathcal{F}'$ is the
saturation of $\mathcal{F}$ in $\mathcal{E}$, we get a surjective
map
\[
\pi_t:\mathcal{S}\mathrm{ub}_{\mathcal{E}}^{n,d;t}\to\mathcal{S}\mathrm{ub}_{\mathcal{E},+}^{n,d+t}
\]
whose fiber over a given saturated $\mathcal{F}'$ is the space of
length $t$ modifications $\mathcal{F}\subset\mathcal{F}'$. It is
straightforward to check that $\pi_t$ is proper and representable in
spatial diamonds, by adapting some arguments from Chapter VI of \cite{FarguesScholze}. More precisely, one proves that it is a fibration in closed subspaces of a twisted form of the Grassmannian $\mathrm{Gr}_{\mathrm{GL}_{n},\mathrm{Div}_{X}^{t}}$.
This reduces us to showing that $\mathcal{S}\mathrm{ub}_{\mathcal{E},+}^{n,d+t}$
is a locally spatial diamond, which follows immediately from \cite[Example
IV.4.7]{FarguesScholze}.
\end{proof}

The following result is the technical key in proving that (any connected component of) the Drinfeld compactification $\wtBun_P$ is proper over $\Bun_G$. 

\begin{theorem}{\label{thm: Subisquasicompact}}
The map $\mathcal{S}\mathrm{ub}_{\mathcal{E}}^{n,d}\to S$ is quasicompact.
\end{theorem}
\begin{proof}Arguing as in the previous proof, we can assume $\mathcal{E}$ is a vector bundle, so $\mathcal{S}\mathrm{ub}_{\mathcal{E}}^{n,d}$ parametrizes fiberwise-injective maps $\mathcal{F} \to \mathcal{E}$ where $\mathcal{F}$ is a vector bundle of (constant) rank $n$ and degree $d$. Sending such a map to the induced map $\wedge^n \calF \to \wedge^n \calE$ defines a map
\[\alpha: \mathcal{S}\mathrm{ub}_{\mathcal{E}}^{n,d} \to \mathcal{S}\mathrm{ub}_{\wedge^n \mathcal{E}}^{1,d}. \]
In general, it is easy to see that $\mathcal{S}\mathrm{ub}_{\mathcal{E}}^{1,d}$ is quasicompact over $S$, by presenting it as $(\calH^0(\calE(-d)) \smallsetminus 0)/\underline{E^\times}$ and invoking \cite[Proposition~II.2.16]{FarguesScholze}. This reduces us to showing that $\alpha$ is quasicompact.

It suffices to see that for any map $\Spa(A,A^+)\to \mathcal{S}\mathrm{ub}_{\wedge^n\mathcal E}^{1,d}$ with a lift to $\mathcal{S}\mathrm{ub}_{\mathcal{E}}^{n,d}$ at each point, we get a (unique) lift to $\mathcal{S}\mathrm{ub}_{\mathcal{E}}^{n,d}$, cf.~Lemma~\ref{lem:quasicompactnesscriterion} below.

This statement can be checked on the cover of the Fargues--Fontaine curve by $Y_{\Spa(A,A^{+}),[1,q]}$ with notation as in \cite[Section~II.2.1]{FarguesScholze}. Arguing locally on $\Spa(A,A^+)$, we can then assume that $\mathcal E|_{Y_{\Spa(A,A^{+}),[1,q]}}$ is free. By v-descent, we are reduced to Lemma~\ref{lem:properness} below.
\end{proof}

\begin{lemma}\label{lem:quasicompactnesscriterion} Let $f: X\to Y$ be a map of small v-stacks. Assume that for some set $I$ we are given a collection $(K_i,K_i^+)$ of algebraically closed affinoid perfectoid fields of characteristic $p$ with pseudo-uniformizers $\varpi_i\in K_i^+$; then let $A^+=\prod_i K_i^+$ be the product with pseudo-uniformizer $\varpi=(\varpi_i)_i$ and $A=A^+[\tfrac 1{\varpi}]$.

Assume that for all $I$ and $(K_i,K_i^+)$ as above, the map
\[
X(A,A^+)\to Y(A,A^+)\times_{\prod_i Y(K_i,K_i^+)} \prod_i X(K_i,K_i^+)
\]
is an isomorphism. Then $f$ is quasicompact and quasiseparated.

If $f$ is representable in locally spatial diamonds, then the converse holds true: In other words, if $f$ is representable in spatial diamonds, then for all $I$ and $(K_i,K_i^+)$ as above, the map
\[
X(A,A^+)\to Y(A,A^+)\times_{\prod_i Y(K_i,K_i^+)} \prod_i X(K_i,K_i^+)
\]
is an isomorphism.
\end{lemma}

\begin{proof} The assumption is stable under base change. We can thus assume that $Y$ is affinoid perfectoid. Assume first that $X$ is $0$-truncated and quasiseparated (which is enough for our application). Choose some set $I$ and points $\Spa(K_i,K_i^+)\to X$ covering the isomorphism classes of all geometric points of $X$; endow them with a choice of pseudouniformizers arising from the composite $\Spa(K_i,K_i^+)\to X\to Y$ and a choice of pseudouniformizer on $Y$. Then the maps $\Spa(K_i,K_i^+)\to Y$ uniquely assemble to a map $\Spa(A,A^+)\to Y$ (as $Y$ is affinoid perfectoid). By assumption, they lift to a unique map $\Spa(A,A^+)\to X$. Thus, there is some map $\Spa(A,A^+)\to X$ that is surjective on topological spaces. But as $X$ is quasiseparated, the map $\Spa(A,A^+)\to X$ is quasicompact and so this implies that $\Spa(A,A^+)\to X$ is surjective as a map of v-sheaves by \cite[Lemma 12.11]{Ecod}.

In general, the assumptions of the lemma are stable under passing to the diagonal map. Now $\Delta_f$ is always $0$-truncated and $\Delta_{\Delta_f}$ is $0$-truncated and quasiseparated. Thus, the above implies that $\Delta_{\Delta_f}$ is quasicompact, and hence $\Delta_f$ is $0$-truncated and quasiseparated. Then the above implies that $\Delta_f$ is $0$-truncated and quasicompact. Thus, $f$ is quasiseparated, and then the argument above applies again to show that $f$ is also quasicompact.

For the converse, we can also first pass to diagonals. Note that if the result is true for the diagonal, then the map
\[
X(A,A^+)\to Y(A,A^+)\times_{\prod_i Y(K_i,K_i^+)} \prod_i X(K_i,K_i^+)
\]
is always an injection, so one only has to prove surjectivity. Given any element of the right-hand side, we can replace $Y$ by $\Spa(A,A^+)$. Now $X$ is a spatial diamond, so admits a pro-\'etale surjection from an affinoid perfectoid space $X'\to X$. All maps $\Spa(K_i,K_i^+)\to X$ then lift to $X'$, and then assemble to a map $\Spa(A,A^+)\to X'\to X$ lifting the given points; this shows surjectivity.
\end{proof}

\begin{lemma}\label{lem:properness} Let $I$ be a set and for $i\in I$ let $(K_i,K_i^+)$ be an algebraically closed affinoid perfectoid field of characteristic $p$ with pseudo-uniformizer $\varpi_i\in K_i^+$. Let $A^+=\prod_i K_i^+$ be the product with pseudo-uniformizer $\varpi=(\varpi_i)_i$ and $A=A^+[\tfrac 1{\varpi}]$. Let $\mathcal R_{A,[1,q]} = \mathcal O(Y_{\Spa(A,A^+),[1,q]})$ and $\mathcal R_{K_i,[1,q]}$ be defined analogously. Fix integers $m\geq n$.

For each $i$, let $M_i\subset \mathcal R_{K_i,[1,q]}^m$ be a finite projective submodule of rank $n$, inducing a line $L_i=\wedge^n M_i\subset \wedge^n(\mathcal R_{K_i,[1,q]}^m)$. Assume that there is a line $L\subset \wedge^n(\mathcal R_{A,[1,q]}^m)$ restricting to all $L_i$ under the projections $A\to K_i$. Then there is a finite projective submodule $M\subset \mathcal R_{A,[1,q]}^m$ of rank $n$ restricting to all $M_i$ under the projections $A\to K_i$.
\end{lemma}

\begin{proof} For each $i$, the quotient
\[
\overline{L}_i=\wedge^n(\mathcal R_{A,[1,q]}^m)/L_i
\]
is a finitely generated module over the principal ideal domain $\mathcal R_{K_i,[1,q]}$. The degree of the torsion part is bounded uniformly, by the existence of $L$. Thus, we may choose some finite $t$ and for each $i$, choose $t$ untilts $K_{i,1}^\sharp,\ldots,K_{i,t}^\sharp$ of $K_i$ such that the torsion part $\overline{L}_i$ is concentrated at the quotients $\mathcal R_{K_i,[1,q]}\to K_{i,j}^\sharp$. (There may be no torsion, in which case we just choose any untilts.) Taking products, this defines untilts $A_1^\sharp,\ldots,A_t^\sharp$. Let $f\in \mathcal R_{A,[1,q]}$ be a function whose vanishing locus is given by those $t$ untilts.

Then $L[\tfrac 1f]$ defines a point of projective space $\mathbb P^{\binom{m}{n}-1}$ over $\mathcal R_{A,[1,q]}[\frac 1f]$. The Pl\"ucker embedding defines a Zariski closed immersion of the classical Grassmannian $\mathrm{Gr}(n,m)$ into $\mathbb P^{\binom{m}{n}-1}$. As $L[\tfrac 1f]$ defines a point in the image on the Zariski dense subspace
\[
\bigcup_i \mathrm{Spec}(\mathcal R_{K_i,[1,q]}[\tfrac 1f])\subset \mathrm{Spec}(\mathcal R_{A,[1,q]}[\tfrac 1f]),
\]
it follows that $L[\tfrac 1f]$ itself sits in the image, so that there is a unique finite projective $\mathcal R_{A,[1,q]}[\frac 1f]$-module
\[ M'\subset \mathcal R_{A,[1,q]}[\tfrac 1f]^m \]
of rank $n$ with
\[ \wedge^n(M') = L[\tfrac 1f]\subset \wedge^n(\mathcal R_{A,[1,q]}[\tfrac 1f]^m).\]

In fact, we can assume that $M'$ is free; more precisely, that it lies in one piece of the standard affine cover of $\mathrm{Gr}(n,m)$. Indeed, it suffices to arrange the same property for $L[\frac 1f]$. But given any point $x\in \Spa(A,A^+)$, we can look at the standard affine piece of $\mathbb P^{\binom{m}{n}-1}$ that $L$ belongs to at the generic point of the Fargues--Fontaine curve over $x$. Passing to a clopen cover of $\Spa(A,A^+)$, we can then arrange everything above (enlargening $t$ if necessary) so that also $L[\tfrac 1f]$ lies in this standard affine piece.

Thus, fix a finite free $\mathcal R_{A,[1,q]}$-model $M_0$ of $M'$. We are now looking for a finite projective $\mathcal R_{A,[1,q]}$-module $M$ with $M[\frac 1f]=M'$ and inducing the given $M_i$ after projection $A\to K_i$. The moduli space of finite projective $\mathcal R_{A,[1,q]}$-lattices $M\subset M'$ is, by Beauville--Laszlo gluing, a Beilinson--Drinfeld affine Grassmannian and therefore ind-proper. It remains to see that the $M_i$ lie in a bounded part. But in one direction, we get a bound as $M$ must necessarily lie in the intersection
\[ M'\cap \mathcal R_{A,[1,q]}^n\]
which is contained in $f^{-C} M_0$ for some sufficiently large $C$. In the other direction, we get a bound because the highest exterior power stays bounded by assumption on the existence of $L$.
\end{proof}
Combining Theorem \ref{thm: Subislocallyspatial} and Theorem \ref{thm: Subisquasicompact}, we obtain the following key Corollary.
\begin{corollary}\label{cor: properness}
The functor $\mathcal{S}\mathrm{ub}_{\mathcal{E}}^{n,d}$ is a locally
spatial diamond, and the map $\mathcal{S}\mathrm{ub}_{\mathcal{E}}^{n,d}\to S$
is proper of locally finite dim.trg.
\end{corollary}

\subsection{The geometry of $\wtBun_P$}

We now wish to study the Drinfeld compactifications. When working with these objects it is often convenient to assume that the derived group of $G$ is simply connected, and we will do this until the end of \S 7. However, in \S 7.3 we will define a variant of the Drinfeld compactification used to propagate our main results to the general case. We note that the assumption that the derived group of $G$ is simply connected implies that for all Levi subgroups $M \subset G$, the derived group of $M$ is simply connected. In particular, the map 
\[ \ol{(-)}: \pi_{1}(M)_{\Gamma} \simeq^{\kappa_{M}^{-1}} B(M)_{\basic} \xrightarrow{\simeq} B(M^{\mathrm{ab}})_{\basic} = B(M^{\mathrm{ab}}) \simeq^{\kappa_{M^{\mathrm{ab}}}} \mathbb{X}_{*}(M^{\mathrm{ab}}_{\ol{E}})_{\Gamma} \]
considered in \S \ref{section: defnofEisFunctors} is an isomorphism, and we will implicitly identify $\pi_{1}(M)_{\Gamma}$ with $\mathbb{X}_{*}(M^{\mathrm{ab}}_{\ol{E}})_{\Gamma}$ in this case.

\subsubsection{Basic definitions}
In order to understand how Verdier duality interacts with the Eisenstein functors, it will be natural to consider compactifications of the map
\[ \mathfrak{p}_P: \Bun_P \rightarrow \Bun_G, \]
referred to as the Drinfeld Compactifications in usual geometric Langlands \cite{BG}. 

To understand this, we first note that, for $S \in \Perf$, the moduli stack $\Bun_P$ parametrizes
\begin{enumerate}
\item a $G$-bundle $\mathcal{F}_G$ on $X_{S}$,
\item a $M$-bundle $\mathcal{F}_M$ on $X_{S}$,
\item and a $G$-equivariant map $\kappa: \mathcal{F}_G \rightarrow G/U \times^M \mathcal{F}_M$.
\end{enumerate}
By the Tannakian formalism, item (3) can in turn be described by injective maps
\[
\kappa^{\mathcal{V}}_P: (\mathcal{V}^{U})_{\mathcal{F}_{M}} \rightarrow \mathcal{V}_{\mathcal{F}_{G}},
\]
functorial in a $G$-module $\mathcal{V}$, such that the cokernel of $\kappa_P^{\mathcal V}$ is a vector bundle, and satisfying the following Pl\"ucker relations:
\begin{enumerate}
\item For the trivial representation $\mathcal{V}$, $\kappa_P^{\mathcal{V}}$ must be the identity map $\mathcal{O}_{X_S} \rightarrow \mathcal{O}_{X_S}$.
\item For a $G$-module map $\mathcal{V}^1 \rightarrow \mathcal{V}^2$, the induced square
\[\begin{tikzcd}
&   ((\mathcal{V}^1)^U) _{\mathcal{F}_M}\arrow[d] \arrow[r, "\kappa_P^{\mathcal{V}^1}"] & \mathcal{V}^1_{\mathcal{F}_G}  \arrow[d] & \\
& ((\mathcal{V}^2)^U)_{\mathcal{F}_M}  \arrow[r, "\kappa_P^{\mathcal{V}^2}"] & \mathcal{V}^2_{\mathcal{F}_G} &
\end{tikzcd}
\]
commutes.
\item For two $G$-modules $\mathcal{V}^1$ and $\mathcal{V}^2$, the diagram
\[\begin{tikzcd}
&   ((\mathcal{V}^1)^U \otimes (\mathcal{V}^2)^U)_{\mathcal{F}_M} \arrow[d] \arrow[r,"\kappa_P^{\mathcal{V}^1} \otimes \kappa_P^{\mathcal{V}^2}"] & \mathcal{V}^1_{\mathcal{F}_G} \otimes \mathcal{V}^2_{\mathcal{F}_G} \arrow[d,"\mathrm{id}"] & \\
& ((\mathcal{V}^1 \otimes \mathcal{V}^2)^U)_{\mathcal{F}_M} \arrow[r, "\kappa_P^{\mathcal{V}^1 \otimes \mathcal{V}^2}"] & \mathcal{V}^1_{\mathcal{F}_G} \otimes \mathcal{V}^2_{\mathcal{F}_G} &
\end{tikzcd}\]
commutes.
\end{enumerate}

In order to compactify the map $\mf{p}_{P}$ it is natural to consider "enhanced" $P$-structures in the sense that the maps $\kappa^{\mathcal{V}}_P$ can have torsion in their cokernel. In particular, we have the following definition for $\widetilde{\Bun}_{P}$:
\begin{definition}{\label{def: tildecompactification}}
We define $\widetilde{\Bun}_{P}$ to be the $v$-stack parametrizing, for $S \in \Perf$, triples
\[
(\mathcal{F}_G,\mathcal{F}_M, \tilde{\kappa}^{\mathcal{V}}_P)
\]
of a $G$-bundle $\mathcal F_G$, an $M$-bundle $\mathcal F_M$ and for every $G$-module $\mathcal{V}$ a map
\[ \tilde{\kappa}^{\mathcal{V}}_P: (\mathcal{V}^U)_{\mathcal{F}_M} \rightarrow \mathcal{V}_{\mathcal{F}_G} \]
that is a fiberwise (over $S$) injective map of vector bundles on $X_S$, satisfying the Pl\"ucker relations in the following sense:
\begin{enumerate}
\item For the trivial representation $\mathcal{V}$, $\tilde{\kappa}_P^{\mathcal{V}}$ is the identity map $\mathcal{O} \rightarrow \mathcal{O}$.
\item For a $G$-module map $\mathcal{V}^1 \rightarrow \mathcal{V}^2$, the induced square
\[\begin{tikzcd}
&   ((\mathcal{V}^1)^U) _{\mathcal{F}_M}\arrow[d] \arrow[r, "\tilde{\kappa}_P^{\mathcal{V}^1}"] & \mathcal{V}^1_{\mathcal{F}_G}  \arrow[d] & \\
& ((\mathcal{V}^2)^U)_{\mathcal{F}_M}  \arrow[r, "\tilde{\kappa}_P^{\mathcal{V}^2}"] & \mathcal{V}^2_{\mathcal{F}_G} &
\end{tikzcd}
\]
commutes.
\item For two $G$-modules $\mathcal{V}^1$ and $\mathcal{V}^2$, the diagram
\[\begin{tikzcd}
&   ((\mathcal{V}^1)^U \otimes (\mathcal{V}^2)^U)_{\mathcal{F}_M} \arrow[d] \arrow[r,"\tilde{\kappa}_P^{\mathcal{V}^1} \otimes \tilde{\kappa}_P^{\mathcal{V}^2}"] & \mathcal{V}^1_{\mathcal{F}_G} \otimes \mathcal{V}^2_{\mathcal{F}_G} \arrow[d,"\mathrm{id}"] & \\
& ((\mathcal{V}^1 \otimes \mathcal{V}^2)^U)_{\mathcal{F}_M} \arrow[r, "\tilde{\kappa}_P^{\mathcal{V}^1 \otimes \mathcal{V}^2}"] & \mathcal{V}^1_{\mathcal{F}_G} \otimes \mathcal{V}^2_{\mathcal{F}_G} &
\end{tikzcd}\]
commutes.
\end{enumerate}
\end{definition}

We will often omit the subscript $P$ in $\tilde{\kappa}_P^{\mathcal{V}}$ when it is clear from context.

We have a natural map $\tilde{j}_{P}: \Bun_{P} \hookrightarrow \wt{\Bun}_{P}$ mapping to the locus where all the maps $\tilde{\kappa}^{\mathcal{V}}$ have vector bundle cokernel. By forgetting all the data except $\mathcal{F}_{G}$ (resp. $\mathcal{F}_{M}$), we get morphisms $\tilde{\mathfrak{p}}_P: \widetilde{\Bun}_P \rightarrow \Bun_G$ (resp. $\tilde{\mathfrak{q}}_P: \widetilde{\Bun}_P \rightarrow \Bun_M$) which extend the maps $\mf{p}_{P}$ (resp. $\mf{q}_{P}$) along $\tilde{j}_{P}$. The space $\wt{\Bun}_{P}$ will be of foremost interest to us; however, its geometry is very complicated and to aid in its study it helps to consider a simpler variant of this space.  

To understand this variant, we recall that a map of bundles whose cokernel is also a bundle is uniquely determined by its top exterior power. In analogous fashion, we note that yet another way of understanding the moduli description of $\Bun_P$ is as parametrizing triples:
\[ (\mathcal{F}_G, \mathcal{F}_{M^{\mathrm{ab}}}, \overline{\kappa}_P: \mathcal{F}_G \rightarrow G/[P,P] \times^{M^{\mathrm{ab}}} \mathcal{F}_{M^{\mathrm{ab}}}), \]
where $\mathcal{F}_{G}$ is a $G$-bundle, $\mathcal{F}_{M^{\mathrm{ab}}}$ is an $M^{\mathrm{ab}}$-bundle, and $\kappa_P: \mathcal{F}_G \rightarrow G/[P,P] \times^{M^{\mathrm{ab}}} \mathcal{F}_{M^{\mathrm{ab}}}$ is a $G$-equivariant map. The Tannakian Formalism again provides us with a similar Pl\"ucker description of the map $\overline{\kappa}_{P}$. For every $G$-module $\mathcal{V}$, one obtains an induced map
\[ \overline{\kappa}_P^{\mathcal{V}}: (\mathcal{V}^{[P,P]})_{\mathcal{F}_{M^{\mathrm{ab}}}} \rightarrow \mathcal{V}_{\mathcal{F}_{G}} \]
of vector bundles on $X_{S}$ with vector bundle cokernel, satisfying the following:
\begin{enumerate}
\item For the trivial representation $\mathcal{V}$, $\overline{\kappa}_P^{\mathcal{V}}$ is the identity map $\mathcal{O} \rightarrow \mathcal{O}$.
\item For a $G$-module map $\mathcal{V}^1 \rightarrow \mathcal{V}^2$, the induced square
\[\begin{tikzcd}
&   ((\mathcal{V}^1)^{[P,P]})_{\mathcal{F}_{M^{\mathrm{ab}}}}\arrow[d] \arrow[r, "\overline{\kappa}_P^{\mathcal{V}^1}"] & \mathcal{V}^1_{\mathcal{F}_G}  \arrow[d] & \\
& ((\mathcal{V}^2)^{[P,P]})_{\mathcal{F}_{M^{\mathrm{ab}}}}  \arrow[r, "\overline{\kappa}_P^{\mathcal{V}^2}"] & \mathcal{V}^2_{\mathcal{F}_G} &
\end{tikzcd}
\]
commutes.
\item For two $G$-modules $\mathcal{V}^1$ and $\mathcal{V}^2$, the diagram
\[\begin{tikzcd}
&   ((\mathcal{V}^1)^{[P,P]} \otimes (\mathcal{V}^2)^{[P,P]})_{\mathcal{F}_{M^{\mathrm{ab}}}} \arrow[d] \arrow[r,"\overline{\kappa}_P^{\mathcal{V}^1} \otimes \overline{\kappa}_P^{\mathcal{V}^2}"] & \mathcal{V}^1_{\mathcal{F}_G} \otimes \mathcal{V}^2_{\mathcal{F}_G} \arrow[d,"\mathrm{id}"] & \\
& ((\mathcal{V}^1 \otimes \mathcal{V}^2)^{[P,P]})_{\mathcal{F}_{M^{\mathrm{ab}}}} \arrow[r, "\overline{\kappa}_P^{\mathcal{V}^1 \otimes \mathcal{V}^2}"] & \mathcal{V}^1_{\mathcal{F}_G} \otimes \mathcal{V}^2_{\mathcal{F}_G} &
\end{tikzcd}\]
commutes.
\end{enumerate}
This allows us to make the following definition. 
\begin{definition}
We define $\overline{\Bun}_{P}$ to be the $v$-stack parametrizing, for $S \in \Perf$, triples $(\mathcal{F}_G,\mathcal{F}_M, \overline{\kappa}^{\mathcal{V}}_P)$, where for every $G$-module $\mathcal V$ the map
\[
\overline{\kappa}^{\mathcal{V}}_P: (\mathcal{V}^{[P,P]})_{\mathcal{F}_{M^{\mathrm{ab}}}} \rightarrow \mathcal{V}_{\mathcal{F}_G} \]
is a fiberwise (on $S$) injective map of vector bundles on $X_S$, satisfying the Pl\"ucker relations as described above.
\end{definition}

Again, we will omit the subscript $P$ when it is obvious from context.

We have a natural map $\ol{j}_{P}: \Bun_{P} \ra \ol{\Bun}_{P}$. By forgetting all the data except $\mathcal{F}_{G}$ (resp. $\mathcal{F}_{M^{\mathrm{ab}}}$), we obtain maps $\ol{\mf{p}}_{P}$ (resp. $\ol{\mf{q}}^{\dagger}_{P}$) extending $\mf{p}_{P}$ (resp. $\mf{q}_{P}^{\dagger}$) along $\ol{j}_{P}$. However, unless $P$ is a Borel so that $M=M^{\mathrm{ab}}$, the stack $\ol{\Bun}_P$ no longer has a map to $\Bun_M$. There is also a natural map 
\[ \mf{t}_{P}: \wt{\Bun}_{P} \ra \ol{\Bun}_{P}, \]
\[ (\mathcal{F}_{G},\mathcal{F}_{M},\tilde{\kappa}^{\mathcal{V}}) \mapsto (\mathcal{F}_{G},\mathcal{F}_{M} \times^{M} M^{\mathrm{ab}} ,\overline{\kappa}^{\mathcal{V}}), \]
where $\overline{\kappa}^{\mathcal{V}}$ is the precomposition of $\tilde{\kappa}^{\mathcal{V}}$ with the embedding
\[
(\mathcal{V}^{[P,P]})_{\mathcal{F}_{M} \times^{M} M^{\mathrm{ab}}}\to (\mathcal V^U)_{\mathcal F_M}.
\]

We have the following basic structural result on these compactifications.
\begin{proposition}{\label{prop: openimmersion}}
The stacks $\ol{\Bun}_{P}$ and $\wt{\Bun}_{P}$ are Artin $v$-stacks and the maps $\ol{j}_{P}$ (resp. $\tilde{j}_{P}$) are open immersions.  
\end{proposition}
\begin{proof}
It suffices to show the claim after base-change to an algebraically closed perfectoid field $\Spa(F,\mathcal{O}_{F})$. We write $X$ for the associated Fargues-Fontaine curve. Recall that, given a scheme $Y$, one defines the affine closure to be $\overline{Y} = \Spec{(\Gamma(Y,\mathcal{O}_{Y}))}$. We let $\overline{G/U}$ be the affine closure of $G/U$. Viewing this as a constant scheme over $X$, we consider the stack
\[ Z := [G\backslash(\overline{G/U})/M] \rightarrow X. \]
Now, it follows by \cite[Theorem 1.12]{BG}, that, for $S \in \Perf$, a section
\[
\begin{tikzcd}
&  & Z \arrow[d] \\
& X_{S} \arrow[ur,"s",dotted] \arrow[r] & X
\end{tikzcd}
\]
is equivalent to the datum of a $M$-bundle (resp. $G$-bundle) $\mathcal{F}_{M}$ (resp. $\mathcal{F}_{G}$) on $X_{S}$ together with a family of maps $\tilde{\kappa}^{\mathcal{V}}$ of $\mathcal{O}_{X_{S}}$-modules satisfying the Pl\"ucker conditions. Therefore, if we consider $\mathcal{M}_{Z}$, the moduli stack parametrizing such sections, then $\wt{\Bun}_{P}$ is the sub-functor corresponding to the locus where these maps are injective after pulling back to a geometric point. By \cite[Remark~3.3]{HamJacCrit}, this is an open subfunctor. By \cite[Theorem~1.7]{HamJacCrit}, $\mathcal{M}_{Z}$ is an Artin $v$-stack; therefore, the same is true for $\wt{\Bun}_{P}$. It remains to see that $\Bun_{P}$ is an open sub-functor. Now, by the work of \cite{Gro}, it follows that $\overline{G/U}$ is strongly quasi-affine in the sense that $G/U \hookrightarrow \overline{G/U}$ is an open immersion. This induces an open immersion of stacks 
\[ [G\backslash(G/U)/M] = [X/P] \hookrightarrow Z \]
which, after passing to moduli stacks of sections, gives a natural map $\Bun_{P} \rightarrow \mathcal{M}_{Z}$ factoring through the open immersion $\wt{\Bun}_{P} \hookrightarrow \mathcal{M}_{Z}$. The map $\Bun_{P} \ra \mathcal{M}_{Z}$ is an open immersion, since $[X/P] \hookrightarrow Z$ is an open immersion, by arguing as in \cite[Proposition~IV.4.22]{FarguesScholze}. The case of $\ol{\Bun}_{P}$ follows in the exact same way, where we replace $G/U$ with $G/[P,P]$.
\end{proof}
For $\alpha \in \Mcoinv=\pi_1(M)_\Gamma$, we write $\wt{\Bun}_{P}^{\alpha}$ (resp. $\ol{\Bun}_{P}^{\alpha}$) for the pullback of $\wt{\Bun}_{P}$ (resp. $\ol{\Bun}_{P}$) to the connected component $\Bun_{M^{\mathrm{ab}}}^{\alpha}$ of $\Bun_{M^{\mathrm{ab}}}$. More generally, we will use the superscript $(-)^{\alpha}$ to denote the corresponding base-change of all relevant maps. We now have our first big theorem, which says that our Drinfeld compactifications are indeed compactifications in this setting. 
\begin{theorem}{\label{thm: drinfeldisrelativecompactification}}
For all $\alpha \in \Mcoinv$, the natural maps
\[ \ol{\mf{p}}_{P}^{\alpha}: \ol{\Bun}_{P}^{\alpha} \ra \Bun_{G} \]
and 
\[ \wt{\mf{p}}_{P}^{\alpha}: \wt{\Bun}_{P}^{\alpha} \ra \Bun_{G} \]
are both proper and representable in spatial diamonds of finite dim.trg.
\end{theorem}
\begin{proof}
We just explain the proof in the case of the compactification $\widetilde{\Bun}_{P}$ with the case of $\ol{\Bun}_{P}$ being strictly easier (it follows from essentially the same argument as \cite[Proposition~5.9]{HamGeomES}). For integers $1 \leq k \leq n$ and $d$, we let $\Bun_{k,n}^{d}$ denote the moduli stack parametrizing pairs $(\mathcal{F},\mathcal{E},i: \mathcal{F} \hookrightarrow \mathcal{E})$, where $\mathcal{E}$ is a vector bundle of rank $n$, $\mathcal{F}$ is a vector bundle of rank $k$ and degree $d$, and $i: \mathcal{F} \hookrightarrow \mathcal{E}$ is a map of $\mathcal{O}_{X_{S}}$-modules such that, after pulling back to any geometric point of $S$, it is an injection. We note that the map
\[ \Bun_{k,n}^{d} \ra \Bun_{\mathrm{GL}_{n}}, \]
remembering the bundle $\mathcal{E}$ is proper and representable in spatial diamonds of finite dim.trg., as follows immediately from Corollary \ref{cor: properness}.

Now we consider a $G$-module $\mathcal{V}$, and choose $d,k,n$ such that the map 
\[ i_{\mathcal{V}}: \wt{\Bun}^{\alpha}_{P} \rightarrow \Bun_G\times_{\Bun_{\mathrm{GL}_n}} \Bun_{k,n}^{d} \]
remembering only $\mathcal F_G$ and the tuple $(\mathcal{V}^{U}_{\mathcal{F}_{M}},\mathcal{V}_{\mathcal{F}_{G}},\tilde{\kappa}^{\mathcal{V}})$ is well-defined. If we choose $\mathcal{V}$ such that the image of $(\mathcal{V}^{U})^{\vee} \otimes \mathcal{V}$ inside the group ring $E[G]$ generates $E[G]^{U}$ as an $E$-algebra then $i_{\mathcal{V}}$ will be an injective map by the algebraic Peter-Weyl theorem, where $U$ acts via the left action on $E[G]$. More specifically, we can choose $\mathcal{V}$ so that $\mathcal V\otimes_E \ol{E}$ admits the direct sum of all fundamental representations of $G_{\ol{E}}$ as a direct summand. This generates $E[G] \otimes \ol{E} = \ol{E}[G_{\ol{E}}]$ as an $\ol{E}$-algebra by the algebraic Peter-Weyl theorem over the algebraic closure (in the form appearing in \cite[Proposition~4.20]{JanRepresentationsofAlgebraicGroups}), and hence the same is already true over $E$. Moreover, it is easy to see that in this case $i_{\mathcal{V}}$ defines a closed embedding; in particular, by taking the fiber over an $S$-point of the image, we can see that it maps to the subspace of Banach-Colmez spaces defined by the vanishing of certain $\mathcal{O}_{X_{S}}$-linear maps, which is closed since Banach-Colmez spaces are separated.
\end{proof}
We now turn our attention to studying a locally closed stratification of these Drinfeld compactifications. 
\subsubsection{Stratifications}{\label{sec: stratifications}}
To stratify the spaces $\wt{\Bun}_{P}$ and $\ol{\Bun}_{P}$, it is useful to consider a locally closed stratification of the space $\ol{\Bun}_{P}$ and then define a locally closed stratification of $\wt{\Bun}_{P}$ by taking the preimage along the natural map
\[ \mf{t}_{P}: \wt{\Bun}_{P} \ra \ol{\Bun}_{P} \]
considered in the previous section.

Recall that the compactification $\ol{\Bun}_P$ was based on the affinization $\overline{G/[P,P]}$ of $G/[P,P]$. Let $D=G^{\mathrm{ab}}$ be the torus quotient of $G$. Then
\[
G/[P,P]\to G/P\times D
\]
is a torsor under the torus $T_P:=\mathrm{ker}(P\to D)/[P,P]\cong \mathrm{ker}(M^{\mathrm{ab}}\to D)$ over the flag variety $G/P$. As the derived group of $G$ is simply connected, this torus admits a canonical trivialization after base change to $\ol{E}$,
\[
T_{P,\ol{E}}\cong \prod_{i\in \tilde{\mathcal J}_{G,P}} \mathbb G_{m,\ol{E}},
\]
where the individual $\mathbb G_m$-factors come from the coroots $\tilde{\mathcal J}_{G,P}$ corresponding to the parabolic $P_{\ol{E}}\subset G_{\ol{E}}$. The isomorphism is equivariant under the absolute Galois group $\Gamma$. We pick representatives of $\mathcal J_{G,P} = \tilde{\mathcal J}_{G,P}/\Gamma$, which gives an isomorphism
\[
T_P\cong \prod_{i\in \mathcal J_{G,P}} \mathrm{Res}_{E_i/E}(\mathbb G_{m,E_i})
\]
where $E_i/E$ is the finite separable extension fixing the lift of $i$ in $\tilde{\mathcal J}_{G,P}$.

In particular, we get a canonical isomorphism $B(T_P)\cong \mathbb Z^{\mathcal J_{G,P}}$; we write $\alpha_i\in B(T_P)$ for the basis element indexed by $i\in \mathcal J_{G,P}$. Moreover, we let $B(T_P)^{\mathrm{pos}}\subset B(T_P)$ be the subset of nonnegative elements. This is canonically isomorphic to the monoid $\Lambda_{G,P}^{\mathrm{pos}}$ defined in the notation section.

\begin{definition}
For $\theta \in B(T_P)^{\mathrm{pos}}$ of the form $\theta = \sum_{i \in \mathcal{J}_{G,P}} n_{i}\alpha_{i}$ with $n_{i} \geq 0$, we consider the partially symmetrized curve 
\[ \Div^{(\theta)} := \prod_{i \in \mathcal{J}_{G,P}} \Div^{(n_{i})}_{E_{i}}, \]
where $E_{i}$ is as above.
\end{definition}

\begin{remark}{\label{rem: symmcurveismodifications}}
For $S \in \Perf$, we note that we can interpret $\Div^{(\theta)}$ as the space parametrizing a $T_P$-bundle $\mathcal F_{T_P}$ with Kottwitz element $\theta$ together with a modification of $T_P$-bundles
\[
\beta_{T_P}: \mathcal F_{T_P}\dashrightarrow \mathcal F^0_{T_P}
\]
to the trivial $T_P$-bundle, with everywhere nonnegative meromorphy in $\mathbb X_\ast(T_{P,\ol{E}})\cong \mathbb Z^{\tilde{\mathcal J}_{G,P}}$. Using pushout along $T_P\to M^{\mathrm{ab}}$ (whose quotient $D$ is still a torus), this can also be described as the space of $M^{\mathrm{ab}}$-bundles $\mathcal F_{M^{\mathrm{ab}}}$ on $X_{S}$ with Kottwitz element equal to (the image of) $\theta$ together with a modification of $M^{\mathrm{ab}}$-bundles 
\[ \beta_{M^{\mathrm{ab}}}: \mathcal{F}_{M^{\mathrm{ab}}} \dashrightarrow \mathcal{F}^{0}_{M^{\mathrm{ab}}} \]
to the trivial $M^{\mathrm{ab}}$-bundle with everywhere meromorphy in the subspace of $\mathbb X_\ast(M^{\mathrm{ab}}_{\ol{E}})$ given by nonnegative elements of $\mathbb X_\ast(T_{P,\ol{E}})$.
\end{remark}

We have a natural map of Artin $v$-stacks
\[ j_{\theta}: \Div^{(\theta)} \times \Bun_{P} \ra \ol{\Bun}_{P} \]
which modifies the $\mathcal F_{M^{\mathrm{ab}}}$-bundle by the twist of the preceding remark. The condition on the meromorphy of this modification ensures that for all representations $\mathcal V$ of $G$, one gets well-defined maps $\overline{\kappa}^{\mathcal V}$. Indeed $\mathcal V^{[P,P]}$ is (over $\ol{E}$) a direct sum of cocharacters of $M^{\mathrm{ab}}$ each of which pairs nonnegatively with the positive coroots $\tilde{\mathcal J}_{G,P}$ and hence with the nonnegative elements of $\mathbb X_\ast(T_{P,\ol{E}})$.

We now make the following definition. Here, for a torus $T$ over $E$, we use the pairing
\[
\langle -,-\rangle: K_0(T)\times \mathbb X_\ast(T)_\Gamma\to \mathbb Z
\]
between the Grothendieck group $K_0(T)$ of representations of $T$ to $\mathbb Z$ and $X_\ast(T)_\Gamma$. This is defined via Galois descent, and the definition reduces to $T=\mathbb G_m$ and the pairing with its tautological cocharacter. Here it is the map
\[
K_0(\mathbb G_m)\cong \bigoplus_{i\in \mathbb Z}\mathbb Z\cdot \chi_i\to \mathbb Z: \chi_i\mapsto i,
\]
where $\chi_i: \mathbb G_m\to \mathbb G_m$ is the weight $i$ character $z\mapsto z^i$.

\begin{definition}{\label{defstrata}}
For $\theta \in \Lambda_{G,P}^{\mathrm{pos}}=B(T_P)^{\mathrm{pos}}$, we define the $v$-stack $\phantom{}_{\theta}\ol{\Bun}_{P}$ (resp. $_{\leq \theta}\ol{\Bun}_{P}$) to be the locus where, for all representations $\mathcal V$ of $G$, the cokernels of $\overline{\kappa}^{\mathcal V}$ have torsion of length equal to (resp. less than or equal to) the pairing $\langle\mathcal V^{[P,P]},\theta\rangle$.
\end{definition}

By Proposition \ref{lemma: uppersemicont} (i), $_{\theta}\ol{\Bun}_{P}$ is a locally closed substack of $\ol{\Bun}_{P}$, and $_{\leq \theta}\ol{\Bun}_{P}$ an open substack.

To work with these strata, we will need the following. 
\begin{proposition}{\label{prop: strata well-behaved}}
For $\theta \in \Lambda_{G,P}^{\mathrm{pos}}$, the map $j_{\theta}$ induces an isomorphism $\Bun_{P} \times \Div^{(\theta)} \simeq \phantom{}_{\theta}\ol{\Bun}_{P}$. In particular, $j_{\theta}$ is a locally closed embedding. 
\end{proposition}
\begin{proof} It is clear that $j_{\theta}$ induces a map into the locally closed stratum $_{\theta}\ol{\Bun}_{P}$. We need to exhibit an inverse of this map. We first have the following lemma.
\begin{lemma}{\cite[Lemma~5.13]{HamGeomES}}
Let $\mathrm{Coh}$ be the $v$-stack parametrizing, for $S \in \Perf$, flat coherent sheaves on $X_{S}$, as in \cite[Theorem~2.11]{AL1}. For $k,n \in \mathbb{N}_{\geq 0}$, we set $\Coh^{k}_{n}$ to be the locally closed (by Proposition \ref{lemma: uppersemicont} (i)) substack parametrizing flat coherent sheaves whose torsion length (resp. vector bundle rank) is equal to $k$ (resp. $n$) after pulling back to a geometric point. There is a well-defined map
\[ \Coh_{n}^{k} \ra \Div^{(k)} \]
of $v$-stacks, sending a $S$-flat coherent sheaf $\mathcal{F}$ with attached short exact sequence $0 \ra \mathcal{F}^{\mathrm{tor}} \ra \mathcal{F} \ra \mathcal{F}^{\mathrm{vb}} \ra 0$, as in Proposition \ref{lemma: uppersemicont} (ii), to the support of $\mathcal{F}^{\mathrm{tor}}$. 
\end{lemma}
We need to exhibit an inverse to the natural map $\Bun_{P} \times \Div^{(\theta)} \ra \phantom{}_{\theta}\ol{\Bun}_{P}$. To do this, for $S \in \Perf$, consider a short exact sequence \[ 0 \ra \mathcal{V}_{1} \ra \mathcal{V}_{n + 1} \ra \mathcal{V}_{n} \ra 0  \]
of $\mathcal{O}_{X_{S}}$-modules, where $\mathcal{V}_{1}$ (resp. $\mathcal{V}_{n + 1}$) is a line bundle (resp. rank $n + 1$ vector bundle), and the first map is a fiberwise-injective map, so that $\mathcal{V}_{n}$ is $S$-flat. Assume that $\mathcal{V}_{n}$ defines a point in $\Coh_{n}^{k}$ for some $k$, and let $D$ be the degree $k$ Cartier divisor in $X_{S}$ defined by the previous Lemma. It follows by an application of Proposition \ref{lemma: uppersemicont} (ii) that we have a short exact sequence 
\[ 0 \ra \mathcal{V}_{n}^{\mathrm{tors}} \ra \mathcal{V}_{n} \ra \mathcal{V}_{n}^{\mathrm{vb}} \ra 0  \]
where $\mathcal{V}_{n}^{\mathrm{tors}}$ will define a point in $\Coh^{k}_{0}$ and $\mathcal{V}_{n}^{\mathrm{vb}}$ is a rank $n$ vector bundle. Let $\widetilde{\mathcal{V}}_{1}$ denote the preimage of $\mathcal{V}_{n}^{\mathrm{tors}}$ in $\mathcal{V}_{n + 1}$. It is then easy to see that $\mathcal{V}_{1} \ra \mathcal{V}_{n + 1}$ gives rise to an isomorphism $\mathcal{V}_{1}(D) \simeq \widetilde{\mathcal{V}}_{1}$.

A similar story works also with line bundles replaced by $M^{\mathrm{ab}}$-torsors; for example, one can do Galois descent in $E$ to reduce to the split case. Now, given a $S$-point of $_{\theta}\ol{\Bun}_{P}$, we can construct the desired inverse by applying the above argument to the short exact sequences coming from the embeddings $\ol{\kappa}^{\mathcal{V}}$, for $\mathcal{V}$ varying over finite-dimensional subrepresentations of $E[G]^{[P,P]}$ (or rather the isotypic pieces for the commuting $M^{\mathrm{ab}}$-action).
\end{proof}

\begin{remark} A priori it is not clear that
\[
\ol{\Bun}_{P} = \bigcup_{\theta} \phantom{}_{\theta}\ol{\Bun}_{P}.
\]
But this is true. To see this, note that this claim would be clear if in the definition we would only stratify according to the order of the torsion of the cokernel $\overline{\kappa}^{\mathcal V}$ for one fixed representation $\mathcal V$. But the previous proof shows that it is sufficient to take any sufficiently large finite-dimensional subrepresentation of $E[G]^{[P,P]}$ as already that stratum will be isomorphic to $\Bun_P\times \Div^{(\theta)}$.
\end{remark}

For $\theta \in \Lambda_{G,P}^{\mathrm{pos}}$, we write $\phantom{}_{\theta}\wt{\Bun}_{P}$ for the pullback of the locally closed strata $\phantom{}_{\theta}\ol{\Bun}_{P}$ along the map $\mf{t}_{P}$. We now want to describe the locally closed strata $\phantom{}_{\theta}\wt{\Bun}_{P}$. We will do this using the following space.
\begin{definition}{\label{def: ModM+}}
For $\theta \in \Lambda_{G,P}^{\mathrm{pos}}$, we define $\mathrm{Mod}^{+,\theta}_{\Bun_{M}}$ to be the $v$-sheaf parametrizing, for $S \in \Perf$, triples $(\mathcal{F}_{M}^{2},\mathcal{F}_{M}^{1},\beta_{M})$ where
\begin{enumerate}
    \item $\mathcal{F}_{M}^{i}$ is an $M$-bundle on $X_{S}$ for $i = 1,2$,
    \item $\beta_{M}: \mathcal F_M^2\dashrightarrow \mathcal F_M^1$ is a modification of $M$-bundles such that, for every $G$-module $\mathcal{V}$, the induced map 
    \[ \beta_{M}^{\mathcal{V}^U}: (\mathcal{V}^{U})_{\mathcal{F}_{M}^{2}} \ra (\mathcal{V}^{U})_{\mathcal{F}_{M}^{1}} \]
    is a well-defined map of vector bundles on $X_S$,
    \item The induced modification 
    \[ \beta_{M^{\mathrm{ab}}}: \mathcal{F}_{M}^{2} \times^{M} M^{\mathrm{ab}} \dashrightarrow \mathcal{F}_{M}^{1} \times^{M} M^{\mathrm{ab}} \]
    has total meromorphy $\theta$.
\end{enumerate}
\end{definition}

We note that the last condition is a purely numerical condition on the connected components of $\Bun_M$ that $\mathcal F_M^1$ and $\mathcal F_M^2$ lie in; in fact, it simply means $\kappa(\mathcal F_M^2)-\kappa(\mathcal F_M^1)=\theta$ after pulling back to each geometric point of $S$.

We have a natural map 
\[ \pi_{M}: \mathrm{Mod}^{+,\theta}_{\Bun_{M}} \ra \Div^{(\theta)}, \]
which records the locus of meromorphy of the modification $\beta_{M^{\mathrm{ab}}}$, as in Remark \ref{rem: symmcurveismodifications}. (The nonnegativity of the modification $\beta_{M^{\mathrm{ab}}}$ results from condition (2).) We also have a projection 
\[ h_{M}^{\ra}: \mathrm{Mod}^{+,\theta}_{\Bun_{M}} \ra \Bun_{M} \]
which remembers the bundle $\mathcal{F}_{M}^{1}$. Using this map, we form the fiber product $\mathrm{Mod}^{+,\theta}_{\Bun_{M}} \times_{\Bun_{M}} \Bun_{P}$, and consider the natural map 
\[ \tilde{j}_{\theta}: \mathrm{Mod}^{+,\theta}_{\Bun_{M}} \times_{\Bun_{M}} \Bun_{P} \hookrightarrow \wt{\Bun}_{P}\]
\[(\mathcal{F}_{M}^{2},\mathcal{F}_{P}^{1},\beta_{M}) \mapsto  (\mathcal{F}_{G},\mathcal{F}_{M}^{2},\tilde{\kappa}^{\mathcal{V}}).\]
Here $\mathcal{F}_{G} = \mathcal{F}_{P}^{1} \times^{P} G$ and $\tilde{\kappa}^{\mathcal{V}}$ is given by the composition
\[ (\mathcal{V}^{\mathcal{U}})_{\mathcal{F}_{M}^{2}} \xrightarrow{\beta_{M}^{\mathcal{V}}} (\mathcal{V}^{\mathcal{U}})_{\mathcal{F}_{M}^{1}} \xrightarrow{\kappa^{\mathcal{V}}} \mathcal{V}_{\mathcal{F}_{G}}, \]
where the last map is the embedding of vector bundles determined by the $P$-bundle $\mathcal{F}_{P}^{1}$. We now have the following.
\begin{proposition}{\label{prop: tildestratadesc}}
For $\theta \in \Lambda_{G,P}^{\mathrm{pos}}$, the map $\tilde{j}_{\theta}$ defines an isomorphism onto $\phantom{}_{\theta}\wt{\Bun}_{P}$. In particular, it is a locally closed embedding. Moreover, we have a Cartesian diagram:
\[ \begin{tikzcd} & \mathrm{Mod}^{+,\theta}_{\Bun_{M}} \times_{\Bun_{M}} \Bun_{P} \arrow[r,"\tilde{j}_{\theta}"] \arrow[d,"\pi_{M} \times \mathrm{id}"] &  \wt{\Bun}_{P} \arrow[d,"\mf{t}_{P}"] \\ & \Div^{(\theta)} \times \Bun_{P} \arrow[r,"\ol{j}_{\theta}"] & \ol{\Bun}_{P}. \end{tikzcd}\]
\end{proposition}
\begin{proof}
It is clear that the image of $\tilde{j}_{\theta}$ lies in $\phantom{}_{\theta}\wt{\Bun}_{P}$ and that the above diagram is commutative and Cartesian once we show that $\tilde{j}_{\theta}$ is an isomorphism. We need to exhibit an inverse to $\tilde{j}_{\theta}$. 

To exhibit an inverse, let $(\mathcal{F}_G,\mathcal{F}^{2}_M, \tilde{\kappa}^{\mathcal{V}})$ be an $S$-point of $\phantom{}_{\theta}\wt{\Bun}_{P}$. We let $(\mathcal{F}_P^{1},(D_{i})_{i \in \mathcal{J}_{G,P}})$ be the corresponding $S$-point of $\Bun_P \times \Div^{(\theta)} \simeq \phantom{}_{\theta}\ol{\Bun}_{P}$ by applying $\mf{t}_{P}$ and Proposition \ref{prop: strata well-behaved}. We write $\mathcal{F}_M^{1} := \mathcal{F}_{P}^{1} \times^{P} M$, and $\mathcal{F}_{G} := \mathcal{F}_{P}^{1} \times^{P} G$. By construction, $\mathcal{F}_M^{1}$ and $\mathcal{F}^{2}_M$ are identified outside of the Cartier divisors  $(D_{i})_{i \in \mc{J}_{G,P}}$ and thus we get a modification 
\[ \beta_{M}: \mathcal{F}_{M}^{2} \dashrightarrow \mathcal{F}_{M}^{1} \]
with meromorphy along this locus.  Moreover, for each $G$-module $\mathcal{V}$, we have a map
\[ \tilde{\kappa}^{\mathcal{V}}: (\mathcal{V}^U)_{\mathcal{F}^{2}_M} \hookrightarrow \mathcal{V}_{\mathcal{F}_G},  \]
as well as the subbundle supplied by the $P$-bundle $\mathcal{F}_P$:
\[ \kappa^{\mathcal{V}}: (\mathcal{V}^U)_{\mathcal{F}_M^{1}} \hookrightarrow \mathcal{V}_{\mathcal{F}_G} \]
where these two maps are connected by the modification $\beta_{M}$. Hence, the a priori meromorphic map $(\mathcal{V}^U)_{\mathcal{F}^{2}_M} \rightarrow (\mathcal{V}^U)_{\mathcal{F}^{1}_M}$ is actually regular along the $(D_{i})_{i \in \mc{J}_{G,P}}$ and from here we see that $(\mc{F}_{P}^{1},\mc{F}_{M}^{2},\beta_{M})$ gives us the desired point inside $\mathrm{Mod}^{+,\theta}_{\Bun_{M}} \times_{\Bun_{M}} \Bun_{P}$.
\end{proof}

\section{Zastava spaces}

In order to show that the sheaf $\tilde{j}_{P!}(\IC_{\Bun_{P}})$ is ULA over $\Bun_{M}$, we will need some nice local models which are related to the compactification $\wt{\Bun}_{P}$ by a cohomologically smooth map. Ideally, these local models will be related to the affine Grasmannian and in turn understandable through their very explicit geometry. Indeed, this can be accomplished through considering Zastava spaces, as was already done in the global function field setting in \cite{BFGM}. 
\subsection{Definition and relationship to Drinfeld compactifications}{\label{sec: Zastava uniformizes Drinfeld}}
So far, $M$ was merely the Levi quotient of $P$, but now we fix an embedding back into $P$. Let $P^{-}$ denote the corresponding opposite parabolic to $P$. We define the following.
\begin{definition}{\label{def: Zastavaspace}}
For $\theta \in \Lambda_{G,P}^{\mathrm{pos}}$, we define $\tilde{Z}^{\theta}_{\Bun_{M}}$ to be the $v$-stack parametrizing, for $S \in \Perf$, a quintuple $(\mathcal{F}_{G},\mathcal{F}_{M}^{2},\mathcal{F}_{M}^{1},\beta_{M},\beta)$, where
\begin{enumerate}
\item $\mathcal{F}_{G}$ is a $G$-bundle on $X_{S}$,
\item $(\mathcal{F}_{M}^{2},\mathcal{F}_{M}^{1},\beta_{M})$ is an $S$-point of $\mathrm{Mod}^{+,\theta}_{\Bun_{M}}$ as in Definition \ref{def: ModM+},
\item $\beta: \mathcal{F}_{G} \dashrightarrow \mathcal{F}_{G}^{1}:= \mathcal{F}_{M}^{1} \times^{M} G$ is a modification of $G$-bundles away from the Cartier divisors defined by the previous item and the map $\pi_{M}: \mathrm{Mod}^{+,\theta}_{\Bun_{M}} \rightarrow \Div^{(\theta)}$.
\item For every $G$-module $\mc{V}$, the meromorphic map defined by $\beta_{M}$ and $\beta$
\[ (\mc{V}^{U})_{\mathcal{F}_{M}^{2}} \xrightarrow{\beta_{M}^{\mc{V}^U}} (\mc{V}^{U})_{\mathcal{F}_{M}^{1}} \hookrightarrow \mathcal{V}_{\mathcal{F}_{G}^{1}} \xdashrightarrow{(\beta^{\mc{V}})^{-1}} \mathcal{V}_{\mathcal{F}_{G}} \]
is a well-defined map of vector bundles on $X_S$, where the second map is defined by the split $P$-structure on $\mathcal{F}_{G}^{1}$.
\item The meromorphic map 
\[ \mathcal{V}_{\mathcal{F}_{G}} \xdashrightarrow{\beta^{\mc{V}}} \mathcal{V}_{\mathcal{F}_{G}^{1}} \rightarrow (\mathcal{V}_{U^{-}})_{\mathcal{F}_{M}^{1}} \]
is a well-defined surjective map of vector bundles on $X_S$, where the second map is defined by the split $P^{-}$-structure on $\mathcal{F}_{G}^{1}$.
\end{enumerate}
We write $\tilde{\pi}_{P}: \tilde{Z}^{\theta}_{\Bun_{M}} \ra \mathrm{Mod}^{+}_{\Bun_{M}}$, $h_{M}^{\ra}: \tilde{Z}^{\theta}_{\Bun_{M}} \ra \Bun_{M}$, and $h_{M}^{\leftarrow}: \tilde{Z}^{\theta}_{\Bun_{M}} \ra \Bun_{M}$ for the natural maps remembering the modification $(\mathcal{F}_{M}^{2},\mathcal{F}_{M}^{1},\beta_{M})$, the bundle $\mathcal{F}_{M}^{1}$, and the bundle $\mathcal{F}_{M}^{2}$, respectively. Given any $v$-stack $S \ra \Bun_{M}$, we write $\tilde{Z}^{\theta}_{S}$ for the base-change to $S$ along $h_M^{\ra}$. If $\ast \ra \Bun_{M}$ is the point defined by the trivial $M$-bundle, we simply write $\tilde{Z}^{\theta}$ for the corresponding $v$-sheaf. We let $j_{Z}: Z^{\theta}_{\Bun_{M}} \hookrightarrow \tilde{Z}^{\theta}_{\Bun_{M}}$ denote the open (by Lemma \ref{lemma: uppersemicont}) subspace for which the maps 
\[ (\mathcal{V}^{U})_{\mathcal{F}_{M}^{2}} \hookrightarrow \mathcal{V}_{\mathcal{F}_{G}} \]
of vector bundles on $X_S$ have torsion-free cokernel.
\end{definition}

This definition seems quite daunting; however, it actually occurs quite naturally. In particular, we can consider the fiber product  
\[ \begin{tikzcd} 
& \wt{\Bun}_{P} \times_{\Bun_{G}} \Bun_{P^{-}} \arrow[r] \arrow[d] & \wt{\Bun}_{P} \arrow[d,"\wt{\mf{p}}_{P}"] \\
& \Bun_{P^{-}} \arrow[r,"\mf{p}_{P^{-}}"] &  \Bun_{G}
\end{tikzcd}, \]
and an $S$-point $(\mathcal{F}_{G},\mathcal{F}_{M}^{2},\mathcal{F}_{M}^{1},\tilde{\kappa}^{\mathcal{V}},\kappa^{-,\mathcal{V}})$ of $\wt{\Bun}_{P} \times_{\Bun_{G}} \Bun_{P^{-}}$, where $\tilde{\kappa}^{\mathcal{V}}: (\mathcal{V}^{U})_{\mathcal{F}_{M}^{2}} \hookrightarrow \mathcal{V}_{\mathcal{F}_{G}}$ defines a point of $\wt{\Bun}_{P}$ and $\kappa^{-,\mathcal{V}}: \mathcal{V}_{\mathcal{F}_{G}} \ra (\mathcal{V}_{U^{-}})_{\mathcal{F}_{M}^{1}}$ defines a point of $\Bun_{P^{-}}$. We restrict to the subspace
\[
(\wt{\Bun}_{P} \times_{\Bun_{G}} \Bun_{P^{-}})_{0} \subset \wt{\Bun}_{P} \times_{\Bun_{G}} \Bun_{P^{-}} 
\]
where the $P$ and $P^-$-reductions are generically transverse.

To describe this in Tannakian terms, we recall that for any $M$-module $\mathcal{W}$, we can form a $G$-equivariant vector bundle still denoted $\mathcal W$ on $G/P^-$, and then form the $G$-module $\mathcal V=H^0(G/P^-,\mathcal W)$. Via restriction to $P^-/P^-\subset G/P^-$, this admits a natural map $\mathcal V_{U^-}\to \mathcal W$. More precisely, the functor $\mathcal V\mapsto \mathcal V_{U^-}$ is left adjoint to $\mathcal W\mapsto H^0(G/P^-,\mathcal W)$. On the other hand, via restriction to $U\cong (U\times P^-)/P^-\subset G/P^-$, there is a natural injective map $\mathcal V^U\to \mathcal W$ that is an isomorphism whenever all weights of the center of $M$ on $\mathcal W$ are nonnegative. In particular, this happens whenever $\mathcal W=\mathcal V'^U$ for some representation $\mathcal V'$ of $G$.

Now we can give the following Tannakian description of the desired open substack.

\begin{definition} 
We define the substack $(\wt{\Bun}_{P} \times_{\Bun_{G}} \Bun_{P^{-}})_{0} \subset \wt{\Bun}_{P} \times_{\Bun_{G}} \Bun_{P^{-}}$ by the condition that for all $M$-modules $\mathcal W$ such that all weights of the center of $M$ on $\mathcal W$ are nonnegative, so that the induced $G$-module $\mathcal V=H^0(G/P^-,\mathcal W)$ satisfies $\mathcal V^U\cong \mathcal W$, the map
\[ \mathcal W_{\mathcal F_M^2}\cong (\mathcal{V}^{U})_{\mathcal{F}_{M}^{2}} \xrightarrow{\tilde{\kappa}^{\mathcal{V}}} \mathcal{V}_{\mathcal{F}_{G}} \xrightarrow{\kappa^{-,\mathcal{V}}} (\mathcal{V}_{U^{-}})_{\mathcal{F}_{M}^{1}}\to \mathcal W_{\mathcal F_M^1}  \]
of vector bundles on $X_S$ is fiberwise (on $S$) injective.
\end{definition}

We have the following.
\begin{lemma}{\label{lemma: transverseisanopencondition}}
The inclusion
\[ (\wt{\Bun}_{P} \times_{\Bun_{G}} \Bun_{P^{-}})_{0} \subset \wt{\Bun}_{P} \times_{\Bun_{G}} \Bun_{P^{-}} \]
is an open immersion.
\end{lemma}
\begin{proof}
This follows from the definition and \cite[Remark~3.3]{HamJacCrit}.
\end{proof}
For $\alpha,\alpha' \in \pi_1(M)_\Gamma=\Lambda_{G,P}$, we write 
\[ (\wt{\Bun}^{\alpha}_{P} \times_{\Bun_{G}} \Bun_{P^{-}}^{\alpha'})_{0} \subset \wt{\Bun}_{P}^{\alpha} \times_{\Bun_{G}} \Bun^{\alpha'}_{P^{-}} \]
for the open subspace defined by restricting to the relevant open and closed subspaces. Given an $S$-point of $(\wt{\Bun}^{\alpha}_{P} \times_{\Bun_{G}} \Bun_{P^{-}}^{\alpha'})_{0}$, by definition we get a fibrewise injective map of vector bundles
\[
\beta_M^{\mathcal W}: \mathcal W_{\mathcal F_M^2}\cong (\mathcal{V}^{U})_{\mathcal{F}_{M}^{2}} \xrightarrow{\tilde{\kappa}^{\mathcal{V}}} \mathcal{V}_{\mathcal{F}_{G}} \xrightarrow{\kappa^{-,\mathcal{V}}} (\mathcal{V}_{U^{-}})_{\mathcal{F}_{M}^{1}}\to \mathcal W_{\mathcal F_M^1}
\]
for all representations $\mathcal W$ of $M$ with induced $G$-representation $\mathcal V$ satisfying $\mathcal V^U\cong \mathcal W$. This extends uniquely to meromorphic maps for all representations $\mathcal W$ of $M$, as we can always twist by central characters of $M$ to achieve the required condition (and then untwist by the map for this character of $M$). In particular, we get an $S$-point $(\mathcal{F}_{M}^{2},\mathcal{F}_{M}^{1},\beta_{M})$ of $\mathrm{Mod}^{+,\alpha - \alpha'}_{\Bun_{M}^{\alpha'}}$, where we note that it must necessarily be the case that $\alpha - \alpha'$ lies in $\Lambda_{G,P}^{\mathrm{pos}}$ in order for $\mathrm{Mod}^{+,\alpha - \alpha'}_{\Bun_{M}^{\alpha'}}$ to be non-empty. We now have the following. 
\begin{lemma}{\label{lemma: relationofZastavatoPstructures}}
For every $\theta \in \Lambda_{G,P}^{\mathrm{pos}}$ and $\alpha \in \Lambda_{G,P}=\pi_1(M)_\Gamma$, there is an isomorphism between $\tilde{Z}^{\theta}_{\Bun_{M}^{\alpha}}$ and $(\wt{\Bun}^{\theta + \alpha}_{P} \times_{\Bun_{G}} \Bun_{P^{-}}^{\alpha})_{0}$. Under this isomorphism, the open subset $Z^\theta_{\Bun_M^{\alpha}}$ is the preimage of $\Bun_P\subset \wt{\Bun}_P$. 
\end{lemma}
\begin{proof} Given an $S$-point of $(\mathcal{F}_{G},\mathcal{F}_{M}^{2},\mathcal{F}_{M}^{1},\beta_{M},\beta)$ of $\tilde{Z}^{\theta}_{\Bun_{M}^{\alpha}}$ it is clear that the maps 
\[ (\mc{V}^{U})_{\mathcal{F}_{M}^{2}} \xrightarrow{\beta_{M}^{\mc{V}^U}} (\mc{V}^{U})_{\mathcal{F}_{M}^{1}} \hookrightarrow \mathcal{V}_{\mathcal{F}_{G}^{1}} \xdashrightarrow{(\beta^{\mc{V}})^{-1}} \mathcal{V}_{\mathcal{F}_{G}} \]
and 
\[ \mathcal{V}_{\mathcal{F}_{G}} \xdashrightarrow{\beta^{\mc{V}}} \mathcal{V}_{\mathcal{F}_{G}^{1}} \rightarrow (\mathcal{V}_{U^{-}})_{\mathcal{F}_{M}^{1}} \]
appearing in the definition of $\tilde{Z}^{\theta}_{\Bun_{M}^{\alpha}}$ define maps $\tilde{\kappa}^{\mathcal{V}}$ and $\kappa^{-,\mathcal{V}}$, respectively, which will give rise to an $S$-point of $(\wt{\Bun}^{\theta + \alpha}_{P} \times_{\Bun_{G}} \Bun_{P^{-}}^{\alpha})_{0}$.

Conversely, given an $S$-point
\[
(\mathcal{F}_{G},\mathcal{F}_{M}^{2},\mathcal{F}_{M}^{1},\tilde{\kappa}^{\mathcal{V}},\kappa^{-,\mathcal{V}})\in (\wt{\Bun}^{\theta + \alpha}_{P} \times_{\Bun_{G}} \Bun_{P^{-}}^{\alpha})_{0}(S),
\]
we let  $(\mathcal{F}_{M}^{2},\mathcal{F}_{M}^{1},\beta_{M})$ be the associated $S$-point of $\mathrm{Mod}^{+,\theta}_{\Bun_{M}^{\alpha}}$ described above. If we consider the effective Cartier divisor $D$ in $X_{S}$ determined by the image of this $S$-point under the map $\pi_{M}: \mathrm{Mod}^{+,\theta}_{\Bun_{M}} \ra \Div^{(\theta)}$, then the $P$- and $P^-$-reductions of $\mathcal F_G$ are both defined and transverse away from $D$, and together define a meromorphic $M$-reduction $\beta$ of $\mathcal F_G$ to $\mathcal F_G^1 = \mathcal F_M^1\times^M G$. The triple $(\mathcal{F}_{G},\mathcal{F}_{M}^{2},\mathcal{F}_{M}^{1},\beta_{M},\beta)$ in turn defines a point of $\tilde{Z}^{\theta}_{\Bun_{M}^{\alpha}}$. 
\end{proof}

In summary, for every $\theta \in \Lambda_{G,P}^{\mathrm{pos}}$ and $\alpha \in \Lambda_{G,P}$, we have a commutative diagram of the form 
\[ \begin{tikzcd} 
\tilde{Z}^{\theta}_{\Bun_{M}^{\alpha}} \arrow[hookrightarrow]{r} \arrow[d] & \wt{\Bun}^{\theta + \alpha}_{P} \times_{\Bun_{G}} \Bun^{\alpha}_{P^{-}} \arrow[r] \arrow[d] & \wt{\Bun}^{\theta + \alpha}_{P} \arrow[d,"\wt{\mf{p}}^{\theta + \alpha}_{P}"] \\
\mathrm{Mod}^{+,\theta}_{\Bun_{M}^{\alpha}} \arrow[dr,"h_{M}^{\ra}"] & \Bun_{P^{-}}^{\alpha} \arrow[r,"\mf{p}^{\alpha}_{P^{-}}"] \arrow[d,"\mf{q}^{\alpha}_{P^{-}}"] &  \Bun_{G} \\
& \Bun_{M}^{\alpha} & 
\end{tikzcd}, \]
If we fix $\theta \in \Lambda_{G,P}^{\mathrm{pos}}$ and let $\alpha$ vary over all $\Lambda_{G,P}$ then these diagrams patch the bigger diagram
\[ \begin{tikzcd} 
\tilde{Z}^{\theta}_{\Bun_{M}} \arrow[hookrightarrow]{r} \arrow[d] & \wt{\Bun}_{P} \times_{\Bun_{G}} \Bun_{P^{-}} \arrow[r] \arrow[d] & \wt{\Bun}_{P} \arrow[d,"\wt{\mf{p}}_{P}"] \\
\mathrm{Mod}^{+,\theta}_{\Bun_{M}} \arrow[dr,"h_{M}^{\ra}"] & \Bun_{P^{-}} \arrow[r,"\mf{p}_{P^{-}}"] \arrow[d,"\mf{q}_{P^{-}}"] &  \Bun_{G} \\
& \Bun_{M} & 
\end{tikzcd} \]
together. As we will see later, the $v$-stack $\tilde{Z}^{\theta}_{\Bun_{M}}$ will be understandable in terms of the affine Grassmannian. We would like to say that it is related to $\wt{\Bun}_{P}$ by a cohomologically smooth map, so that we might reduce to questions concerning Verdier duality on $\wt{\Bun}_{P}$ to the analogous questions on $\tilde{Z}^{\theta}_{\Bun_{M}}$, where they will be more approachable. 

We know that the map $\tilde{Z}^{\theta}_{\Bun_{M}} \hookrightarrow \wt{\Bun}_{P} \times_{\Bun_{G}} \Bun_{P^{-}}$ is an open immersion, by combining the previous two Lemmas. Therefore, we would like to say that the map 
\[ \wt{\Bun}_{P} \times_{\Bun_{G}} \Bun_{P^{-}} \ra \wt{\Bun}_{P} \]
is cohomologically smooth. However, this is not true because the map $\mf{p}_{P^{-}}$ is not cohomologically smooth. Fortunately, we do have good control over the smoothness properties of $\mf{p}_{P^{-}}$ by invoking the Jacobian criterion of Fargues-Scholze. In particular, we define the following.
\begin{definition}
Consider the Lie algebra $\mf{u}$ of $U$ viewed as an $M$-module. We define $\Bun_{M}^{\mathrm{sm}}$ to be the locus parametrizing, for $S \in \Perf$, $M$-bundles $\mathcal{F}_{M}$ on $X_{S}$ such that the vector bundle $\mathcal{F}_{M} \times^{M,\mathrm{Ad}} \mf{u}$ has strictly positive slopes, where $\mf{u}$ is the Lie algebra of $U$. We let $\Bun_{P^{-}}^{\mathrm{sm}}$ denote the preimage of $\Bun_{M}^{\mathrm{sm}}$ under the map $\mf{q}_{P^{-}}$.
\end{definition}
We note that $\Bun_{P^{-}}^{\mathrm{sm}}$ (resp. $\Bun_{M}^{\mathrm{sm}}$) are open substacks of $\Bun_{P^{-}}$ (resp. $\Bun_{M}$) using the upper semi-continuity of the slope polygon. We have the following result.  
\begin{proposition}{\label{prop: DrinfeldSimpson}}
The natural map $\Bun_{P^{-}}^{\mathrm{sm}} \ra \Bun_{G}$ is a cohomologically smooth map of Artin $v$-stacks, and its topological image is all of $\Bun_{G}$. 
\end{proposition}
\begin{proof}
The smoothness of the map $\Bun_{P^{-}}^{\mathrm{sm}} \ra \Bun_{G}$ follows from the Jacobian criterion of Fargues and Scholze. In particular, if we are given an $S$-point of $\Bun_{G}$ corresponding to a $G$-bundle $\mathcal{F}_{G}$ then the fibers of $\mf{p}_{P^{-}}$ parametrize $P^{-}$-structures on the $G$-bundle $\mathcal{F}_{G}$. If we consider the smooth projective variety $Z = \mathcal{F}_{G}/P^{-} \ra X_{S}$ and write $\mathcal{M}_{Z}$ for the moduli space parametrizing, for $T \in \Perf_{S}$, sections $s: X_{T} \ra \mathcal{F}_{G}/P^{-}$ over $X_{S}$ then $\mathcal{M}_{Z} \ra S$ is precisely the fiber of $\mf{p}_{P^{-}}$ over $S \ra \Bun_{G}$.  

If $T_{Z/X_{S}}$ denotes the relative tangent bundle then the Jacobian criterion \cite[Theorem~IV.4.2]{FarguesScholze} tells us that the open subspace $\mathcal{M}_{Z}^{\mathrm{sm}} \subset \mathcal{M}_{Z}$ parametrizing sections $s$ such that $s^{*}(T_{Z/X_{S}})$ has strictly positive slopes is a cohomologically smooth locally spatial diamond. We now recall that that if $s$ corresponds to a $P^{-}$-structure $\mathcal{F}_{P^{-}}$ on $\mathcal{F}_{G}$, we have an identification
\[ s^{*}(T_{Z/X_{S}}) \simeq \mathcal{F}_{P^{-}} \times^{P^{-}} \mf{g}/\mf{p}^{-}. \] 
The Jordan-H\"older factors of $\mf{g}/\mf{p}^{-}$ as a $P^{-}$-modules all have trivial $U^{-}$-action and therefore are $M$-modules. Moreover, these coincide with the Jordan-H\"older factors of $\mf{u}$. This establishes the smoothness claim.

In order to establish surjectivity of the map $\Bun_{P^{-}}^{\mathrm{sm}} \ra \Bun_{G}$, we will argue as in Drinfeld-Simpson \cite{DS}. Under the isomorphism, $B(M) \simeq |\Bun_{M}|$, we can describe the subset $B(M)^{\mathrm{sm}}$ corresponding to $\Bun_{M}^{\mathrm{sm}}$ more concretely. In particular, $b_{M} \in B(M)^{\mathrm{sm}}$ if and only if the numbers $d_{i}(b_{M}) := \langle \hat{\alpha}_{i},\nu_{b_{M}} \rangle$ are strictly less than $0$ for $i \in \mathcal{J}_{G,P}$ (Recall the minus sign when passing between isocrystals and $G$-bundles, as in Remark \ref{rem: confusingslopes}). This follows from considering the action of the protorus $\mathbb{D}$ on $\mathcal{F}_{b_{M}} \times^{M} \mf{u}$. We need to show that, for any $G$-bundle $\mathcal{F}_{G}$ on $X$, there exists a $P^{-}$-structure $\mathcal{F}_{P^{-}}$ on $\mathcal{F}_{G}$ such that Kottwitz element $b_{M}$ attached to $\mathcal{F}_{P^{-}} \times^{P^{-}} M$ satisfies $-d_{i}(b_{M}) > 0$. We prove the stronger claim that in fact for any $N \geq 0$ one can assume one has that $-d_{i}(b_{M}) > N$.   

Given a fixed $\mathcal{F}_{G}$, by \cite[Theorem~7.1]{GtorseurFarg} (See also \cite[Theorem~6.5]{GroupSchemesoverFFAns}), there exists a modification 
\[ \mathcal{F}_{G} \dashrightarrow \mathcal{F}_{G}^{0}, \]
to the trivial $G$-bundle. It follows from the fact that $G/P^{-}$ is proper that one has a bijection between $P^{-}$-structures on $\mathcal{F}_{G}^{0}$ and $P^{-}$-structures on $\mathcal{F}_{G}$. Moreover, this will only change the invariants $d_{i}(b_{M})$ up to a bounded amount determined by the meromorphy of the modification. Therefore, we can assume that $\mathcal{F}_{G}$ is the trivial $G$-bundle. By choosing a covering $\tilde{G} \ra G$ where $\tilde{G}$ is simply connected, we can reduce to the case where $\Bun_{G}$ is connected, using \cite[Corollary~IV.1.23]{FarguesScholze}. Given a fixed $N \in \mathbb{N}_{\geq 0}$, we can define $\Bun_{P^{-}}^{\mathrm{sm},> N}$ to be the preimage of the open subset in $\Bun_{M}^{\mathrm{sm},> N}$ defined by the elements $b_{M} \in B(M)_{\mathrm{sm}}$ such that $-d_{i}(b_{M}) > N$. We note that this subset is always non-empty, as is exhibited by choosing a sufficiently $G$-antidominant element inside $\Mcoinv \simeq B(M)_{\mathrm{basic}}$. The natural map
\[ \Bun_{P^{-}}^{\mathrm{sm},> N} \ra \Bun_{G} \]
is cohomologically smooth by the above discussion, and therefore universally open by \cite[Proposition~23.1]{Ecod}. In particular, its image contains an open stratum. These are precisely the basic strata by \cite[proof of Corollary IV.1.23]{FarguesScholze} (or the more precise result \cite{Vi}), and as $G$ is simply connected this is only the stratum of the trivial $G$-bundle.
\end{proof}

Recall that $Z^\theta_{\Bun_M}\subset \tilde{Z}^\theta_{\Bun_M}$ is the preimage of $\Bun_P\subset \wt{\Bun}_P$. We deduce the following corollary.

\begin{corollary}{\label{cor: ZthetaBunMissmooth}}
For all $\theta \in \Lambda_{G,P}^{\mathrm{pos}}$, the moduli space $Z_{\Bun_{M}^{\mathrm{sm}}}^{\theta}$ is cohomologically smoooth. 
\end{corollary}

The map $\wt{\Bun}_{P} \times_{\Bun_{G}} \Bun^{\mathrm{sm}}_{P^{-}} \ra \wt{\Bun}_{P}$ has the correct shape for an atlas that would allow us to study $\wt{\Bun}_{P}$. However, it is only the open subset $\tilde{Z}^{\theta}_{\Bun_{M}^{\mathrm{sm}}} \subset \wt{\Bun}_{P} \times_{\Bun_{G}} \Bun^{\mathrm{sm}}_{P^{-}}$ that will admit an interpretation in terms of the affine Grassmannian. Fortunately, it is still large enough. For the moment, we will establish this only on the open stratum $Z_{\Bun_{M}^{\mathrm{sm}}}^{\theta}$, but later a similar result will be proved that also includes the boundary of the Drinfeld compactification.

\begin{proposition}{\label{prop: Zastavauniformizes}}
The map
\[
\bigsqcup_\theta Z^{\theta}_{\Bun_{M}^{\mathrm{sm}}} \ra \Bun_{P}
\]
is cohomologically smooth and surjective.
\end{proposition}
\begin{proof}
As discussed above, since the map $\tilde{Z}^{\theta}_{\Bun_{M}} \subset \wt{\Bun}_{P} \times_{\Bun_{G}} \Bun_{P^{-}}$ is an open immersion, the cohomological smoothness follows from the previous Proposition. For the surjectivity, it suffices to check this at the level of points since cohomologically smooth maps are universally open. Concretely, for any geometric point $\Spa(C,C^+)$ and any $P$-bundle $\mathcal F_P$ on $X_C$, we have to show that the associated $G$-bundle $\mathcal F_P\times^P G$ admits a $P^-$-structure $\mathcal F_{P^-}$ that is generically transverse to the $P$-structure (inducing an $M$-structure), and so that the $M$-bundle $\mathcal F_{P^-}\times^{P^-} M$ defines a point of $\Bun_M^{\mathrm{sm}}$.

But in the previous proposition, we proved that any $G$-bundle admits arbitarily instable $P^-$-reductions. Any such $P^-$-reduction will do, as it will automatically be generically transverse to any given $P$-structure once it is sufficiently instable.
\end{proof}

We end this section with a few remarks about the boundary strata of $\tilde{Z}^\theta_{\Bun_M}$. For an element $\theta' \in \Lambda_{G,P}^{\mathrm{pos}}$, we write $\phantom{}_{\theta'}\tilde{Z}^{\theta}_{\Bun_{M}}$ for the preimage of the locally closed substack $\phantom{}_{\theta'}\wt{\Bun}_{P}$ under the map $\tilde{Z}^{\theta}_{\Bun_{M}} \ra \wt{\Bun}_{P}$. We write $\pi_{P}: \tilde{Z}^{\theta}_{\Bun_{M}} \ra \mathrm{Mod}^{+,\theta}_{\Bun_{M}}$ for the natural forgetful map. This map has a section
\[ \mf{s}^{\theta}: \mathrm{Mod}^{+,\theta}_{\Bun_{M}} \ra \tilde{Z}^{\theta}_{\Bun_{M}}, \]
which sends $(\mathcal{F}_{M}^{2},\mathcal{F}_{M}^{1},\beta_{M})$ to $(\mathcal{F}_{M}^{1} \times^{M} G,\mathcal{F}_{M}^{2},\mathcal{F}_{M}^{1},\beta_{M},\beta^{0})$, where $\beta^{0}$ is the modification induced by the identity map.
\begin{lemma}
The stack $\phantom{}_{\theta'}\tilde{Z}^{\theta}_{\Bun_{M}}$ is empty unless $\theta'\leq \theta$. If $\theta = \theta'$ the map $\mf{s}^{\theta}: \mathrm{Mod}^{+,\theta}_{\Bun_{M}} \ra \tilde{Z}^{\theta}_{\Bun_{M}}$ induces an isomorphism onto the closed substack $\phantom{}_{\theta}\tilde{Z}^{\theta}_{\Bun_{M}}$. 
\end{lemma}
\begin{proof}
Suppose we are given an $S$-point $(\mathcal{F}_{G},\mathcal{F}_{M}^{2},\mathcal{F}_{M}^{1},\beta_{M},\beta)$ of $\phantom{}_{\theta'}\tilde{Z}^{\theta}_{\Bun_{M}}$. Recall that the map $\tilde{Z}^{\theta}_{\Bun_{M}} \ra \wt{\Bun}_{P} \xrightarrow{\mf{t}_{P}} \ol{\Bun}_{P} $ is remembering the composition 
\[
\ol{\kappa}^{\mc{V}}: (\mc{V}^{[P,P]})_{\mathcal{F}_{M^{\mathrm{ab}}}^{2}} \xrightarrow{\beta^{\mc{V}^{[P,P]}}_{M^{\mathrm{ab}}}} (\mc{V}^{[P,P]})_{\mathcal{F}_{M^{\mathrm{ab}}}^{1}} \hookrightarrow \mathcal{V}_{\mathcal{F}_{G}^{1}} \xdashrightarrow{(\beta^{\mc{V}})^{-1}} \mathcal{V}_{\mathcal{F}_{G}} 
\]
for all $G$-modules $\mathcal V$.  Moreover, since $\phantom{}_{\theta'}\tilde{Z}^{\theta}_{\Bun_{M}}$ maps to $\phantom{}_{\theta'}\ol{\Bun}_{P}$, by definition, we have that the cokernel of the map $\ol{\kappa}^{\mc{V}}$ has torsion of length equal to $\langle \mathcal V^{[P,P]}, \theta' \rangle$.

On the other hand, by point (5) in the definition of $\tilde{Z}^\theta_{\Bun_M}$, we can compose $\ol{\kappa}^{\mc{V}}$ with the well-defined surjective map
\[
\mathcal{V}_{\mathcal{F}_{G}}\xrightarrow{\beta^{\mc{V}}} \mathcal{V}_{\mathcal F_G^1}\xdashrightarrow{\kappa^{-,\mathcal V}} (\mathcal V_{U^-})_{\mathcal F_M^1}.
\]
If we start with the $M^{\mathrm{ab}}$-module $\mathcal W=\mathcal V^{[P,P]}$ and set $\mathcal V=H^0(G/P^-,\mathcal W)$, there is a natural map $\mathcal V_{U^-}\to \mathcal W$ and we get the composite map
\[
\mc{W}_{\mathcal F_{M^{\mathrm{ab}}}^2}\xrightarrow{\beta^{\mc{W}}_{M^{\mathrm{ab}}}} \mc{W}_{\mathcal{F}_{M^{\mathrm{ab}}}^{1}}\to (\mc{V}_{U^-})_{\mathcal F_M^1}\to \mc{W}_{\mathcal F_M^1}
\]
which is really just the injective map $\beta^{\mc{W}}_{M^{\mathrm{ab}}}$, and has cokernel given by a torsion bundle of length $\langle \mathcal W,\theta\rangle$. The torsion part of the cokernel of $\ol{\kappa}^{\mc{V}}$ injects into the cokernel of this composite map. Thus,
\[
\langle \mathcal V^{[P,P]},\theta'\rangle \leq \langle \mathcal W,\theta\rangle = \langle \mathcal V^{[P,P]},\theta\rangle
\]
for all $G$-modules $\mathcal V$, which implies $\theta'\leq \theta$.

If $\theta = \theta'$, we can saturate $\ol{\kappa}^{\mc{V}}$ to get a well-defined map
\[
(\mathcal V^{[P,P]})_{\mathcal F_{M^{\mathrm{ab}}}^1}\to \mathcal V_{\mathcal F_G}
\]
whose cokernel is a vector bundle, yielding a $P$-reduction of $\mathcal F_G$ that is transverse to the $P^-$-reduction; thus $\mathcal F_G=\mathcal F_G^1$.
\end{proof}

It would be very nice to say that $\tilde{Z}^{\theta}_{\Bun_{M}}$ surjects onto the locus $\phantom{}_{\leq \theta}\wt{\Bun}_{P}$. However, this is not true. In particular, it only uniformizes the locus of $\phantom{}_{\theta}\wt{\Bun}_{P}$ coming from split reductions, as one easily verifies from the construction.
\begin{lemma}{\label{lemma: imageofthesection}}
The composition 
\[ \mf{s}^{\theta}: \mathrm{Mod}^{+,\theta}_{\Bun_{M}} \xrightarrow{\simeq } \phantom{}_{\theta}\tilde{Z}^{\theta}_{\Bun_{M}} \ra \phantom{}_{\theta}\wt{\Bun}_{P} \simeq \Bun_{P} \times_{\Bun_{M}} \mathrm{Mod}^{+,\theta}_{\Bun_{M}} \]
is given by the map $\alpha \times \mathrm{id}$, where
\[ \alpha: \mathrm{Mod}^{+,\theta}_{\Bun_{M}} \xrightarrow{h_{M}^{\ra}} \Bun_{M} \ra \Bun_{P}, \]
and the last map is the section of $\mf{q}_{P}$ given by sending $\mathcal{F}_{M}$ to $\mathcal{F}_{M} \times^{M} P$.
\end{lemma}
Fortunately, after trivializing the bundle $\mathcal{F}_{M}^{1}$, we can arrange cohomologically smooth locally that $\tilde{Z}^{\theta}_{\Bun_{M}}$ uniformizes $\phantom{}_{\leq \theta}\wt{\Bun}_{P}$. To see this, we will need to use the factorization structure on the Zastava spaces.

\subsection{Factorization}

In this section, we will stick to studying the Zastava space $\tilde{Z}^{\theta}$, which we recall is the space studied in the previous section, but the bundle $\mathcal{F}_{M}^{1}$ is equal to the trivial $M$-bundle. This space maps naturally to
\[
\mathrm{Mod}^{+,\theta}_M := \mathrm{Mod}^{+,\theta}_{\Bun_M}\times_{\Bun_M} \ast.
\]
These spaces exhibit a nice factorization structure.

In particular, let $\theta = \theta_{1} + \theta_{2}$ with $\theta_{i} \in \Lambda_{G,P}^{\mathrm{pos}}$. We write
\[ (\Div^{(\theta_{1})} \times \Div^{(\theta_{2})})_{\disj} \subset \Div^{(\theta_{1})} \times \Div^{(\theta_{2})}  \]
for the open subset where the corresponding Cartier divisors are disjoint. We note that there is a natural \'etale map 
\[ (\Div^{(\theta_{1})} \times \Div^{(\theta_{2})})_{\disj} \ra \Div^{(\theta)} \] 
given by taking the union of Cartier divisors. We have the following factorization property.
\begin{proposition}{\label{prop: factorizationproperty}}
With notation as above, there is a natural isomorphism 
\[ (\Div^{(\theta_{1})} \times \Div^{(\theta_{2})})_{\disj} \times_{\Div^{(\theta)}} \tilde{Z}^{\theta} \simeq (\Div^{(\theta_{1})} \times \Div^{(\theta_{2})})_{\disj} \times_{\Div^{(\theta_{1})} \times \Div^{(\theta_{2})}} (\tilde{Z}^{\theta_{1}} \times \tilde{Z}^{\theta_{2}}) \]
lying over a natural isomorphism 
\[ (\Div^{(\theta_{1})} \times \Div^{(\theta_{2})})_{\disj} \times_{\Div^{(\theta)}} \mathrm{Mod}_{M}^{+,\theta} \simeq (\Div^{(\theta_{1})} \times \Div^{(\theta_{2})})_{\disj} \times_{\Div^{(\theta_{1})} \times \Div^{(\theta_{2})}} (\mathrm{Mod}_{M}^{+,\theta_{1}} \times \mathrm{Mod}_{M}^{+,\theta_{2}}). \]
such that the sections $\mf{s}^{\theta}$ and $\mf{s}^{\theta_{1}} \times \mf{s}^{\theta_{2}}$ agree after base-changing and applying these identifications.
\end{proposition}
\begin{proof}
For $S \in \Perf$, consider an $S$-point of $(\mathcal{F}_{G},\mathcal{F}_{M},\beta_{M},\beta)$ of $\tilde{Z}^{\theta} \times_{\Div^{(\theta)}} (\Div^{(\theta_{1})} \times \Div^{(\theta_{2})})$ lying over the open subspace $(\Div^{(\theta_{1})} \times \Div^{(\theta_{2})})_{\disj}$, and let $D_{S}^{1}$ and $D_{S}^{2}$ be the effective Cartier divisors in $X_{S}$ corresponding to $\theta_{1}$ and $\theta_{2}$, respectively. We let $X_{S} \setminus D_{S}^{i} =: X_{S}^{i}$ for $i = 1,2$, and set $X_{S}^{0} := X_{S} \setminus (\bigcup_{i = 1}^{2} D_{S}^{i})$. For $i = 1,2$, by using Beauville-Laszlo gluing \cite[Lemma~5.2.9]{SW} on $X_{S}$ we can define $G$-bundles $\mathcal{F}_{G}^{i}$ by modifying the trivial $G$-bundle via $\beta$ at just the formal completions of the $D_{S}^{i}$. By construction, we have modifications $\beta^{i}: \mathcal{F}_{G}^{i} \dashrightarrow \mathcal{F}_{G}^{0}$ which are isomorphisms away from the $D_{i}$. Similarly, from the $S$-point  $(\mathcal{F}_{M},\beta_{M})$ in $\mathrm{Mod}^{+,\theta}_{M} \times_{\Div^{(\theta)}} (\Div^{(\theta_{1})} \times \Div^{(\theta_{2})})_{\disj}$, we obtain $S$-points $(\mathcal{F}_{M}^{i},\beta_{M}^{i}) \in \mathrm{Mod}^{+,\theta_{i}}_{M}$ for $i =1,2$, and in turn points $(\mathcal{F}_{G}^{i},\mathcal{F}_{M}^{i},\beta_{M}^{i},\beta^{i}) \in \tilde{Z}^{\theta_{i}}$, as desired. The converse direction proceeds analogously. 
\end{proof}
To simplify the notation, if we are given a stack $Z$ living over $\Div^{(\theta_{1})} \times \Div^{(\theta_{2})}$, we write $Z_{\disj}$ for its pullback to the open subspace $(\Div^{(\theta_{1})} \times \Div^{(\theta_{2})})_{\disj}$ so that the factorization property reads $(\Div^{(\theta_{1})} \times \Div^{(\theta_{2})})_{\disj} \times_{\Div^{(\theta)}} \tilde{Z}^{\theta} \simeq (Z^{\tilde{\theta}_{1}} \times Z^{\tilde{\theta}_{2}})_{\disj}$.

Fix an element $\theta' \in \Lambda_{G,P}^{\mathrm{pos}}$. As in the previous section, the space $\tilde{Z}^{\theta'}$ has a stratification given by $\phantom{}_{\theta}\tilde{Z}^{\theta'}$ for $\theta \leq \theta'$ given by fixing the torsion cokernel of the embeddings
\[ (\mc{V}^{U})_{\mathcal{F}_{M}} \xrightarrow{\beta^{\mc{V}^U}_{M}} (\mc{V}^{U})_{\mathcal{F}_{M}^{0}} \hookrightarrow \mathcal{V}^{\mc{V}}_{\mathcal{F}_{G}^{0}} \xdashrightarrow{(\beta^{\mc{V}})^{-1}} \mathcal{V}^{\mc{V}}_{\mathcal{F}_{G}}, \]
where the open stratum corresponding to $\theta = 0$ is precisely $Z^{\theta'}$. 

Similarly, we write $\phantom{}_{\leq \theta}\tilde{Z}^{\theta'}$ for the corresponding open subspaces given by bounding the torsion by $\theta$. If we are given a partition $\theta_{1}' + \theta_{2}' = \theta'$ then we write $(\tilde{Z}^{\theta_{1}'} \times \tilde{Z}^{\theta_{2}'})_{\disj, \leq \theta}  \subset (\tilde{Z}^{\theta_{1}'} \times \tilde{Z}^{\theta_{2}'})_{\disj}$ for the open subspace corresponding to $(\Div^{(\theta_{1})} \times \Div^{(\theta_{2})})_{\disj} \times_{\Div^{(\theta)}} \phantom{}_{\leq \theta}\tilde{Z}^{\theta'}$ under the factorization isomorphism. One easily sees the following Corollary.
\begin{corollary}
For $\theta, \theta' \in \Lambda_{G,P}^{\mathrm{pos}}$ with $\theta' \geq \theta$, and a partition $\theta' = \theta'_{1} + \theta'_{2}$, the subspace $(\tilde{Z}^{\theta_{1}'} \times \tilde{Z}^{\theta_{2}'})_{\disj, \leq \theta} \subset (\tilde{Z}^{\theta_{1}'} \times \tilde{Z}^{\theta_{2}'})_{\disj}$ is stratified by 
\[ (\phantom{}_{\theta_{1}}\tilde{Z}^{\theta_{1}'} \times \phantom{}_{\theta_{2}}\tilde{Z}^{\theta_{2}'})_{\disj}  \]
for $\theta_{1} \leq \theta_{1}'$ and $\theta_{2} \leq \theta_{2}'$ such that $\theta_{1} + \theta_{2} = \theta$.
\end{corollary}
We will make use of the following related result.
\begin{corollary}{\label{cor: facotrizationwithstratifications}}
For $\theta, \theta' \in \Lambda_{G,P}^{\mathrm{pos}}$ with $\theta' \geq \theta$, we have an open immersion 
\[  (\tilde{Z}^{\theta} \times Z^{\theta' - \theta})_{\disj} \hookrightarrow (\Div^{(\theta)} \times \Div^{(\theta' - \theta)})_{\disj} \times_{\Div^{(\theta')}} \phantom{}_{\leq \theta}\tilde{Z}^{\theta'} \]
\end{corollary}
The factorization property is an expression of the fact that the Zastava space $\tilde{Z}^{\theta}$ is really determined by local information around the corresponding point in $\Div^{\theta}$. One consequence of this "locality" is a $\mathbb{G}_{m}$-action on the Zastava space, studied in the next section.

\subsection{Hyperbolic localization on Zastava spaces}

In \S \ref{sec: Zastava uniformizes Drinfeld}, we saw that the map
\[ h_{M}^{\ra}: \tilde{Z}^{\theta} \ra \mathrm{Mod}^{+,\theta}_{M} \]
admits a section $\mf{s}^{\theta}$. In this section, we construct (\'etale locally on $\Div^{(\theta)}$) a $\mathbb{G}_{m}$-action on $\tilde{Z}^{\theta}$ which contracts to this section. 
Consider $\Gr_{G,\Div^{(\theta)}} \ra \Div^{(\theta)}$, the $B_{dR}^{+}$ Beilinson-Drinfeld Grassmannian over $\Div^{(\theta)}$, as in \cite[Definition~VI.1.8]{FarguesScholze} defined using the loop groups $L^{+}_{\Div^{(\theta)}}G$ and $L_{\Div^{(\theta)}}G$. For $d \geq 0$, we will also consider the mirror curves
\[ \Div^{(d)}_{\mathcal{Y}} := \Spd(\Breve{E})^{d}/S_{d}, \]
which, for $S \in \Perf$, parametrizes Cartier divisors in the open adic Fargues-Fontaine curve $\mathcal{Y}_{S}$, by \cite{FarguesScholze}.\footnote{Our notation here differs from that in \cite{FarguesScholze}. In particular, what we denote as $\mathcal{Y}_{S}$ is denoted there as $Y_{S}$; we emphasize that we do not include the locus where $\pi = 0$.} Similarly, for a finite extension $E'/E$, we write $\Div^{(d)}_{\mathcal{Y},E'}$ for the corresponding object constructed from $E'$.

We recall that we have an \'etale covering (in fact a local isomorphism) $\mathcal{Y}_{S} \ra \mathcal{X}_{S}$ which in turn induces a map
\[ \Div^{(d)}_{\mathcal{Y}} \ra \Div^{(d)} \]
of mirror curves that is \'etale locally split. More precisely, restricting to the open subspace $\Div^{(d),\prime}$ of $\Div^{(d)}_{\mathcal Y}$ where no point is a proper Frobenius translate of another point, the map to $\Div^{(d)}$ is \'etale.

For $\theta \in \Lambda_{G,P}^{\mathrm{pos}}$ with $\theta = \sum_{i \in \mathcal{J}_{G,P}} n_{i}\alpha_{i}$, we set 
\[\Div^{(\theta)}_{\mathcal Y} := \prod_{i \in \mathcal{J}_{G,P}} \Div^{(n_{i}),\prime}_{\mathcal Y,E_{i}}\supset \Div^{(\theta),\prime} := \prod_{i \in \mathcal{J}_{G,P}} \Div^{(n_{i}),\prime}_{E_{i}},\] where $E_{i}$ denotes the reflex field of the Galois orbit of simple coroots corresponding to $i \in \mathcal{J}_{G,P}$. Thus, $\Div^{(\theta),\prime}\subset \Div^{(\theta)}_{\mathcal Y}$ is an open subset with an \'etale surjective map to $\Div^{(\theta)}$.

We write $\Gr_{G,\Div^{(\theta)}_{\mathcal{Y}}} \ra \Div^{(\theta)}_{\mathcal{Y}}$ for the associated Beilinson-Drinfeld Grassmannian given as the \'etale sheafification of the quotient: $L_{\Div^{(\theta)}_{\mathcal{Y}}}G/L^{+}_{\Div^{(\theta)}_{\mathcal{Y}}}G$.

We choose a cocharacter $\lambda: \mathbb{G}_{m} \ra G$ such that the associated parabolic $P_{\lambda}$ under the dynamical method is $P^{-}$. We get an action of $\mathbb{G}_{m}$ on $\Gr_{G,\Div^{(\theta)}_{\mathcal{Y}}}$ by considering the composition
\[ \mathbb{G}_{m} \ra L^{+}_{\Div^{(\theta)}_{\mathcal{Y}}}\mathbb{G}_{m} \xrightarrow{L^{+}\lambda} L^{+}_{\Div^{(\theta)}_{\mathcal{Y}}}G \]
and using the left action of $L^{+}_{\Div^{(\theta)}_{\mathcal{Y}}}$. Here the first map is given by the Teichm\"uller character, where we note that this is only well-defined on the curve $\mathcal Y_S$ and not on the curve $X_S$, since on the latter we quotient out the relevant Witt vector ring by Frobenius.

We let $\Gr_{U^{-},\Div^{(\theta)}_{\mathcal{Y}}} \ra \Div^{(\theta)}_{\mathcal{Y}}$ (resp. $\Gr_{U^{-},\Div^{(\theta)}} \ra \Div^{(\theta)}$) denote the Beilinson-Drinfeld affine Grassmannian attached to $U^{-}$ over $\Div^{(\theta)}_{\mathcal{Y}}$ (resp. $\Div^{(\theta)}$). We recall that the projection map has a unit section corresponding to the trivial $L^{+}_{\Div^{(\theta)}}U^{-}$ (resp . $L^{+}_{\Div^{(\theta)}_{\mathcal{Y}}}U^{-}$) orbit.
\begin{proposition}
The natural inclusion $\Gr_{U^{-},\Div^{(\theta)}_{\mathcal{Y}}} \hookrightarrow \Gr_{G,\Div^{(\theta)}_{\mathcal{Y}}}$ is a locally closed immersion. For $\lambda$ as above, the $\mathbb{G}_{m}$-action on $\Gr_{G,\Div^{(\theta)}_{\mathcal{Y}}}$ preserves the locally closed subspace $\Gr_{U^{-},\Div^{(\theta)}_{\mathcal{Y}}}$ and extends to an action of the multiplicative monoid $\mathbb{A}^{1}$. The $\mathbb{G}_{m}$ fixed points of this action is the closed image of the unit section $\Div^{(\theta)}_{\mathcal{Y}} \ra \Gr_{U^{-},\Div^{(\theta)}_{\mathcal{Y}}}$.
\end{proposition}
\begin{proof}
By \cite[Proposition~VI.3.1]{FarguesScholze}, its proof, and our assumption on $\lambda$, we have that the $\mathbb{G}_{m}$-action on $\Gr_{G,\Div^{(\theta)}_{\mathcal{Y}}}$ preserves the subspace $\Gr_{P^{-},\Div^{(\theta)}_{\mathcal{Y}}} \subset \Gr_{G,\Div^{(\theta)}_{\mathcal{Y}}}$, which is a disjoint union of locally closed subspaces; the induced action follows by taking loop groups of the associated action of $\mathbb{G}_{m}$ on $P$ by conjugation. Moreover, it extends to an action of $\mathbb{A}^{1}$ such that its fixed points are $\Gr_{M,\Div^{(\theta)}_{\mathcal{Y}}}$. It follows by \cite[Proposition~20.3.7]{SW} and the fact that $U^{-}$ is preserved under conjugation by $\mathbb{G}_{m}$ that  $\Gr_{U^{-},\Div^{(\theta)}_{\mathcal{Y}}} \subset \Gr_{G,\Div^{(\theta)}_{\mathcal{Y}}}$ is also a locally closed subspace, which is preserved under a $\mathbb{G}_{m}$-action that extends to $\mathbb{A}^{1}$. Its fixed points are precisely the unit section, by virtue of the fact that the $\mathbb{G}_{m}$-action on $U^{-}$ by conjugation contracts to the identity.  
\end{proof}

The Grassmannian $\Gr_{U^{-},\Div^{(\theta)}}$ will parametrize modifications of $G$-bundles $\beta: \mathcal{F}_{G} \dashrightarrow \mathcal{F}_{G}^{0}$ satisfying condition (5) of Definition \ref{def: Zastavaspace}. Therefore, we have a natural map 
\[ \tilde{Z}^{\theta} \hookrightarrow \Gr_{U^{-},\Div^{(\theta)}} \times_{\Div^{(\theta)}} \mathrm{Mod}^{+,\theta}_{M} \] 
mapping to the locus of the right-hand side where condition (4) of Definition \ref{def: Zastavaspace} is satisfied. It follows that this is a closed subspace by an application of Lemma \ref{lemma: uppersemicont}. Moreover, under this embedding the section $\mf{s}^{\theta}$ is given by the base-change of the unit section of $\Gr_{U^{-},\Div^{(\theta)}} \ra \Div^{(\theta)}$.

Since the map $\mathcal{Y}_{S} \ra \mathcal{X}_{S}$ is a local isomorphism, it induces an isomorphism on formal completions at Cartier divisors as long as no two distinct points of the Cartier divisor on $\mathcal Y_S$ map to the same point of $\mathcal X_S$. This is enforced on $\Div^{(\theta),\prime}\subset \Div^{(\theta)}_{\mathcal Y}$, so it follows that we have a Cartesian square 
\[ \begin{tikzcd}
\Gr_{U^{-},\Div^{(\theta)}_{\mathcal{Y}}}\times_{\Div^{(\theta)}_{\mathcal Y}} \Div^{(\theta),\prime} \arrow[r] \arrow[d] & \Div^{(\theta),\prime} \arrow[d] & \\
\Gr_{U^{-},\Div^{(\theta)}} \arrow[r] & \Div^{(\theta)} &,
\end{tikzcd} \]
where the vertical maps are \'etale. If we write $\tilde{Z}^{\theta,\prime}$ for the base-change of $\tilde{Z}^{\theta}$ along the etale covering $\Div^{(\theta),\prime} \ra \Div^{(\theta)}$ then we obtain a closed subspace
\[ \tilde{Z}^{\theta,\prime} \hookrightarrow \Gr_{U^{-},\Div^{(\theta),\prime}} \times_{\Div^{(\theta)}} \mathrm{Mod}^{+,\theta}_{M}. \]

On the target, there is the $\mathbb{G}_{m}$-action constructed above. We claim that this preserves $\tilde{Z}^{\theta,\prime}$, as well as the strata $\phantom{}_{\theta'} \tilde{Z}^{\theta,\prime}$. To see this, note that $\tilde{Z}^{\theta,\prime}(S)$ is defined functorially in the triple $(G\times_E \mathcal Y_S,P\times_E \mathcal Y_S,P^-\times_E \mathcal Y_S)$, and conjugation by $[\lambda]$ defines an action of $\mathbb G_m$ on this triple. It is this action that gets transported through all structures, yielding the action both on $\Gr_{U^{-},\Div^{(\theta),\prime}} \times_{\Div^{(\theta)}} \mathrm{Mod}^{+,\theta}_{M}$ (trivially on the second factor, as $\lambda$ is central in $M$) and on $\tilde{Z}^{\theta,\prime}$, with its intrinsically defined stratification.

\begin{corollary}{\label{cor: GmactiononZastava}}
After base-changing along the \'etale cover $\Div^{(\theta),\prime} \ra \Div^{(\theta)}$, there exists a strata-preserving $\mathbb{G}_{m}$-action on $\tilde{Z}^{\theta,\prime}$ which extends to an action of $\mathbb{A}^{1}$. Its fixed points are precisely the closed image of the base-change of the section $\mf{s}^{\theta}: \mathrm{Mod}^{+,\theta}_{M} \ra \tilde{Z}^{\theta}$.
\end{corollary}

\subsection{Smooth localizations of the Drinfeld compactification}

In this section, we will compile the results of the previous section in order to obtain some nice cohomologically smooth local descriptions of the Drinfeld compactifications in terms of Zastava spaces.

We start with the following lemma, which shows that base-changing $\mathrm{Mod}^{+,\theta}_{\Bun_M}\to \Bun_M$ along $\ast\to \Bun_M$ does not change the geometry much. Recall that this base change is denoted $\mathrm{Mod}^{+,\theta}_M$.

\begin{lemma}
For all $\theta \in \Lambda_{G,P}^{\mathrm{pos}}$, there exists an Artin $v$-stack $T$ admitting cohomologically smooth surjections $T \ra \mathrm{Mod}^{+,\theta}_{M}$ and $T \ra \mathrm{Mod}^{+,\theta}_{\Bun_{M}}$ such that the pullbacks of $\tilde{Z}^{\theta} \ra \mathrm{Mod}^{+,\theta}_{M}$ and $\tilde{Z}^{\theta}_{\Bun_{M}} \ra \mathrm{Mod}^{+,\theta}_{\Bun_{M}}$ to $T$ are isomorphic, compatibly with their stratifications.

In fact, one can arrange that $T^{\mathrm{sm}}:= T\times_{\Bun_M} \Bun_M^{\mathrm{sm}}$ still surjects onto $\mathrm{Mod}^{+,\theta}_{M}$.
\end{lemma}

\begin{proof} It suffices to find a cohomologically smooth Artin $v$-stack $T_0$ with a cohomologically smooth surjective map $T_0\to \Bun_M\times \Div^{(\theta)}$ such that the corresponding $M$-bundle on $X_{T_0}$ is trivial in a neighborhood of the Cartier divisor corresponding to the map $T_0\to \Div^{(\theta)}$. Indeed, the data parametrized by $\mathrm{Mod}^{+,\theta}_{\Bun_M}\to \Bun_M\times \Div^{(\theta)}$ and $\tilde{Z}^\theta_{\Bun_M}\to \Bun_M\times \Div^{(\theta)}$ (as well as conditions defining the stratifications) depend only on the $M$-bundle on the Fargues-Fontaine curve in a neighborhood of the Cartier divisor. Thus, if we define
\[
T = T_0\times_{\Bun_M\times \Div^{(\theta)}} \Mod^{+,\theta}_{\Bun_M},
\]
then $T$ is also isomorphic to the base change $T_0\times_{\Div^{(\theta)}} \Mod^{+,\theta}_M$, so has a natural map to $\Mod^{+,\theta}_M$, and the base changes of the Zastava spaces are also isomorphic. With this choice, $T^{\mathrm{sm}} = T\times_{\Bun_M} \Bun_M^{\mathrm{sm}}$ will continue to surject onto $\Mod^{+,\theta}_M$ (as $T_0\times_{\Bun_M} \Bun_M^{\mathrm{sm}}$ continues to surject onto $\Div^{(\theta)}$).

To construct $T_0$, we can use Beauville-Laszlo gluing as in \cite[Theorem~IV.1.19]{FarguesScholze}, to reduce to finding a cohomologically smooth Artin $v$-stack $T_1$ with a cohomologically smooth surjective map $T_1\to [\ast/M(E)]$ such that the associated $M$-bundle on $X_{T_1}$ is trivial away from a Cartier divisor given by a map $T_1\to \Div^d$. (Then on $T_1\times \Div^{(\theta)}$, the locus where the two Cartier divisors are disjoint will surject onto $T_1$.) This can be reduced to $\GL_n$, and in this case one can take $T_1$ to be the open subset of the space parametrizing fibrewise trivial rank $n$ bundles $\mathcal E$ on $X_S$ together with $n$ global sections of $\mathcal E(1)$ where the resulting map $\mathcal O_{X_S}^n\to \mathcal E(1)$ is generically an isomorphism.
\end{proof}

By combining this with Corollary \ref{cor: ZthetaBunMissmooth}, we obtain the following result which will be used in the sequel. 
\begin{corollary}{\label{cor: Zthetaissmooth}}
For all $\theta \in \Lambda_{G,P}^{\mathrm{pos}}$, the moduli spaces $Z^{\theta}$ and $Z^\theta_{\Bun_M}$ are cohomologically smooth.
\end{corollary}
We fix $\theta \in \Lambda_{G,P}^{\mathrm{pos}}$. For $\theta' \geq \theta$ we consider the cohomologically smooth map 
\[ \phantom{}_{\leq \theta}\tilde{Z}^{\theta'}_{\Bun_{M}^{\mathrm{sm}}} \ra \phantom{}_{\leq \theta}\wt{\Bun}_{P}.\]
Using the previous Lemma, we have a $v$-stack $T$ surjecting onto $\mathrm{Mod}^{+,\theta'}_{M}$ and $\mathrm{Mod}^{+,\theta'}_{\Bun_{M}^{\mathrm{sm}}}$ and an isomorphism
\[ \tilde{Z}^{\theta'} \times_{\mathrm{Mod}^{+,\theta}_{M}} T \simeq \tilde{Z}^{\theta'}_{\Bun_{M}^{\mathrm{sm}}} \times_{\mathrm{Mod}^{+,\theta'}_{\Bun_{M}^{\mathrm{sm}}}} T. \]
In summary, we have a commutative diagram
\[ \begin{tikzcd}
& \phantom{}_{\leq \theta}\tilde{Z}^{\theta'} \times_{\mathrm{Mod}^{+,\theta}_{M}} T \arrow[r] \ar[dr, phantom, "\square"] \arrow[d]  & \phantom{}_{\leq \theta}\tilde{Z}^{\theta'}_{\Bun_{M}^{\mathrm{sm}}} \arrow[r] \arrow[d] & \phantom{}_{\leq \theta}\wt{\Bun}_{P} \arrow[d,"\phantom{}_{\leq \theta}\tilde{\mf{q}}_{P}"] & \\
& T \arrow[r]  & \mathrm{Mod}^{+,\theta'}_{\Bun_{M}^{\mathrm{sm}}} \arrow[r, "h_M^{\la}"]  & \Bun_{M}, &
\end{tikzcd}
\]
where the top horizontal arrows are cohomologically smooth.

We let $T_{\theta,\theta' - \theta} \ra T$ be the base-change of the \'etale map $(\Div^{(\theta)} \times \Div^{(\theta' - \theta)})_{\mathrm{disj}} \ra \Div^{(\theta')}$ along $T \ra \Mod^{+,\theta'}_{M}\to \Div^{(\theta')}$. By the factorization property, we also have an open immersion
\[
\mathrm{Mod}^{+,\theta'}_{M} \times_{\Div^{(\theta')}} (\Div^{(\theta)} \times \Div^{(\theta' - \theta)})_{\mathrm{disj}}  \simeq (\mathrm{Mod}_{M}^{+,\theta} \times \mathrm{Mod}_{M}^{+,\theta' - \theta})_{\disj} \hookrightarrow \mathrm{Mod}_{M}^{+,\theta} \times \mathrm{Mod}_{M}^{+,\theta' - \theta}.
\]
Note that the map $T_{\theta,\theta' - \theta} \ra \mathrm{Mod}^{+,\theta'}_{M} \times_{\Div^{(\theta')}} (\Div^{(\theta)} \times \Div^{(\theta' - \theta)})_{\mathrm{disj}}$ is cohomologically smooth and hence so is the composition $T_{\theta,\theta' - \theta} \ra \mathrm{Mod}_{M}^{+,\theta} \times \mathrm{Mod}_{M}^{+,\theta' - \theta}$ with the previous \'etale map.  We obtain  a diagram as follows.
\begin{equation}{\label{diag: uniformizing Drinfeld}} \begin{tikzcd}
\tilde{Z}^{\theta}  \times Z^{\theta' - \theta} \times_{\mathrm{Mod}_{M}^{+,\theta} \times \mathrm{Mod}_{M}^{+,\theta' - \theta}} T_{\theta,\theta' - \theta} \arrow[r] \arrow[d] & \phantom{}_{\leq \theta}\tilde{Z}^{\theta'} \times_{\mathrm{Mod}^{+,\theta}_{M}} T \arrow[r] \ar[dr, phantom, "\square"] \arrow[d]  & \phantom{}_{\leq \theta}\tilde{Z}^{\theta'}_{\Bun_{M}^{\mathrm{sm}}} \arrow[r] \arrow[d] & \phantom{}_{\leq \theta}\wt{\Bun}_{P} \arrow[d,"\phantom{}_{\leq \theta}\tilde{\mf{q}}_{P}"] & \\
T_{\theta,\theta' - \theta} \arrow[r] & T \arrow[r]  & \mathrm{Mod}^{+,\theta'}_{\Bun_{M}^{\mathrm{sm}}} \arrow[r]  & \Bun_{M}. &
\end{tikzcd}
\end{equation}

Finally, we get the following critical uniformization result for $\phantom{}_{\leq \theta}\wt{\Bun}_{P}$.
\begin{proposition}\label{prop: important uniformation}
For all $\theta \in \Lambda_{G,P}^{\mathrm{pos}}$, the map 
\[
\bigsqcup_{\theta'\geq \theta} (\tilde{Z}^{\theta}  \times Z^{\theta' - \theta}) \times_{\mathrm{Mod}_{M}^{+,\theta} \times \mathrm{Mod}_{M}^{+,\theta' - \theta}} T_{\theta,\theta' - \theta} \ra \phantom{}_{\leq \theta}\wt{\Bun}_{P}
\]
is cohomologically smooth and surjective.  
\end{proposition}
\begin{proof} We already know that it is cohomologically smooth, so we have to see surjectivity. Let us write this surjectivity in terms that are independent of the choice of $T_{\theta,\theta'-\theta}$. Namely, we have the locus inside
\[
\phantom{}_{\leq \theta}\tilde{Z}^{\theta'}_{\Bun_M^{\mathrm{sm}}}\times_{\Div^{(\theta')}} (\Div^{(\theta)}\times \Div^{(\theta'-\theta)})_{\mathrm{disj}}
\]
where the locus of indeterminacy of $\beta_M$ is concentrated on the part of the Cartier divisor corresponding to the summand parametrized by $\Div^{(\theta)}$. This space is actually given by
\[
(\tilde{Z}^\theta_{\Bun_M^{\mathrm{sm}}}\times_{\Bun_M^{\mathrm{sm}}} Z^{\theta'-\theta}_{\Bun_M^{\mathrm{sm}}})_{\mathrm{disj}},
\]
using the factorization structure of the Zastava spaces. There is a natural surjection
\[
(\tilde{Z}^{\theta}  \times Z^{\theta' - \theta}) \times_{\mathrm{Mod}_{M}^{+,\theta} \times \mathrm{Mod}_{M}^{+,\theta' - \theta}} T_{\theta,\theta' - \theta}\to (\tilde{Z}^\theta_{\Bun_M^{\mathrm{sm}}}\times_{\Bun_M^{\mathrm{sm}}} Z^{\theta'-\theta}_{\Bun_M^{\mathrm{sm}}})_{\mathrm{disj}}
\]
and we have to see that the map
\[
\bigsqcup_{\theta'\geq \theta} (\tilde{Z}^\theta_{\Bun_M^{\mathrm{sm}}}\times_{\Bun_M^{\mathrm{sm}}} Z^{\theta'-\theta}_{\Bun_M^{\mathrm{sm}}})_{\mathrm{disj}}\to \phantom{}_{\leq \theta}\wt{\Bun}_P
\]
is surjective.

By induction on $\theta$, it suffices to show that the subspace $\phantom{}_{\theta}\wt{\Bun}_{P}$ lies in the image. To do this, we use Proposition \ref{prop: Zastavauniformizes} to see that the space $\bigsqcup_{\theta'\geq \theta} Z^{\theta' - \theta}_{\Bun_{M}^{\mathrm{sm}}}$ uniformizes the open subspace $\Bun_{P} \subset \wt{\Bun}_{P}$. In fact, given any fixed point of $\Div^{(\theta)}$, even the locus of $\bigsqcup_{\theta'\geq \theta} Z^{\theta' - \theta}_{\Bun_{M}^{\mathrm{sm}}}$ where the Cartier divisor is disjoint from the fixed point of $\Div^{(\theta)}$ will surject. Thus, we see that $\phantom{}_{\theta}\wt{\Bun}_{P} \simeq \mathrm{Mod}^{+,\theta}_{\Bun_{M}} \times_{\Bun_{M}} \Bun_{P}$ lies in the image, as desired.
\end{proof}

The upper left corner of (\ref{diag: uniformizing Drinfeld}) also sits in the following set of Cartesian diagrams 
\[ \begin{tikzcd}
(\tilde{Z}^{\theta}  \times Z^{\theta' - \theta}) \times_{\mathrm{Mod}_{M}^{+,\theta} \times \mathrm{Mod}_{M}^{+,\theta' - \theta}} T_{\theta,\theta' - \theta} \arrow[d] \arrow[r] \ar[dr, phantom, "\square"] & \tilde{Z}^{\theta}  \times Z^{\theta' - \theta} \ar[dr, phantom, "\square"] \arrow[d] \arrow[r] & \tilde{Z}^{\theta} \arrow[d] \\
(\mathrm{Mod}_{M}^{+,\theta} \times Z^{\theta' - \theta}) \times_{\mathrm{Mod}_{M}^{+,\theta} \times \mathrm{Mod}_{M}^{+,\theta' - \theta}} T_{\theta,\theta' - \theta} \ar[dr, phantom, "\square"] \arrow[r] \arrow[d] &  \mathrm{Mod}_{M}^{+,\theta} \times Z^{\theta' - \theta} \arrow[r] \arrow[d] & \mathrm{Mod}^{+,\theta}_{M} \\
T_{\theta,\theta' - \theta} \arrow[r]  & \mathrm{Mod}^{+,\theta}_{M} \times \mathrm{Mod}^{+,\theta' - \theta}_{M} & 
\end{tikzcd} \]
where as noted before the bottom horizontal arrow is cohomologically smooth. Therefore, since $Z^{\theta' - \theta}$ is cohomologically smooth, we have a cohomologically smooth surjection 
\[ (\tilde{Z}^{\theta}  \times Z^{\theta' - \theta}) \times_{\mathrm{Mod}_{M}^{+,\theta} \times \mathrm{Mod}_{M}^{+,\theta' - \theta}} T_{\theta,\theta' - \theta} \ra \tilde{Z}^{\theta} \] 
As a consequence of the above discussion, we have shown the following, which says that $\tilde{Z}^{\theta}$ is a local model for the singularities of the open subspace $\phantom{}_{\leq \theta}\wt{\Bun}_{P}$ of our Drinfeld compactification. 
\begin{proposition}
There exists a $v$-stack $X$ with cohomologically smooth surjections onto $\tilde{Z}^{\theta}$ and $\phantom{}_{\leq \theta}\wt{\Bun}_{P}$ such that for $\theta'\leq \theta$ the pullbacks of the strata $\phantom{}_{\theta'} \tilde{Z}^{\theta}$ and $\phantom{}_{\theta'}\wt{\Bun}_{P}$ to $X$ agree.
\end{proposition}

Namely, we set
\[
X = \bigsqcup_{\theta'\geq \theta} (\tilde{Z}^{\theta}  \times Z^{\theta' - \theta}) \times_{\mathrm{Mod}_{M}^{+,\theta} \times \mathrm{Mod}_{M}^{+,\theta' - \theta}} T_{\theta,\theta' - \theta}.
\]

One issue with this proposition is that the diagram
\[\begin{tikzcd}
X\arrow[r]\arrow[d] & \tilde{Z}^\theta\arrow[d, "h_M^{\la}"]\\
\phantom{}_{\leq \theta}\wt{\Bun}_{P}\arrow[r, "\phantom{}_{\leq \theta} \tilde{\mathfrak q}"] & \Bun_M
\end{tikzcd}\]
does not commute, so the proposition is not suitable to study the geometry of the map $\phantom{}_{\leq \theta} \tilde{\mathfrak q}$. However, this can be remedied by working with $X$ directly (and remembering its correct map to $\Bun_M$).

For our applications, we need to endow $\tilde{Z}^\theta$ with a $\mathbb G_m$-action, which it only acquires after base-changing along the \'etale covering $\Div^{(\theta),\prime} \ra \Div^{(\theta)}$, as in Corollary \ref{cor: GmactiononZastava}.

Thus, we replace the $X$ defined above with its base-change along the \'etale map $\Div^{(\theta),\prime} \ra \Div^{(\theta)}$, and redefine
\[
X:= \bigsqcup_{\theta'\geq \theta} ((\tilde{Z}^\theta \times Z^{\theta' - \theta}) \times_{\mathrm{Mod}_{M}^{+,\theta} \times \mathrm{Mod}_{M}^{+,\theta' - \theta}} T_{\theta,\theta' - \theta}) \times_{\Div^{(\theta)}} \Div^{(\theta),\prime}.
\]
Moreover, set
\[
S:= \bigsqcup_{\theta'\geq \theta} ((\mathrm{Mod}_{M}^{+,\theta} \times Z^{\theta' - \theta}) \times_{\mathrm{Mod}_{M}^{+,\theta} \times \mathrm{Mod}_{M}^{+,\theta' - \theta}} T_{\theta,\theta' - \theta}) \times_{\Div^{(\theta)}} \Div^{(\theta),\prime}.
\]
Then if we write $\pi: X \ra S$ for the natural map, together with its section $i: S \ra X$ obtained as the base-change of $\tilde{\mf{s}}^{\theta}$, we have a commutative diagram of the form
\begin{equation}{\label{diag: keydiagULA}} \begin{tikzcd}
X \arrow[r,"p"] \arrow[d,"\pi"] & \phantom{}_{\leq \theta}\wt{\Bun}_{P} \arrow[d,"\phantom{}_{\leq \theta}\mf{q}"] & \\
S \arrow[r,"q"] \arrow[u,bend left = 50,"i"] & \Bun_{M}, & 
\end{tikzcd} 
\end{equation}
where specifically $q$ is the precomposition  of
\[(\mathrm{Mod}_{M}^{+,\theta} \times Z^{\theta' - \theta}) \times_{\mathrm{Mod}_{M}^{+,\theta} \times \mathrm{Mod}_{M}^{+,\theta' - \theta}} T_{\theta,\theta' - \theta} \ra T_{\theta,\theta' - \theta} \ra \Bun_{M}, \]
with the \'etale  covering induced by the base-change of  $\Div^{(\theta)}_{\mathcal{Y}} \ra \Div^{(\theta)}$.
Here the map $T_{\theta,\theta' - \theta} \ra \Bun_{M}$ is the bottom arrow in the diagram \eqref{diag: uniformizing Drinfeld}.

We consider the $\mathbb{G}_{m}$-action on $X$ which only acts on the first factor $\tilde{Z}^{\theta,\prime}$, and is trivial on the other factors. Taking the trivial $\mathbb G_m$-action on $S$, this makes $\pi$ and $i$ equivariant.

Summarizing the above discussion, we finally arrive at the following result, which is the key geometric result in our proof of the ULA property for $\tilde{j}_{P!}\IC_{\Bun_{P}}$.
\begin{corollary}{\label{cor: propertofkeydiag}}
For all $\theta$, there exists a diagram as in (\ref{diag: keydiagULA}) such that the following is true. 
\begin{enumerate}
\item The map $p$ is surjective and cohomologically smooth.
\item The map $i$ is a section of $\pi$ which maps $S$ isomorphically onto the closed substack given by the preimage along $p$ of the closed substack $j_{\theta}: \phantom{}_{\theta}\wt{\Bun}_{P} \hookrightarrow \phantom{}_{\leq \theta}\wt{\Bun}_{P}$. 
\item There exists a $\mathbb{G}_{m}$-action on $X$ lying over the trivial action on $S$ which is totally attracting towards the closed section $i$. The map $\pi$ is a partially proper $!$-able map of small v-stacks that is representable in locally spatial diamonds.
\item The $\mathbb{G}_{m}$-action on $X$ preserves the preimages under $p$ of the strata $\phantom{}_{\theta'}\wt{\Bun_P}$ for $\theta'\leq \theta$.
\end{enumerate}
\end{corollary}

We are now finally in a position to prove the ULA theorem. 

\section{The ULA theorem}

We will now show that $\wt{j}_{P!}(\IC_{\Bun_{P}})$ is ULA with respect to $\wt{\mf{q}}_{P}$. This will be accomplished by combining the geometry of the Zastava spaces considered in the previous section with some new results on ULA sheaves. Roughly speaking, the key idea here is that in the presence of a contracting $\mathbb{G}_m$-action, the ULA property automatically extends from the complement of the fixed locus. Our first task is to make this idea precise.

\subsection{ULA magic}

In the following, we make use of certain results from \cite{FarguesScholze} proved in the context of the $\ell$-adic $6$-functor formalism. However, these are general results valid for any good $6$-functor formalism, and the results extend verbatim to the motivic formalism used in this paper for general $\mathbb Z_\ell$-algebras; details will appear in \cite{MotivicGeometrization}.

\begin{lemma}[The contraction principle]{\label{lemma: hyperboliclocalizationtotallyattracting}}Let $\pi:X \to S$ be a partially proper $!$-able map of small v-stacks that is representable in locally spatial diamonds, and equipped with a section $i:S\to X$ and a $\mathbb{G}_m$-action which is totally attracting towards the section. If $A \in D(X,\Lambda)$ is $\mathbb{G}_m$-equivariant, then $\pi_! A \cong i^! A$.
\end{lemma}

\begin{proof} Special case of (the proof of) hyperbolic localization, more precisely of \cite[Proposition VI.6.6]{FarguesScholze} (cf.~also the proof of \cite[Theorem IV.6.5]{FarguesScholze}).
\end{proof}

\begin{theorem}{\label{thm: hyplocULA}}
Let $\pi:X \to S$ be a partially proper $!$-able map that is representable in locally spatial diamonds, and equipped with a section $i:S\to X$ and a $\mathbb{G}_m$-action which is totally attracting towards the section. Let $X^\circ = X\smallsetminus S$ with $j:X^\circ \to X$ the natural open immersion, and set $\pi^\circ = \pi \circ j$. If $A\in D(X^\circ,\Lambda)$ is $\mathbb{G}_m$-equivariant and $\pi^\circ$-ULA, then $j_!A$ is $\pi$-ULA.
\end{theorem}
As a formal consequence, we get that also $j_{\ast}A$ is $\pi$-ULA, and that $i^\ast j_\ast A \in D(S,\Lambda)$ is dualizable. We also deduce that if $B \in D(X,\Lambda)$ is any $\mathbb{G}_m$-equivariant sheaf, then it is $\pi$-ULA iff $j^\ast B$ is $\pi^\circ$-ULA and $i^\ast B$ is dualizable.

\begin{proof}Let $p_1,p_2: X \times_S X \to X$ be the evident projections. We need to see that the natural map
\[ p_{1}^\ast \mathbb{D}_{X/S}(j_! A) \otimes p_{2}^{\ast} j_!A \overset{\alpha}{\to} \sHom(p_1^\ast j_!A, p_2^! j_!A)\]
is an isomorphism. Over $X^\circ \times_S X$ this follows from the fact that $A \boxtimes 1$ is ULA for $X^\circ \times_S X \to X$. Similarly over $X \times_S X^\circ$ (argue as in the proof of \cite[Lemma 4.3]{HansenScholze}). So the cone of this map is supported on $(i \times i):S \hookrightarrow X \times_S X$. Thus it suffices to see the cone is killed by $(i \times i)^!$. Applying the previous lemma to $X\times_S X$ with the diagonal $\mathbb{G}_m$-action, we get that $(i \times i)^!=(\pi \times \pi)_!$ as functors applied to the source or target of $\alpha$. We claim that this functor actually kills both the source and target of $\alpha$, and therefore kills the cone. We first note that $(\pi \times \pi)_!$  kills the source of $\alpha$. For this,  use K\"unneth to write
\[(\pi \times \pi)_! ( p_{1}^\ast \mathbb{D}_{X/S}(j_! A) \otimes p_{2}^{\ast} j_!A ) \cong \pi_!\mathbb{D}_{X/S}(j_! A) \otimes \pi_!j_!A. \]
Now use the fact that 
\[\pi_! \mathbb{D}_{X/S}(j_! A) = \pi_! j_\ast \mathbb{D}_{X^\circ /S}(A) = i^! j_\ast \mathbb{D}_{X^\circ /S}(A)=0\] where the second equality follows from the previous lemma, and we used the general vanishing $i^! j_*=0$ in the final equality. To see that $(i \times i)^!$  kills the target of $\alpha$, use the general isomorphism
\[(i \times i)^! \sHom(p_1^\ast j_!A, p_2^! j_!A) \cong \sHom((i \times i)^\ast p_1^\ast j_!A, (i \times i)^!p_2^! j_!A)  \]
and then observe that $(i \times i)^\ast p_1^\ast j_! = i^* j_! =0$.
\end{proof}

For our purposes, we will actually need the following generalization of the previous result.

\begin{theorem}{\label{thm: hyplocalpropvariant}}
Let $\pi:X \to S$ be a partially proper $!$-able map of small v-stacks that is representable in locally spatial diamonds, and equipped with a section $i:S\to X$ and a $\mathbb{G}_m$-action which is totally attracting towards the section. Let $q: S \ra S'$ be any $!$-able map of small v-stacks and set $\pi' := q \circ \pi$. Let $X^\circ = X\smallsetminus S$ with $j:X^\circ \to X$ the natural open immersion, and set $\pi'^\circ = \pi' \circ j$. If $A\in D(X^\circ,\Lambda)$ is $\mathbb{G}_m$-equivariant and $\pi'^\circ$-ULA, then $j_!A$ is $\pi'$-ULA.
\end{theorem}
\begin{proof}Let $p_1,p_2: X \times_{S'} X \to X$ be the evident projections. We need to see that the natural map
\[ p_{1}^\ast \mathbb{D}_{X/S'}(j_! A) \otimes p_{2}^{\ast} j_!A \overset{\alpha}{\to} \sHom(p_1^\ast j_!A, p_2^! j_!A)\]
is an isomorphism. As in the previous proof, it is clear that this map is an isomorphism after pullback along the open immersion $X^\circ \times_{S'} X \to X \times_{S'} X$. It is then enough to see that the source and target of $\alpha$ are killed by $i'^!$, where $i'= i \times \mathrm{id} : S\times_{S'} X \to X \times_{S'} X$ is the evident base change of $i$. Write also $p_S: S\times_{S'} X \to S$ and $p_X: S \times_{S'} X \to X$ for the evident projections. For the target of $\alpha$ the vanishing is clear, because
\begin{align*}
i'^! \sHom(p_1^\ast j_!A, p_2^! j_!A) & \cong \sHom(i'^\ast p_1^\ast j_!A, i'^! p_2^! j_!A) \\
 & \cong \sHom(p_S^\ast i^\ast j_!A, i'^! p_2^! j_!A)\\
 & \cong 0
\end{align*}
using $i^\ast j_!A=0$.

To analyze the source of $\alpha$, we will again use the contraction principle. For this we consider the map $\pi_X : X \times_{S'} X \to S \times_{S'} X$ given as the evident base change of $\pi$, so $i'$ is a section of $\pi_X$. Then
\begin{align*}
i'^! (p_{1}^\ast \mathbb{D}_{X/S'}(j_! A) \otimes p_{2}^{\ast} j_!A) & \cong \pi_{X!}(p_{1}^\ast \mathbb{D}_{X/S'}(j_! A) \otimes p_{2}^{\ast} j_!A) \\
& \cong \pi_{X!}(p_1^\ast \mathbb D_{X/S'}(j_! A)\otimes \pi_X^\ast p_X^\ast j_! A)\\
& \cong \pi_{X!} p_1^\ast \mathbb D_{X/S'}(j_! A)\otimes p_X^\ast j_! A \\
& \cong p_S^\ast \pi_! \mathbb D_{X/S'}(j_! A)\otimes p_X^\ast j_! A\\
 & \cong p_{S}^{\ast}\pi_! j_\ast \mathbb{D}_{X^\circ / S'}(A) \otimes p_{X}^{\ast} j_! A\\
 & \cong p_{S}^{\ast}i^! j_\ast \mathbb{D}_{X^\circ / S'}(A) \otimes p_{X}^{\ast} j_! A\\
&\cong 0
\end{align*}
where the first isomorphism follows from the contraction principle, the second isomorphism from $p_2 = p_X\circ \pi_X$, the third isomorphism from the projection formula, the fourth isomorphism from proper base change, the fifth isomorphism from the general commutation of Verdier duality, the sixth isomorphism from another application of the contraction principle, and the final vanishing from $i^! j_\ast =0$.
\end{proof}

We also have the following in the previous situation, which is a variant of \cite[Proposition~IV.6.14]{FarguesScholze}.

\begin{lemma}{\label{lemma: hyplocalpreservesULA}}
In the situation of Theorem \ref{thm: hyplocalpropvariant}, let $A \in D(X,\Lambda)$ be any sheaf which is $\mathbb{G}_{m}$-equivariant and $\pi'$-ULA. Then $i^{\ast}A$ and $i^{!}A \in D(S,\Lambda)$ are $q'$-ULA.
\end{lemma}
\begin{proof}
For the first claim, it is equivalent to see that $i_{\ast}i^{\ast}A $ is $\pi'$-ULA. This follows from the distinguished triangle
\[ j_!j^{\ast} A \to A \to i_{\ast}i^{\ast} A\to, \]
using that $j_!j^{\ast} A$ is $\pi'$-ULA by the previous theorem. The second claim follows from the first claim by writing 
\begin{align*}
    i^{!}A & \cong i^! \mathbb{D}_{X/S'}^{2} A \\
    & \cong \mathbb{D}_{S/S'} i^{\ast} \mathbb{D}_{X/S'}A
\end{align*}
and invoking the preservation of ULA sheaves under relative Verdier duality, using Verdier biduality for ULA sheaves to justify the first line.
\end{proof}

For later use, we record some additional properties of ULA sheaves we will need.

\begin{lemma}\label{lem:pullback-tensor-ULA-star-pushforward} Consider a diagram $X\overset{j}{\to} \overline{X} \overset{\overline{q}}{\to}S$ where $j$ is an open immersion and $\overline{q}$ is a $!$-able map of small v-stacks. Set $q=\overline{q}\circ j$. Suppose we are given a sheaf $A\in D(X,\Lambda)$ such that $j_!\mathbb D_{X/S}(A)$ is $\overline{q}$-ULA (hence so is its Verdier dual $j_\ast A$). Then $\overline{q}^\ast(-)\otimes j_{\ast}A \cong j_{\ast}(q^{\ast}(-) \otimes A)$.
\end{lemma}
\begin{proof}Set $A'=\bD_{X/S}(A)$. Then, using \cite[Proposition IV.2.19]{FarguesScholze} twice,
\begin{align*}
\overline{q}^\ast(-)\otimes j_{\ast}A & \cong \overline{q}^\ast(-)\otimes \bD_{\overline{X}/S}(j_! A') \\
 & \cong \sHom(j_!A',\overline{q}^{!}(-))\\
 & \cong j_*\sHom(A',q^{!}(-))\\
 & \cong j_*(q^\ast(-) \otimes \bD_{X/S}(A'))\\
 & \cong j_*(q^\ast(-) \otimes A),
\end{align*}
which gives the result.
\end{proof}

\begin{lemma}\label{lem:functor-preserves-compacts}Suppose given a diagram $X_1 \overset{f_1}{\leftarrow} Y \overset{f_2}{\rightarrow} X_2$ and a sheaf $K \in D(Y,\Lambda)$ which is $f_2$-ULA and whose support is relatively proper over $X_1$. Then the functor $F(-)=f_{2!}(K \otimes f_{1}^{\ast}-)$ preserves compact objects.
\end{lemma}
By "relatively proper support" we mean that there is an isomorphism $K\simeq i_\ast L$ where $i:Y' \to Y$ is a closed immersion such that $f_1 \circ i$ is proper.
\begin{proof}
It is formal that $F$ has right adjoint given by $G(-) = f_{1\ast}\sHom(K,f_{2}^{!}-)$. Using the ULA hypothesis, this can be rewritten as $ f_{1\ast}(\bD_{Y/X_2}(K) \otimes f_{2}^{\ast}-)$. Using the relatively proper support hypothesis, this can be rewritten as $f_{1!}(\bD_{Y/X_2}(K) \otimes f_{2}^{\ast}-)$. Now it is clear that $G$ commutes with all direct sums, so its left adjoint necessarily preserves compact objects.
\end{proof}

\subsection{Proof of the ULA theorem}
Thanks to our labor in the previous sections, we can finally conclude our main geometric result.  
\begin{theorem}{\label{thm: ULAtheorem}}
The sheaf $\tilde{j}_{P!}(\IC_{\Bun_{P}})$ is ULA with respect to $\tilde{\mf{q}}_{P}$, and, for all $\theta \in \Lambda_{G,P}^{\mathrm{pos}}$ the sheaf $\tilde{j}_{\theta}^{!}\tilde{j}_{P!}(\IC_{\Bun_{P}})$ is ULA with respect to the composition 
\[ \tilde{j}_{\theta}: \phantom{}_{\theta}\wt{\Bun}_{P} \xrightarrow{\tilde{j}_{\theta}} \wt{\Bun}_{P} \xrightarrow{\tilde{\mf{q}}_{P}} \Bun_{M}. \]
Similarly, the sheaf $\tilde{j}_{P*}(\IC_{\Bun_{P}})$ is ULA with respect to $\tilde{\mf{q}}_{P}$, and, for all $\theta \in \Lambda_{G,P}^{\mathrm{pos}}$, the sheaf $\tilde{j}_{\theta}^{*}\tilde{j}_{P*}(\IC_{\Bun_{P}})$ is ULA with respect to $\tilde{\mf{q}}_{P} \circ \tilde{j}_{\theta}$.
\end{theorem}
\begin{proof}
The second claim follows from the first by  applying \cite[Corollary~IV.2.25]{FarguesScholze}. 

We have the presentation $\colim_{\theta \in \Lambda_{G,P}^{\mathrm{pos}}} \phantom{}_{\leq \theta}\wt{\Bun}_{P} = \wt{\Bun}_{P}$ as an increasing union of open substacks where the transition maps are dictated by the Bruhat ordering on $\Lambda_{G,P}^{\mathrm{pos}}$. Using this, we reduce to showing that, for all $\theta \in \Lambda_{G,P}^{\mathrm{pos}}$, the sheaf $\phantom{}_{\leq \theta}\tilde{j}_{!}(\IC_{\Bun_{P}})$ is ULA with respect to $\phantom{}_{\leq \theta}\wt{\mf{q}} = \wt{\mf{q}} \circ \phantom{}_{\leq \theta}\tilde{j}$, where $\phantom{}_{\leq \theta}\tilde{j}: \Bun_{P} \hookrightarrow \phantom{}_{\leq \theta}\wt{\Bun}_{P}$ is the natural open immersion and that $\tilde{j}_{\theta}^{!}\phantom{}_{\leq \theta}\tilde{j}_{!}(\IC_{\Bun_{P}})$ is ULA with respect to $\tilde{\mf{q}}_{P} \circ \tilde{j}_{\theta}$. We will do this by induction on the Bruhat ordering.

The base case ($\theta = 0$) follows from combining \cite[Proposition~IV.2.33]{FarguesScholze}, \cite[Corollary~IV.2.25]{FarguesScholze}, Proposition \ref{prop: qissmooth}, and Corollary \ref{cor: reldualcomploverBunM}. For the inductive step, we can assume that $\phantom{}_{\leq \theta}\tilde{j}_{!}(\IC_{\Bun_{P}})$ is ULA with over the open substack $\phantom{}_{< \theta}\wt{\Bun}_{P}$ given by the complement of $\phantom{}_{\theta}\wt{\Bun}_{P} \subset \phantom{}_{\leq \theta}\wt{\Bun}_{P}$, and we want to show it is ULA over all of $\phantom{}_{\leq \theta}\wt{\Bun}_{P}$. For this, we use the diagram constructed in Corollary \ref{cor: propertofkeydiag}. By Corollary \ref{cor: propertofkeydiag} (1) and \cite[Proposition~IV.2.13 (ii)]{FarguesScholze}, we deduce that it suffices to show that $p^{*}\phantom{}_{\leq \theta}\tilde{j}_{!}(\IC_{\Bun_{P}})$ is ULA with respect to $q \circ \pi$. By \cite[Proposition~IV.2.13 (ii)]{FarguesScholze}, our inductive hypothesis, and Corollary \ref{cor: propertofkeydiag} (2), we know that this is true away from the closed image of the section $i$, with notation as in \ref{cor: propertofkeydiag}. The desired claim now follows by combining Corollary \ref{cor: propertofkeydiag} (3)-(4) with Theorem \ref{thm: hyplocalpropvariant}, which we apply with $S' = \Bun_{M}$. Now the ULAness of $\tilde{j}_{\theta}^{!}\phantom{}_{\leq \theta}j_{!}(\IC_{\Bun_{P}})$ with respect to $\tilde{\mf{q}}_{P} \circ \tilde{j}_{\theta}$ follows by applying Lemma \ref{lemma: hyplocalpreservesULA}.
\end{proof}
\begin{remark}{\label{rem: ULAoverBun implies ULAoverpoint}}
We note that the constant sheaf $\Lambda$ on $\Bun_{M}$ is ULA over $\ast$. In particular, if we have sheaves like $j_{!}(\IC_{\Bun_{P}})$ and $\tilde{j}_{\theta}^{!}j_{!}(\IC_{\Bun_{P}})$ which are ULA over $\Bun_{M}$ then it follows that they are also ULA over the point by \cite[Proposition~IV.2.26]{FarguesScholze}.
\end{remark}
Before moving into applications, we will now remove our running assumption that the derived group of $G$ is simply connected, so that we may conclude the most general possible results for the geometric Eisenstein functors. 

\subsection{Fixing the center}{\label{sec: fixingthecenter}}
In Sections 5.3-7.2, we had a running assumption that the derived group of $G$ is simply connected. For the applications in the next section, we will want to remove this assumption. To do this, we consider a variant of the Drinfeld compactification that has the correct properties when the derived group of $G$ is not simply connected, partially modelled on the analysis in \cite[Section~7]{Schi}. 

From now on, we let $G$ denote a general quasi-split connected reductive group and let 
\[ 0 \ra Z \ra \tilde{G} \ra G \ra 0 \]
be a central extension, where $\tilde{G}$ has simply connected derived group. 

We now consider a parabolic subgroup $P \subset G$, and write $\tilde{P} \subset \tilde{G}$ for the pre-image under the previous map. We consider the induced map 
\[ \Bun_{\tilde{G}} \ra \Bun_{G}. \]
We note, by \cite[Lemma~III.2.10]{FarguesScholze} and the proceeding discussion that the Picard stack $\Bun_{Z}$ acts on $\Bun_{\tilde{G}}$ by tensoring, and the quotient stack is isomorphic to $\Bun_{G}$. Similarly, we can define an action of $\Bun_{Z}$ on $\wt{\Bun}_{\tilde{P}}$ by tensoring the triple $(\mathcal{F}_{G},\mathcal{F}_{M},\tilde{\kappa}_{P})$ by the $Z$-bundle. This allows us to define the following.
\begin{definition}
We define $\wt{\Bun}_{P}$ to be the quotient of $\wt{\Bun}_{\tilde{P}}$ under the action of $\Bun_{Z}$ described above.
\end{definition}

A priori, this depends on the choice of $\tilde{G}$, but the category of such choices is naturally cofiltered, with transition maps inducing isomorphisms.

This space sits in the following Cartesian diagrams.
\begin{proposition}{\label{prop: CartesianDiagramofDrinfeld}}
There are a set of Cartesian diagrams
\[
\begin{tikzcd} 
\Bun_{\tilde{P}} \arrow[r,"\tilde{j}_{\tilde{P}}"] \arrow[d] & \wt{\Bun}_{\tilde{P}} \arrow[r,"\tilde{\mf{p}}_{\tilde{P}}"] \arrow[d] & \Bun_{\tilde{G}} \arrow[d] & \\
\Bun_{P} \arrow[r,"\tilde{j}_{P}"]  & \wt{\Bun}_{P} \arrow[r,"\tilde{\mf{p}}_{P}"] & \Bun_{G} &.
\end{tikzcd}
\]
\end{proposition}
\begin{proof}
The fact that the right most diagram is Cartesian follows from the definition and the fact that $\tilde{\mf{p}}_{\tilde{P}}$ is equivariant for the action of $\Bun_{Z}$. Similarly, we can see that $\tilde{j}_{\tilde{P}}: \Bun_{\tilde{P}} \ra \wt{\Bun}_{\tilde{P}}$ is equivariant for the natural $\Bun_{Z}$ action on the source. It remains to identify $\Bun_{P}$ with the quotient under this action. This is true for the natural induced map on Levi factors $\Bun_{\tilde{L}} \ra \Bun_{L}$ by another application of \cite[Lemma~III.2.10]{FarguesScholze}. To see it for the parabolics, we can combine \cite[Lemma~III.2.10]{FarguesScholze} with \cite[Lemma~3.3]{HI} noting that the $L^{\mathrm{ad}} = \tilde{L}^{\mathrm{ad}}$ and that the natural maps $\tilde{L} \ra L \ra L^{\mathrm{ad}}$ are all central extensions.
\end{proof}
We can use this result to deduce the following properties for the the stack $\wt{\Bun}_{P}$. 
\begin{proposition}{\label{prop: keypropertiesofDrinfeldInGeneralCase}}
For $G$ a general quasi-split connected reductive group, the following is true. 
\begin{enumerate}
\item There is a naturally defined small Artin v-stack $\wtBun_P$ sitting in a commutative diagram
\[\xymatrix{
\Bun_{P}\ar[drr]^{\mathfrak{p}_{P}}\ar@{_(->}[dr]_{\tilde{j}_P}\ar[dd]_{\mathfrak{q}_{P}}\\
 & \widetilde{\mathrm{Bun}}_{P}\ar[dl]^{\tilde{\mathfrak{q}}_{P}}\ar[r]_{\tilde{\mathfrak{p}}_{P}} & \mathrm{Bun}_{G}\\
\text{} \mathrm{Bun}_{M} & 
}\]
where $\tilde{j}_P$ is an open immersion. The map $\tilde{\mathfrak{p}}_{P}$ is partially proper and representable in locally spatial diamonds, and of locally finite $\mathrm{dim.trg}$.

\item If $\Bun_{M}^{\alpha} \subset \Bun_M$ is any connected component with open-closed preimage $\wtBun_{P}^{\alpha}$, the induced map $\tilde{\mathfrak{p}}_{P}^{\alpha}: \wtBun_{P}^{\alpha} \to \Bun_G$ is proper. 
\item The stack $\wt{\Bun}_{P}$ admits a locally closed stratification by certain strata $\phantom{}_{\theta}\wt{\Bun}_{P}$, indexed by $\theta \in \Lambda_{G,P}^{\mathrm{pos}}$. Each stratum has a canonical fiber product decomposition
\[  \phantom{}_{\theta}\wt{\Bun}_{P} \cong \Bun_{P} \times_{\Bun_{M}} \Mod^{+,\theta}_{\Bun_{M}}, \]
where $\Mod^{+,\theta}_{\Bun_{M}}$ is defined exactly as in the case where $G$ has simply connected derived group.
\item The sheaves $\tilde{j}_{P!}(\IC_{\Bun_P})$ and $\tilde{j}_{P\ast}(\IC_{\Bun_P})$ are $\tilde{\mathfrak{q}}_{P}$-ULA.
\end{enumerate}
\end{proposition}
\begin{proof}
Since the map $\Bun_{\tilde{G}} \ra \Bun_{G}$ is a surjective map of $v$-stacks, using Proposition \ref{prop: CartesianDiagramofDrinfeld} we deduce that $\tilde{j}_{P}$ is an open immersion by \cite[Proposition~10.11 (2)]{Ecod} and the analogous claim for $\tilde{j}_{\tilde{P}}$. It is clear that the map $\tilde{\mf{p}}_{P}$ is partially proper, and it is representable in locally spatial diamonds, and of locally finite $\mathrm{dim.trg}$ by \cite[Proposition~13.4 (v)]{FarguesScholze}. It then follows that $\wt{\Bun}_{P}$ is Artin $v$-stack using that $\Bun_{G}$ is Artin and \cite[Proposition~IV.18 (iii)]{FarguesScholze}

If we fix an element $\tilde{\alpha} \in \pi_{1}(\tilde{M})_{\Gamma}$ mapping to $\pi_{1}(M)_{\Gamma}$ with images $\tilde{\alpha}_{G} \in \pi_{1}(\tilde{G})_{\Gamma}$ and $\alpha_{G} \in \pi_{1}(G)_{\Gamma}$, respectively, then Proposition \ref{prop: CartesianDiagramofDrinfeld} implies that we have a Cartesian diagram 
\[
\begin{tikzcd}
\wt{\Bun}_{\tilde{P}}^{\tilde{\alpha}_{M}} \arrow[r,"\tilde{\mf{p}}_{P}^{\tilde{\alpha}_{M}}"] \arrow[d] & \Bun_{\tilde{G}}^{\tilde{\alpha}_{G}} \arrow[d] & \\
\wt{\Bun}_{P}^{\alpha_{M}} \arrow[r,"\tilde{\mf{p}}_{P}^{\alpha_{M}}"] & \Bun_{G}^{\alpha_{G}}, & 
\end{tikzcd}
\]
where the vertical maps are torsors under the neutral component of $\Bun_{Z}$, and in particular are surjections of $v$-stacks. Since the bottom map is proper, the top map is also proper using \cite[Proposition~10.11 (o)]{Ecod}. This establishes (1)-(2). For (3), we note that the locally closed immersion
\[ \tilde{j}_{\theta,\tilde{P}}: \Mod^{+,\theta}_{\Bun_{\tilde{M}}} \times_{\Bun_{\tilde{M}}} \Bun_{\tilde{P}} \hookrightarrow \wt{\Bun}_{\tilde{P}} \]
is equivariant for the natural action of $\Bun_{Z}$ on source and target and thus descends to a map 
\[ \tilde{j}_{\theta,P}: \Mod^{+,\theta}_{\Bun_{M}} \times_{\Bun_{M}} \Bun_{P} \hookrightarrow \wt{\Bun}_{P} \]
which is also a locally closed immersion by using \cite[Proposition~10.11 (i)]{Ecod}. 

It remains to show the ULA claim. It suffices to establish the claim for $\tilde{j}_{P!}(\IC_{\Bun_{P}})$ by \cite[Corollary~IV.2.25]{FarguesScholze}. Since ULA can be checked after a smooth surjective pullback, by applying proper base-change to the left Cartesian diagram in Proposition \ref{prop: CartesianDiagramofDrinfeld}, we reduce to showing that $\tilde{j}_{\tilde{P}!}(\IC_{\Bun_{\tilde{P}}})$ is ULA over $\Bun_{M}$, where we know that it is ULA over $\Bun_{\tilde{M}}$ by Theorem \ref{thm: ULAtheorem}. Therefore, by \cite[Proposition~IV.2.26]{FarguesScholze} the claim reduces to the following lemma. 
\begin{lemma}
The constant sheaf $\Lambda$ on $\Bun_{\tilde{M}}$ is ULA with respect to the map $\Bun_{\tilde{M}} \ra \Bun_{M}$.
\end{lemma}
\begin{proof}
By choosing a $v$-cover of $\Bun_{M}$ trivializing the $\Bun_{Z}$-torsor and using \cite[Proposition~IV.2.25]{FarguesScholze}, we reduce to showing that the constant sheaf on $\Bun_{Z}$ is ULA over $\ast$, and this is clear.
\end{proof}
\end{proof}
\begin{remark}
One could of course perform the analogous constructions and deduce the analogous properties for the compactification $\ol{\Bun}_{P}$.
\end{remark}

\section{Applications}

\subsection{Finiteness theorems}

Let $G$ be any reductive group over $E$.

\begin{theorem} \label{thm:hardfiniteness}Fix a parabolic $P=MU\subset G$, and let $\Lambda$ be any $\mathbb{Z}_{\ell}[\sqrt{q}]$-algebra. Let $\alpha \in \pi_0(\Bun_M)$ be any element.
\begin{enumerate}
    \item The functors $\Eis_{P!}^{\alpha}$ and $\Eis_{P\ast}^{\alpha}$ preserve ULA objects.

    \item The functor $\CT_{P!}^{\alpha} \cong \CT_{P^- \ast}^{\alpha}$ preserves compact objects.
\end{enumerate}
\end{theorem}
\begin{proof}
By Theorem \ref{thm: ULAtheorem} and its extension in Proposition \ref{prop: keypropertiesofDrinfeldInGeneralCase} (5), the sheaves $\tilde{j}^{\alpha}_{P!} \mathrm{IC}_{\Bun_{P}^{\alpha}}$ and $\tilde{j}^{\alpha}_{P\ast} \mathrm{IC}_{\Bun_{P}^{\alpha}}$ are $\tilde{\mathfrak{q}}_{P}^{\alpha}$-ULA. For (1), we then note the isomorphisms
\[ \Eis_{P!}^{\alpha}(A) \cong \tilde{\mathfrak{p}}^{\alpha}_{P!} (\tilde{\mathfrak{q}}_{P}^{\alpha \ast}A \otimes \tilde{j}^{\alpha}_{P!} \mathrm{IC}_{\Bun_{P}^{\alpha}}) \]
and
\[ \Eis_{P\ast}^{\alpha}(A) \cong \tilde{\mathfrak{p}}^{\alpha}_{P!} (\tilde{\mathfrak{q}}_{P}^{\alpha \ast}A \otimes \tilde{j}^{\alpha}_{P\ast} \mathrm{IC}_{\Bun_{P}^{\alpha}}) \]
using Lemma \ref{lem:pullback-tensor-ULA-star-pushforward} and the ULA property of $\tilde{j}^{\alpha}_{P\ast} \mathrm{IC}_{\Bun_{P}^{\alpha}}$ for the latter. Then ULA-twisted pullback preserves the ULA condition, and $\tilde{\mathfrak{p}}^{\alpha}_{P!}=\tilde{\mathfrak{p}}^{\alpha}_{P\ast}$ preserves the ULA condition since $\tilde{\mathfrak{p}}_{P}^{\alpha}$ is proper by \ref{prop: keypropertiesofDrinfeldInGeneralCase} (2). This gives (1).

For (2), we apply Lemma \ref{lem:functor-preserves-compacts}, taking $f_1=\tilde{\mathfrak{p}}_{P}^{\alpha}$, $f_2=\tilde{\mathfrak{q}}_{P}^{\alpha}$, and $K=\tilde{j}^{\alpha}_{!} \mathrm{IC}_{\Bun_{P}^{\alpha}}$ in the notation of that lemma. The lemma then implies that the functor $F=\CT_{P!}^{\alpha}$ preserves compact objects, so $\CT_{P^- \ast}^{\alpha}$ does as well by second adjointness. 
\end{proof}

\begin{remark}By second adjointness and the basic definitions and adjunctions, $\Eis_{P!}$ preserves colimits and $\Eis_{P\ast}$ preserves limits while $\CT_{P \ast} \cong \CT_{P^- !}$ preserves all limits and colimits. After restricting to components, it is also true that $\Eis_{P \ast}^{\alpha}$ preserves colimits and $\Eis_{P !}^{\alpha}$ preserves limits. The first claim follows from the isomorphism
\[ \Eis_{P\ast}^{\alpha}(-) \cong \tilde{\mathfrak{p}}^{\alpha}_{P!} ( \tilde{j}^{\alpha}_{P\ast} \mathrm{IC}_{\Bun_{P}^{\alpha}} \otimes \tilde{\mathfrak{q}}_{P}^{\alpha \ast}- ) \]
established in the previous proof, noting that all functors on the right commute with colimits. The second claim follows from the isomorphism
\[ \Eis_{P!}^{\alpha}(-) \cong \tilde{\mathfrak{p}}^{\alpha}_{P\ast} R\mathscr{H}\mathrm{om}(\tilde{j}^{\alpha}_{P\ast} \mathrm{IC}_{\Bun_{P}^{\alpha}},\tilde{\mathfrak{q}}_{P}^{\alpha ! }-) \]
which follows by a similar argument, noting that all functors on the right commute with limits.

This also shows that the functor $\Eis_{P!}^{\alpha}(-)$ has a left adjoint which preserves compact objects, given explicitly by
\[ \CT_{P ?}^{\alpha} =\tilde{\mathfrak{q}}_{P!}^{\alpha} (\tilde{j}^{\alpha}_{P\ast} \mathrm{IC}_{\Bun_{P}^{\alpha}} \otimes \tilde{\mathfrak{p}}^{\alpha \ast}_{P}-).  \]
It is not hard to deduce from second adjointness that Bernstein-Zelevinsky duality exchanges $\CT_{P !}^{\alpha}$ and $\CT_{P ?}^{\alpha}$. However, we are not aware of any applications of this functor.
\end{remark}

\begin{theorem}\label{thm:gluingfiniteness}The functor $i_{b\sharp}$ preserves ULA objects, and the functor $i_{b}^!$ preserves compact objects.
\end{theorem}
\begin{proof} When $G$ is quasisplit, we can let $P=MU$ be the dynamic parabolic of the Newton cocharacter $\nu_b$. Then, by \ref{cor: dominantreductionConstantEisensteinDescription} (5) up to an irrelevant shift and twist, $i_{b\sharp}$ agrees with $\Eis_{P^{-}!} \circ i_{b!}^{M}$, and by  Theorem \ref{thm:maintheorem}.(1), this composition of functors preserves the ULA property. 

Similarly, by Corollary \ref{cor: dominantreductionConstantEisensteinDescription} (4) up to an irrelevant shift and twist, $i_{b}^!$ agrees with $i_{b}^{M,\ast} \CT_{P \ast}$, and by Theorem \ref{thm:maintheorem}.(2)-(3), this composition of functors preserves compactness.

Removing the quasisplitness assumption, we can first deduce the case when $G$ has connected center, and then the general case by z-extensions.
\end{proof}

\begin{theorem}If $j:U \to \Bun_G$ is the inclusion of a quasicompact open substack and $A\in D(\Bun_G,\Lambda)$ is compact, then $j_\ast A$ has compact stalks. In particular, if $A \in D(G_b(E),\Lambda)$ is compact, then $i_{b\ast}A$ has compact stalks.
\end{theorem}
\begin{proof} 
The second part follows from the first by factoring $i_{b}$ as the composition $\Bun_{G}^{b} \overset{h}{\to} \Bun_{G}^{\preceq b} \overset{j}{\to}$ and then writing $i_{b\ast} = j_{\ast} h_{\ast}$, using that $h_{\ast} = h_!$ preserves compactness. 

For the first claim, it is equivalent to prove that for any open immersion $j: U\subset V$ of quasicompact open substacks of $\Bun_G$, the functor $j_\ast$ preserves compact objects. By induction, we may assume that $U = V \setminus \{b\}$ for some $b\in B(G)$. We need to prove that $i_b^{\ast} j_\ast$ preserves compact objects. But like for any open-closed stratification, we have $i_{b}^\ast j_* = i_{b}^{!} j_![1]$, and we already know that $j_!$ and $i_{b}^{!}$ preserve compact objects, using the previous theorem for the latter.
\end{proof}

\subsection{The cuspidal-Eisenstein decomposition}
\begin{proof}[Proof of Theorem \ref{thm:cuspEisdecomposition}]
(1) is clear. Indeed, the functors $\CT_{P\ast}$ commute with all limits since they are right adjoints, and with all colimits since their left adjoints preserve compact objects. Moreover, by second adjointness they also commute with Verdier duality up to swapping $P$ for its opposite. Moreover, if $A$ is cuspidal then $i_{b}^{\ast}A=i_{b}^{!}A=0$ for all non-basic $b$, since the stalk functors can be obtained from suitable constant term functors, by Corollary \ref{cor: dominantreductionConstantEisensteinDescription} (3)-(4).

(2) is completely formal, since $D(\Bun_G,\Lambda)^{\Eis}$ is stable under all colimits and $D(\Bun_G,\Lambda)^{\mathrm{cusp}}$ is exactly the right orthogonal of $D(\Bun_G,\Lambda)^{\Eis}$.

(3) requires some work. The essential point is that for each basic $b$, \[D(\Bun_G,\overline{\mathbb{Q}_\ell})^{\mathrm{cusp}} \cap i_{b!} D(G_b(E),\overline{\mathbb{Q}_\ell})\] identifies with the support of a union of supercuspidal Bernstein components for $G_b(E)$, which by (1) is stable under Verdier duality. But then Bernstein-Zelevinsky duality and Verdier duality induce the same involution on the set of Bernstein components, so one deduces that $D(\Bun_G,\overline{\mathbb{Q}_\ell})^{\mathrm{cusp}}$ is generated under colimits by objects of $D(\Bun_G,\overline{\mathbb{Q}_\ell})^{\mathrm{cusp}} \cap D(\Bun_G,\overline{\mathbb{Q}_\ell})^\omega$, and that $D(\Bun_G,\overline{\mathbb{Q}_\ell})^{\mathrm{cusp}}\cap D(\Bun_G,\overline{\mathbb{Q}_\ell})^\omega$ is stable under Bernstein-Zelevinsky duality.

Now to conclude, take any $A\in D(\Bun_G,\overline{\mathbb{Q}_\ell})^{\mathrm{cusp}}\cap D(\Bun_G,\overline{\mathbb{Q}_\ell})^\omega$. Then for any $P=MU \subset G$ and any compact $B\in D(\Bun_M,\overline{\mathbb{Q}_\ell})$, we have
\begin{align*}
0 &= \mathrm{RHom}(\Eis_{P!}B,A) \\
& \simeq \mathrm{RHom}(\mathbb{D}_{\mathrm{BZ}}A,\mathbb{D}_{\mathrm{BZ}}\Eis_{P!}B) \\
& \simeq \mathrm{RHom}(\mathbb{D}_{\mathrm{BZ}}A,\Eis_{P^- !} \mathbb{D}_{\mathrm{BZ}}B) 
\end{align*}
using second adjointness and the preservation of compact objects under $!$-Eisenstein series. Varying $B$ and $P$, and using the involutivity of Bernstein-Zelevinsky duality on compact objects together with the previous step, we deduce that $D(\Bun_G,\overline{\mathbb{Q}_\ell})^{\mathrm{cusp}}$ is the \emph{left} orthogonal of $D(\Bun_G,\overline{\mathbb{Q}_\ell})^{\Eis}$. Since it is also the right orthogonal by definition, we get the desired orthogonal decomposition.
\end{proof}

\subsection{Continuity of gluing functors}

We end with an application to representation theory. To put this result into context, let $G$ be a connected reductive group over a nonarchimedean local field $E$. Fix an algebraically closed field $C$ of characteristic zero. Write $\Groth(G)$ for the Grothendieck group of finite-length admissible $C$-representations of $G(E)$. Any element $\phi \in C_c^\infty(G(E),C)$ defines a linear form $\Groth(G) \to C$ by sending an irreducible representation $\pi$ to $\mathrm{tr}(\phi |\pi)$ and then extending linearly. By definition, the linear maps $\Groth(G) \to C$ of this form are called \emph{trace forms}.

\begin{definition} A homomorphism $f:\Groth(G) \to \Groth(H)$ is \textit{continuous} if for all trace forms $h$ on $\Groth(H)$, $h\circ f$ is a trace form on $\Groth(G)$.
\end{definition}

This condition turns out to be satisfied by many natural operations appearing in representation theory.

\begin{example}The involution on $\Groth(G)$ induced by smooth duality is continuous.
\end{example}

\begin{example}If $P=MU \subset G$ is a parabolic, the maps $[i_P^G]:\Groth(M) \to \Groth(G)$ and $[r_{G}^P]:\Groth(G) \to \Groth(M)$ induced by parabolic induction and Jacquet module are continuous. For $[i_{P}^{G}]$ this follows from van Dijk's formula, which gives an explicit recipe for a map $\phi \in C_c^\infty(G(E)) \mapsto \phi^P \in C_c^\infty(M(E))$ such that $\mathrm{tr}(\phi | i_{P}^{G} \pi) = \mathrm{tr}(\phi^{P}|\pi)$. 

For $[r_{G}^{P}]$ the result follows by induction on the semisimple rank, by combining the trace Paley-Wiener theorem, the geometric lemma, and the result for parabolic induction.
\end{example} 

\begin{example}The Aubert-Zelevinsky involution on $\Groth(G)$ is continuous. This follows, for instance, by writing it as a virtual combination of compositions of parabolic inductions and Jacquet modules, and then using the previous example.
\end{example}

\begin{example} If $G=\mathrm{GL}_n$ and $G'$ is an inner form of $G$, Badulescu's "inverse Jacquet-Langlands map" $\mathrm{LJ} : \Groth(G) \to \Groth(G')$ is continuous.
\end{example}

\begin{example}If $(G,\mu,b)$ is any local shtuka datum, the map $\mathrm{Mant}_{G,\mu,b}:\Groth(G_b) \to \Groth(G)$ is continuous. This is Theorem 6.5.4 in \cite{HKW}.
\end{example}

As an application of our finiteness theorems, we prove the following result.

\begin{theorem}For any $b,b'\in B(G)$, the map $[i_{b'}^{\ast} i_{b \sharp}]: \Groth(G_b) \to \Groth(G_{b'})$ is continuous.    
\end{theorem}
Note that $i_{b'}^{\ast} i_{b \sharp}$ preserves finite length representations by Theorem \ref{thm:maintheorem}, so this statement is well-posed.
\begin{proof}It is equivalent, and notationally simpler, to prove this result for the renormalized variant $[i_{b'}^{\ast \mathrm{ren}} i_{b \sharp}^{\mathrm{ren}}]$. We closely follow the proof of \cite[Theorem 6.5.4]{HKW}. More precisely, choosing an element $\phi \in C_c^\infty(G_{b'}(E))$, we will show that the linear form on $\Groth(G_b)$ given by $\pi \mapsto \mathrm{tr}(\phi \mid i_{b'}^{\ast \mathrm{ren}} i_{b \sharp}^{\mathrm{ren}} \pi)$ satisfies the conditions of the trace Paley-Wiener theorem, applied to the group $G_b(E)$. We refer to loc. cit. for the statement of the trace Paley-Wiener theorem in the form we will use. For the remainder of the proof, we fix a pro-$p$ compact open subgroup $K \subset G_{b'}(E)$ such that $\phi$ is bi-$K$-invariant.

\textbf{Verification of Condition 1.}  If $\mathrm{tr}(\phi \mid i_{b'}^{\ast \mathrm{ren}} i_{b \sharp}^{\mathrm{ren}} \pi)$ is nonzero, then $(i_{b'}^{\ast \mathrm{ren}} i_{b \sharp}^{\mathrm{ren}} \pi)^K$ is nonzero. Therefore, it suffices to see that there is some open compact $J \subset G_b(E)$ such that $(i_{b'}^{\ast \mathrm{ren}} i_{b \sharp}^{\mathrm{ren}} \pi)^K\neq 0$ only if $\pi^{J}\neq 0$. For this, we compute
\begin{align*}
(i_{b'}^{\ast \mathrm{ren}} i_{b \sharp}^{\mathrm{ren}} \pi)^K & \cong R\Hom(i_{b' \sharp}^{\mathrm{ren}} \mathrm{ind}_{K}^{G_{b'}(E)}\overline{\mathbb{Q}_\ell},i_{b \sharp}^{\mathrm{ren}} \pi) \\
& \cong R\Hom(i_{b !}^{\mathrm{ren}} \mathbb{D}_{\mathrm{coh}}\pi,i_{b' !}^{\mathrm{ren}} \mathrm{ind}_{K}^{G_{b'}(E)}\overline{\mathbb{Q}_\ell})
\\
& \cong R\Hom( \mathbb{D}_{\mathrm{coh}}\pi,i_{b}^{! \mathrm{ren}}i_{b' !}^{\mathrm{ren}} \mathrm{ind}_{K}^{G_{b'}(E)}\overline{\mathbb{Q}_\ell})
\\
& \cong R\Hom( \mathbb{D}_{\mathrm{coh}}i_{b}^{! \mathrm{ren}}i_{b' !}^{\mathrm{ren}} \mathrm{ind}_{K}^{G_{b'}(E)}\overline{\mathbb{Q}_\ell},\pi)
\end{align*}
where the second line follows from Bernstein-Zelevinsky duality and the isomorphism $\mathbb{D}_{\mathrm{BZ}}i_{b \sharp}^{\mathrm{ren}} \cong \mathbb{D}_{\mathrm{coh}} i_{b !}^{\mathrm{ren}}$ (applied also for $b'$), and in passing to the fourth line we crucially used that $i_{b}^{! \mathrm{ren}}$ preserves compact objects, by Theorem \ref{thm:gluingfiniteness}. Now $\mathbb{D}_{\mathrm{coh}}i_{b}^{! \mathrm{ren}}i_{b' !}^{\mathrm{ren}} \mathrm{ind}_{K}^{G_{b'}(E)}\overline{\mathbb{Q}_\ell}$ is compact, so it is supported on finitely many Bernstein components for $G_b(E)$. This shows that the irreducible $\pi$'s with $(i_{b'}^{\ast \mathrm{ren}} i_{b \sharp}^{\mathrm{ren}} \pi)^K \neq 0$ are supported on finitely many Bernstein components for $G_b(E)$. The existence of the desired open compact $J$ is then immediate.

\textbf{Verification of Condition 2.} Fix $P=MU \subset G_b$ and $\sigma$ an irreducible smooth $M(E)$-representation. Let $X=\Spec R$ be the smooth affine algebraic variety over $\overline{\mathbb{Q}_\ell}$ parametrizing unramified characters of $M(E)$. Let $\boldsymbol{\psi}:M(E)\to R^\times$ be the universal character, and form $\Pi=i_{P}^{G_b}(\sigma \boldsymbol{\psi})$. This is an admissible smooth $R[G_b(E)]$-module interpolating the parabolic inductions $i_{P}^{G_b}(\sigma \psi)$ for varying unramified characters $\psi$. Since $\Pi$ is admissible, the pushforward $i_{b \sharp}^{\mathrm{ren}}\Pi \in D(\Bun_G,R)$ is ULA by Theorem \ref{thm:gluingfiniteness}, so then also $i_{b'}^{\ast \mathrm{ren}} i_{b \sharp}^{\mathrm{ren}}\Pi \in D(G_{b'}(E),R)$ is ULA, and in particular $(i_{b'}^{\ast \mathrm{ren}} i_{b \sharp}^{\mathrm{ren}}\Pi)^K$ is a perfect complex of $R$-modules. The element $\phi$ naturally defines an endomorphism of this perfect complex, and its trace $f\in R$ interpolates the individual traces $\mathrm{tr}(\phi \mid i_{b'}^{\ast \mathrm{ren}} i_{b \sharp}^{\mathrm{ren}} i_{P}^{G_b}(\sigma \psi))$ as $\psi$ varies. This shows that $\mathrm{tr}(\phi \mid i_{b'}^{\ast \mathrm{ren}} i_{b \sharp}^{\mathrm{ren}} i_{P}^{G_b}(\sigma \psi))$ is an algebraic function of $\psi$, as desired.
\end{proof}

\bibliographystyle{amsalpha}
\phantomsection

\end{document}